\documentclass[a4paper]{article}
\usepackage{amsmath, amssymb, amsthm, stackrel, graphicx, enumerate,etoolbox,endfloat}
\usepackage{wrapfig, tikz, spath3,pgf}
\usepackage{tikz-cd,units}
\usetikzlibrary{arrows,knots}
\newtheorem{thma}{Theorem}

\newtheorem{thm}{Theorem}
\newtheorem{cor}[thm]{Corollary}
\newtheorem{lem}[thm]{Lemma}
\newtheorem{prop}[thm]{Proposition}
\theoremstyle{definition}

\newtheorem{df}[thm]{Definition}
\newtheorem{rmk}[thm]{Remark}

\newtheorem{notation}[thm]{Notation}
\newtheorem{dfa}[thma]{Definition}
\newtheorem{convent}[thm]{Convention}

\makeatletter
\newcommand\xleftrightarrow[2][]{%
	\ext@arrow 9999{\longleftrightarrowfill@}{#1}{#2}}
\newcommand\longleftrightarrowfill@{%
	\arrowfill@\leftarrow\relbar\rightarrow}
\makeatother

\newcommand{\rar}{\rightarrow}

\newcommand{\C}{\mathbb{C}}
\newcommand{\cat}[1]{\mathcal{#1}}

\newcommand{\id}{\textnormal{id}}

\newcommand{\End}{\textnormal{End}}

\newcommand{\ev}{\textnormal{ev}}
\newcommand{\coev}{\tn{coev}}
\newcommand{\Vect}{\mathbf{Vect}}

\newcommand{\tn}[1]{\textnormal{#1}}

\newcommand{\svec}{\mathbf{sVect}}

\newcommand{\Rep}{\textnormal{Rep}}

\newcommand{\dcentcat}[1]{\cat{Z}(\cat{#1})}

\newcommand{\SA}{\cat{O}(\cat{A})}

\title{The Symmetric Tensor Product on the Drinfeld Centre of a Symmetric Fusion Category}
\author{Thomas A. Wasserman}

\begin{document}

\maketitle

\begin{abstract}
	We define a symmetric tensor product on the Drinfeld centre of a symmetric fusion category, in addition to its usual tensor product. We examine what this tensor product looks like under Tannaka duality, identifying the symmetric fusion category with the representation category of a finite (super)-group. Under this identification, the Drinfeld centre is the category of equivariant vector bundles over the finite group (underlying the super-group, in the super case). In the non-super case, we show that the symmetric tensor product corresponds to the fibrewise tensor product of these vector bundles. In the super case, we define for each super-group structure on the finite group a super-version of the fibrewise tensor product. We show that the symmetric tensor product on the Drinfeld centre of the representation category of the resulting finite super-groups corresponds to this super-version of the fibrewise tensor product on the category of equivariant vector bundles over the finite group.
\end{abstract}

\setcounter{tocdepth}{2}
\tableofcontents

\section{Introduction}
Let $(\cat{A},\otimes)$ be a symmetric ribbon fusion category over $\C$. It is well-known \cite{Muger2003a} that its Drinfeld centre $\dcentcat{A}$ is a modular tensor category, with tensor product $\otimes_c$. By Tannaka duality \cite{Deligne1990}, there is a finite group or supergroup $G$, such that $\cat{A}\cong\tn{Rep}(G)$. With this identification, we have another description \cite[Chapter 3.2]{Bakalov2001a} of the Drinfeld centre as the category $\Vect_G[G]$ of $G$-equivariant vector bundles on $G$, equipped with the convolution tensor product. This category carries an additional tensor structure given by fibrewise tensor product (or in the super-case, a modified fibrewise tensor product, introduced in this paper), and this tensor structure is symmetric. 

Our goal is to define a symmetric tensor product
$$
\otimes_s : \dcentcat{A}\boxtimes\dcentcat{A}\rar \dcentcat{A},
$$
that is a purely categorical version of the fibrewise tensor product. We avoid using Tannaka duality in defining $\otimes_s$. In particular, this categorical description will treat the super and non-super Tannakian cases on equal footing. In the super-Tannakian case, this will lead us to define a generalisation of the fibrewise tensor product to equivariant vector bundles over a super-group. 

We will show in a follow-up paper \cite{Wasserman2017a} that the symmetric tensor product $\otimes_s$ together with the usual tensor product $\otimes_c$ makes the Drinfeld centre into a bilax 2-fold tensor category. That is, there are morphisms between $(z \otimes_s z')\otimes_c(y\otimes_s y')$ and $(z \otimes_c y)\otimes_s(z' \otimes_c y')$ for all $z, z',y,y'\in \dcentcat{A}$ satisfying coherence conditions. This in turn will be used to develop the theory of $\dcentcat{A}$-crossed tensor categories \cite{Wasserman2017b}. These can be seen as a groupoid equivalence invariant version of so-called $G$-crossed braided tensor categories, and they play a central role in defining the reduced tensor product \cite{Wassermanc}. This reduced tensor product gives a way of tensoring two braided tensor categories containing $\cat{A}$ together to produce a braided tensor category containing $\cat{A}$. One of the threads through this series of papers is that the theory is set up independently of Tannaka duality at every stage. This ensures the reduced tensor product only depends on the braiding and tensor product in the braided tensor categories, and can be computed in terms of these.

Outside of this series of papers, in parallel work Tham \cite{Tham2019} developed the closely related notion of a reduced tensor product on the Drinfeld Centre of a braided fusion category. This reduced tensor product corresponds to the stacking of annuli in a four-dimensional Crane-Yetter Topological Quantum Field Theory. In the case where the braided fusion category is symmetric, this construction recovers the symmetric tensor product discussed here.

To define $z \otimes_s z'$ we take the subobject of $z\otimes_c z'$ associated to the idempotent given by
$$
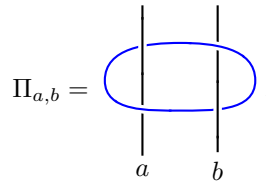
\Pi_{z,z'}=
\hbox{
	\begin{tikzpicture}[baseline=(current  bounding  box.center)]
	\coordinate (west) at (-1,0);
	\coordinate (north) at (0,0.5);
	\coordinate (east) at (1,0);
	\coordinate (south) at (0,-0.4);
	\node (a) at (-0.5,-1.2) {$z$};
	\node (b) at (0.5,-1.2) {$z'$};
	\coordinate (ao) at (-0.5,1);
	\coordinate (bo) at (0.5,1);
	
	\begin{knot}[clip width=4]
	\strand [blue, thick] (west)
	to [out=90,in=-180] (north)
	to [out=0,in=90] (east)
	to [out=-90,in=0] (south)
	to [out=-180,in=-90] (west);
	\strand [thick] (a) to (ao);
	\strand [thick] (b) to (bo);
	\flipcrossings{1,4}
	\end{knot}
	\end{tikzpicture}
}.
$$
The ring represents a sum over representatives for the isomorphism classes of simple objects of $\cat{A}\subset\dcentcat{A}$, the under- and over-crossings represent half-braidings in $\dcentcat{A}$. The alternating appearance of the crossings along the ring ensures the idempotent picks out the subobject of $z\otimes_c z'$ on which the half-braidings for $z\otimes_c z'$ obtained by using the symmetry in $\cat{A}$ and either the half-braiding of $z$ or the half-braiding of $z'$ agree. We then equip this subobject with either one of these the half-braidings, and define this to be the symmetric tensor product $z \otimes_s z'$ of $z$ and $z'$.

The outline of this paper is as follows. In Section \ref{STsetupsect} we recall the definition of the Drinfeld centre, and introduce some notation and useful lemmas about subobjects in idempotent complete categories and string diagrams. Then, in Section \ref{STsymtenssector}, we will define the symmetric tensor product on $\dcentcat{A}$. We will do this in two parts. First we will define the operation $\otimes_s$ on objects, and establish the associators, unit object and unitors, and symmetry objectwise. Secondly, we define $\otimes_s$ on morphisms and show that this definition makes $(\dcentcat{A},\otimes_s)$ into a symmetric monoidal category. This is our main result:

\begin{thma}[Theorem \ref{STsymtensmainthrm}]
	$(\dcentcat{A},\otimes_s,\mathbb{I}_s)$ is a symmetric monoidal category.
\end{thma}

In the final Section \ref{STtannacases}, we verify that, given a fibre functor on $\cat{A}$, the product $\otimes_s$ agrees with the fibrewise tensor product on $\Vect_G[G]$ in the Tannakian case: 

\begin{thma}[Theorem \ref{STsymprodtannafibthm}]
	Let $G$ be a finite group. Then the equivalence between $(\cat{Z}(\tn{Rep}(G)),\otimes_s)$ and $(\Vect_G[G],\otimes_f)$ is a symmetric monoidal equivalence. Here $\otimes_f$ denotes the fibrewise tensor product.
\end{thma}

In the super-Tannakian case, where $\cat{A}=\Rep(G,\omega)$, we first define a new tensor product on $\Vect_G[G]$ that depends on the choice of central element $\omega$. 

\begin{dfa}[Definition \ref{STfibrewisesupertensordef}]
	Let $(G,\omega)$ be a finite super-group. The \emph{fibrewise super-tensor product} of homogeneous vector bundles $V,W\in \Vect_{G}[G]$ is the $G$-equivariant vector bundle $V\otimes_{f}^{\omega} W$ with fibres
	$$
	(V\otimes_{f}^{\omega} W)_g=V_{\omega^{|W|}g}W_{\omega^{|V|}g},
	$$
	and $G$-action given by the tensor product of the $G$-actions.	
\end{dfa}

We then show that the symmetric tensor product on $\dcentcat{A}$ is taken to this tensor product on $\Vect_{G}[G]$ under the equivalence $\dcentcat{A}\cong \Vect_{G}[G]$:

\begin{thma}[Theorem \ref{STsymtenssupertannacase}]
	Let $(G,\omega)$ be a finite super-group. Then the equivalence between $(\cat{Z}(\Rep(G,\omega)),\otimes_s)$ and $(\Vect_{G}[G],\otimes_f^{\omega})$ is symmetric monoidal.
\end{thma}

\section{Preliminaries}\label{STsetupsect}
\subsection{Fusion categories}
\emph{Fusion categories} are idempotent complete semi-simple rigid tensor categories with a simple unit object and a finite set of isomorphism classes of simple objects. A \emph{tensor category} is a linear ($\Vect$-enriched abelian) category equipped with a bilinear monoidal structure, for the general theory see for example \cite{Calaque2004}. Idempotent completeness is discussed below (Section \ref{STidemnotation}), \emph{semi-simplicity} requires every short exact sequence to split, and by a \emph{simple object} we mean an object with no other subobjects than the zero object and itself. 
\begin{convent}
	In this article, we will restrict ourselves to fusion categories defined over $\C$. 
\end{convent}
For more details on fusion categories, see \cite{Etingof2002}. We will in particular be interested in \emph{braided} fusion categories, i.e. fusion categories equipped with a braiding for the monoidal structure. These braided fusion categories are extensively studied in \cite{Drinfeld2009}, while the general theory of braided monoidal categories can be found in \cite{Joyal1986}. To recall, a \emph{braiding} is a monoidal natural isomorphism $\beta$ with components $\beta_{c,d}: c\otimes d \xrightarrow{\cong} d\otimes c$ between the two different possible orderings of the objects in the monoidal structure. The monoidality of this natural isomorphism is captured by a family of equations known as the hexagon equations, see \cite{Joyal1986}. A braided monoidal category is called \emph{symmetric} if $\beta_{c,d}^{-1}=\beta_{d,c}$ for all objects $c$ and $d$.

\subsection{The Drinfeld Centre}
We recall the definition of the Drinfeld centre of a monoidal category (introduced in \cite{Majid1991,Joyal1991a}, see \cite[Definition 4.1]{Kassel1995} for more details) for convenience.
\begin{df}\label{STdrinfeldcenterdef}
	Let $\cat{A}$ be a monoidal category. The \emph{Drinfeld centre $\cat{Z}(\cat{A})$ of $\cat{A}$} is the braided monoidal category with objects pairs $(a,\beta)$, where $a$ is an object of $\cat{A}$ and $\beta$ is a natural isomorphism
	$$
	\beta: - \otimes a \Rightarrow a\otimes -,
	$$
	referred to as a \emph{half-braiding}. The $\beta$ are further required to satisfy 
	\begin{equation}\label{SThalfbraidreq}
	\beta_{bb'}= ( \beta_{b}\otimes\id_{b'})\circ( \id_{b}\otimes\beta_{b'}),
	\end{equation}
	for all $b,b'\in\cat{A}$, where we have suppressed the associators in $\cat{A}$. This condition is also sometimes called the hexagon equation.
	
	The morphisms in $\cat{Z}(\cat{A})$ are those morphisms in $\cat{A}$ that commute with the half-braidings in the obvious way. The tensor product is induced from the one on $\cat{A}$ and the braiding is the one specified by the half-braidings. 
	
	The Drinfeld centre comes with a \emph{forgetful functor} $\Phi:\dcentcat{A}\rar \cat{A}$, which forgets the half-braiding. This functor is monoidal.
\end{df}

It is well-known (\cite{Etingof2002}) that the centre of a fusion category is again fusion.

If $\cat{A}$ is braided, there is an obvious inclusion functor
\begin{equation}\label{STinclusionbrtodrinfeld}
\cat{A}\subset \dcentcat{A},
\end{equation}
which takes an object $a\in\cat{A}$ to $(a,\beta_{-,a})$, where $\beta_{-,a}$ denotes the natural isomorphism between $-\otimes a$ and $a\otimes -$ given by the braiding in $\cat{A}$.

\subsection{Notation}
We remind the reader that we use $\cat{A}$ to denote the symmetric tensor category, with tensor product $\otimes_\cat{A}$, on the Drinfeld centre $\dcentcat{A}$ of which we want to define a second tensor product. We will denote the usual tensor product on $\dcentcat{A}$ from Definition \ref{STdrinfeldcenterdef} by $\otimes_c$. Throughout, we will suppress the associators of $\cat{A}$ (and hence of $\dcentcat{A}$). We will suppress the symbols $\otimes_\cat{A}$ and $\otimes_c$ for the tensor product on $\cat{A}$ and $\dcentcat{A}$, respectively, when there is no risk of confusion.

\subsubsection{String Diagrams}
This article makes heavy use of string diagram calculus. String diagrams are a powerful tool to do computations in a braided monoidal category \cite{Joyal1991}, and have for example been used to give efficient proofs of theorems about fusion categories \cite{Bartlett2015}. If the reader is not familiar with this calculus, we can recommend \cite{Bartlett2009, Selinger2010}, on top of the aforementioned references. We give a very brief introduction here.

The central idea of string diagram calculus is that morphisms in a braided monoidal category can be represented by diagrams where we read composition bottom to top and tensoring left to right, while labels at the top and bottom of the diagram tell us which objects are involved. For example, with $f_1\colon a_1\rar b_1$, $f_2\colon a_2 \rar b_2$, and $g\colon b_1 \otimes b_2 \rar c$, we have the representations:
$$
g\circ(f_1\otimes f_2) =
\hbox{
	\begin{tikzpicture}[baseline=(current  bounding  box.center)]
	\node (a1) at (0,0){$a_1$};
	\node (a2) at (1, 0){$a_2$};
	
	\node (f1) at (0,1)[draw]{$f_1$};
	\node (f2) at (1,1)[draw]{$f_2$};
	
	\node (b1) at (-0.2,1.6){$b_1$};
	\node (b2) at (1.2, 1.6){$b_2$};
	
	\node (g) at (0.5,2.1)[draw]{$g$};
	
	\node (c) at (0.5,2.7){$c$};
	
	\begin{knot}[clip width=4]
		\strand [thick] (a1) to [out=90,in=-90] (f1);
		\strand [thick] (a2) to [out=90,in=-90] (f2);
		\strand [thick] (f1) to [out=90,in=-120] (g);
		\strand [thick] (f2) to [out=90,in=-60] (g);
		\strand [thick] (g) to (c);
	\end{knot}
\end{tikzpicture}
}=
\hbox{
	\begin{tikzpicture}[baseline=(current  bounding  box.center)]
	\node (a1) at (0,0){$a_1$};
	\node (a2) at (1, 0){$a_2$};

	\node (g) at (0.5,1.5)[draw]{$g\circ(f_1\otimes f_2)$};
	
	\node (c) at (0.5,2.7){$c$};
	
	\begin{knot}[clip width=4]
		\strand [thick] (a1) to [out=90,in=-120] (g);
		\strand [thick] (a2) to [out=90,in=-60] (g);
		\strand [thick] (g) to (c);
	\end{knot}
\end{tikzpicture}
}.
$$
Here the second equality illustrates one of the relations that we have between these string diagrams: if we can get from one to the other by composing the morphisms corresponding to the coupons, they represent the same morphism. In the braided setting, with the conventions we will introduce below in Section \ref{braidingconventionsection}, these diagrams can further be thought of as representing strings living in three dimensions. Two such diagrams represent the same morphism in the braided monoidal category if they represent the same string configuration, and this relation can be expressed in terms of Reidemeister moves, just like for knot diagrams.

The identity on the monoidal unit $\mathbb{I}$ is represented by the empty string diagram. For an object $a$ with a right dual $(a^*,\ev,\coev)$, we represent $\ev: a\otimes a^* \rar \mathbb{I}$ by an arc $\cap$ and $\coev:\mathbb{I}\rar a^*\otimes a$ by an oppositely oriented arc $\cup$. These satisfy the so-called snake identities:
\begin{equation}\label{STsnakes}
\hbox{
	\begin{tikzpicture}[baseline=(current  bounding  box.center)]

	\node (a) at (1, 0){$a$};

	\coordinate (ao) at ( 1,1);
	\begin{knot}[clip width=4]
	\strand [thick] (a) to [out=90,in=-90] (ao);
	\end{knot}
	\end{tikzpicture}
}
=
\hbox{
	\begin{tikzpicture}[baseline=(current  bounding  box.center)]
	\node (a) at (0, 0){$a$};
	
	\coordinate (lt) at (0.75,0.25);
	
	\coordinate (ut) at (0.25,0.75);
	\coordinate (ao) at ( 1,1);
	\begin{knot}[clip width=4]
	\strand [thick] (a) to [out=90,in=-180] (ut) to [out=0,in=180] (lt) to [out=0,in=-90](ao);
	\end{knot}
	\end{tikzpicture}
}
\mbox{ and }
\hbox{
	\begin{tikzpicture}[baseline=(current  bounding  box.center)]
	
	\node (a) at (1, 0){$a^*$};

	\coordinate (ao) at ( 1,1);
	\begin{knot}[clip width=4]
	\strand [thick] (a) to [out=90,in=-90] (ao);
	\end{knot}
	\end{tikzpicture}
}
=
\hbox{
	\begin{tikzpicture}[baseline=(current  bounding  box.center)]
	\node (a) at (0, 0){$a^*$};
	
	\coordinate (lt) at (-0.75,0.25);
	
	\coordinate (ut) at (-0.25,0.75);
	\coordinate (ao) at (-1,1);
	\begin{knot}[clip width=4]
	\strand [thick] (a) to [out=90,in=0] (ut) to [out=180,in=0] (lt) to [out=180,in=-90](ao);
	\end{knot}
	\end{tikzpicture}
},
\end{equation}
corresponding to obvious isotopies of the strings.

\subsubsection{Braiding Conventions}\label{braidingconventionsection}
When drawing string diagrams in $\dcentcat{A}$ we will use the convention that crossings correspond to braiding according to the half-braiding of the over-crossing object. That is, if $(a,\beta)\in\dcentcat{A}$, with $\beta: -\otimes a \Rightarrow a\otimes-$, and $c\in \dcentcat{A}$, we will denote:
$$
\beta_c=
\hbox{
	\begin{tikzpicture}[baseline=(current  bounding  box.center)]
	\node (c) at (0.5,0){$c$};
	\node (a) at (1, 0){$a$};

	\coordinate (co) at (1,1);
	\coordinate (ao) at ( 0.5,1);
	\begin{knot}[clip width=4]
	\strand [thick] (a) to [out=90,in=-90] (ao);
	\strand [thick] (c) to [out=90,in=-90] (co);
	\end{knot}
	\end{tikzpicture}
}
.
$$ 
Unresolved crossings will denote the use of the symmetry $s$ in $\cat{A}$. So for $(a,\beta),(a',\beta')\in\dcentcat{A}$,
$$
s_{a',a}=:
\hbox{
	\begin{tikzpicture}[baseline=(current  bounding  box.center)]
	\node (c) at (0.5,0){$a'$};
	\node (a) at (1, 0){$a$};

	\coordinate (co) at (1,1);
	\coordinate (ao) at ( 0.5,1);
	\begin{knot}[clip width=4]
	\strand [thick] (a) to [out=90,in=-90] (ao);
	\end{knot}
	\draw [thick] (c) to [out=90,in=-90] (co);
	\end{tikzpicture}
}.
$$
We will sometimes choose to resolve crossings between objects in $\cat{A}\subset \dcentcat{A}$ and objects in $\dcentcat{A}$, in order to make manipulations of the string diagrams easier to follow. Given $(a,s_{-,a})\in\cat{A}\subset \dcentcat{A}$ and $z\in\dcentcat{A}$, 
$$
s_{c,a}=:
\hbox{
	\begin{tikzpicture}[baseline=(current  bounding  box.center)]
	\node (c) at (0.5,0){$z$};
	\node (a) at (1, 0){$a$};

	\coordinate (co) at (1,1);
	\coordinate (ao) at ( 0.5,1);
	\begin{knot}[clip width=4]
	\strand [thick] (a) to [out=90,in=-90] (ao);
	\end{knot}
	\draw [thick] (c) to [out=90,in=-90] (co);
	\end{tikzpicture}
}
=
\hbox{
	\begin{tikzpicture}[baseline=(current  bounding  box.center)]
	\node (c) at (0.5,0){$z$};
	\node (a) at (1, 0){$a$};

	\coordinate (co) at (1,1);
	\coordinate (ao) at ( 0.5,1);
	\begin{knot}[clip width=4]
	\strand [thick] (a) to [out=90,in=-90] (ao);
	\strand [thick] (c) to [out=90,in=-90] (co);
	\end{knot}
	\end{tikzpicture}
}.
$$
In the case where also $z=(a',s_{-,a'})\in\cat{A}\subset\dcentcat{A}$, we have:
\begin{equation}\label{STtransparancy}
s_{a',a}=:
\hbox{
	\begin{tikzpicture}[baseline=(current  bounding  box.center)]
	\node (c) at (0.5,0){$z$};
	\node (a) at (1, 0){$a$};

	\coordinate (co) at (1,1);
	\coordinate (ao) at ( 0.5,1);
	\begin{knot}[clip width=4]
	\strand [thick] (a) to [out=90,in=-90] (ao);
	\end{knot}
	\draw [thick] (c) to [out=90,in=-90] (co);
	\end{tikzpicture}
}
=
\hbox{
	\begin{tikzpicture}[baseline=(current  bounding  box.center)]
	\node (c) at (0.5,0){$z$};
	\node (a) at (1, 0){$a$};

	\coordinate (co) at (1,1);
	\coordinate (ao) at ( 0.5,1);
	\begin{knot}[clip width=4]
	\strand [thick] (a) to [out=90,in=-90] (ao);
	\strand [thick] (c) to [out=90,in=-90] (co);
	\end{knot}
	\end{tikzpicture}
}
=
\hbox{
	\begin{tikzpicture}[baseline=(current  bounding  box.center)]
	\node (c) at (0.5,0){$z$};
	\node (a) at (1, 0){$a$};

	\coordinate (co) at (1,1);
	\coordinate (ao) at ( 0.5,1);
	\begin{knot}[clip width=4]
	\strand [thick] (a) to [out=90,in=-90] (ao);
	\strand [thick] (c) to [out=90,in=-90] (co);
	\flipcrossings{1}
	\end{knot}
	\end{tikzpicture}
},
\end{equation}
because in this case both half-braidings are given by the symmetry in $\cat{A}$. The following notion will be used throughout:
\begin{df}\label{STtransparentdef}
	Let $z,z'\in\cat{Z}$ be objects of a braided monoidal category. If 
	$$
	\hbox{
		\begin{tikzpicture}[baseline=(current  bounding  box.center)]
		\node (c) at (0.5,0){$z'$};
		\node (a) at (1, 0){$z$};

		\coordinate (co) at (1,1);
		\coordinate (ao) at ( 0.5,1);
		\begin{knot}[clip width=4]
		\strand [thick] (a) to [out=90,in=-90] (ao);
		\strand [thick] (c) to [out=90,in=-90] (co);
		\end{knot}
		\end{tikzpicture}
	}
	=
	\hbox{
		\begin{tikzpicture}[baseline=(current  bounding  box.center)]
		\node (c) at (0.5,0){$z'$};
		\node (a) at (1, 0){$z$};

		\coordinate (co) at (1,1);
		\coordinate (ao) at ( 0.5,1);
		\begin{knot}[clip width=4]
		\strand [thick] (a) to [out=90,in=-90] (ao);
		\strand [thick] (c) to [out=90,in=-90] (co);
		\flipcrossings{1}
		\end{knot}
		\end{tikzpicture}
	},
	$$
	then $z$ and $z'$ are said to be \emph{transparent} to each other. For a subcategory $\cat{Y}\subset\cat{Z}$, the \emph{centraliser $\cat{Z}_2(\cat{Y},\cat{Z})$ of $\cat{Y}$ in $\cat{Z}$} is the full subcategory on those objects of $\cat{Z}$ that are transparent to all objects of $\cat{Y}$.
\end{df}

Because of the naturality and monoidality of the symmetry, the resolved and unresolved crossings satisfy:
\begin{equation}\label{STcrossinginteraction}
\hbox{
	\begin{tikzpicture}[baseline=(current  bounding  box.center)]
	\coordinate (c) at (0.5,0);
	\coordinate (a) at (1, 0);
	
	\coordinate (b) at (0,0);
	\coordinate (bo) at (0.5,2);
	
	\coordinate (co) at (1,2);
	\coordinate (ao) at ( 0,2);
	\begin{knot}[clip width=4]
	\strand [thick] (a) to [out=90,in=-90] (ao);
	\strand [thick] (c) to [out=90,in=-90] (co);
	\flipcrossings{1}
	\end{knot}
	\draw [thick] (b) to [out=90,in=-90] (bo);
	\end{tikzpicture}
}=
\hbox{
	\begin{tikzpicture}[baseline=(current  bounding  box.center)]
	\coordinate (c) at (0.5,0);
	\coordinate (a) at (1, 0);
	
	\coordinate (b) at (0,0);
	\coordinate (bo) at (0.5,2);
	
	\coordinate (lc) at (0.1,1);
	\coordinate (rc) at (0.5,1);
	
	\coordinate (co) at (1,2);
	\coordinate (ao) at ( 0,2);
	\begin{knot}[clip width=4]
	\strand [thick] (a) to [out=90,in=-90] (ao);
	\strand [thick] (c) to [out=90,in=-90] (lc) to [out=90,in=-90] (co);
	\flipcrossings{1}
	\end{knot}
	\draw [thick] (b) to [out=90,in=-90] (rc) to [out=90,in=-90] (bo);
	\end{tikzpicture}
}.	
\end{equation}

\subsubsection{Quantum Dimensions and Global Dimension}
In the rest of this paper, we will denote a set of representatives of the isomorphism classes of simple objects of our symmetric category $\cat{A}$ by $\cat{O}(\cat{A})$. The quantum dimension $d_i$ of $i\in\cat{O}(\cat{A})$ is defined by: 
$$
d_i=
\hbox{
	\begin{tikzpicture}[baseline=(current  bounding  box.center)]
	\coordinate (west) at (-1,0);
	\coordinate (north) at (0,1);
	\coordinate (east) at (1,0);
	\coordinate (south) at (0,-1);
	\node (i) at (-0.8,0) {$i$};
	
	\begin{knot}[clip width=4]
	\strand [thick] (east)
	to [out=-90,in=0] (south)
	to [out=-180,in=-90] (west)
	to [out=90,in=-180] (north)
	to [out=0,in=90] (east);
	\end{knot}
	\end{tikzpicture}
},
$$
where we have used the pivotal structure $i= i^{**}$\footnote{A pivotal structure for a rigid monoidal category is a monoidal natural isomorphism between the identity functor and the double dual functor. An important open conjecture \cite{Etingof2002} is that every fusion category is pivotal. Rigid symmetric monoidal categories are pivotal with pivotal structure given by the identity natural transformation. It is routine to show that composing the evaluation and coevaluation morphisms for a choice of right dual $a^*$ of an object $a$ produces the evaluation and coevaluation morphisms that exhibit $a$ as $(a^*)^*$.} in on the right hand side of the loop. We will also make use of the following notation:
\begin{equation}\label{STbulletdef}
\theta_i=
\hbox{
	\begin{tikzpicture}[baseline=(current  bounding  box.center)]
	\node (i) at (0,0) {$i$};
	\node (idd) at (0,2) {$i$};
	
	\coordinate (dot) at (0,1);
	
	\begin{knot}[clip width=4]
	\strand [thick] (i)
	to [out=90,in=-90] (dot)
	to [out=90,in=-90] (idd);
	\end{knot}
	\fill (dot) circle[radius=2pt];
	\end{tikzpicture}
}
:=
\hbox{
	\begin{tikzpicture}[baseline=(current  bounding  box.center)]
	\node (i) at (0,0) {$i$};
	\node (idd) at (0,2) {$i$};
	
	\coordinate (dot) at (0.25,1.2);
	\coordinate (cc) at (0.25,0.8);
	
	\begin{knot}[clip width=4]
	\strand [thick] (i)
	to [out=90,in=180] (dot)
	to [out=0,in=0] (cc)
	to [out=180,in=-90] (idd);
	\end{knot}
	\end{tikzpicture}
}.
\end{equation}
This makes $\cat{A}$ into a ribbon category (see \cite{Henriques2016} for a careful exposition of braided pivotal categories and the ribbon condition for these). From this we read off that, because $\cat{A}$ is symmetric and $\theta_i^2=\id$, the twist will be $\pm\id$ on any simple object $i$ of $\cat{A}$. The global dimension of $\cat{A}$ will be denoted by
$$
D:=\sum_{i\in\cat{O}(\cat{A})}d_i^2.
$$
This global dimension will always be non-zero, as we are working with fusion categories over the complex numbers \cite[Theorem 2.3]{Etingof2002}.
 
We will use the additional notation
\begin{equation}\label{STringdef}
\hbox{
\begin{tikzpicture}[baseline=(current  bounding  box.center)]
\coordinate (west) at (-1,0);
\coordinate (north) at (0,1);
\coordinate (east) at (1,0);
\coordinate (south) at (0,-1);

\begin{knot}[clip width=4]
\strand [blue, thick] (west)
to [out=90,in=-180] (north)
to [out=0,in=90] (east)
to [out=-90,in=0] (south)
to [out=-180,in=-90] (west);
\end{knot}
\end{tikzpicture}
}
\mbox{ represents }\sum_{i\in\cat{O}(\cat{A})}\frac{d_i}{D}
\hbox{
\begin{tikzpicture}[baseline=(current  bounding  box.center)]
\coordinate (west) at (-1,0);
\coordinate (north) at (0,1);
\coordinate (east) at (1,0);
\coordinate (south) at (0,-1);

\node (i) at (-0.8,0) {$i$};
\begin{knot}[clip width=4]
\strand [thick] (east)
to [out=-90,in=0] (south)
to [out=-180,in=-90] (west)
to [out=90,in=-180] (north)
to [out=0,in=90] (east);
\end{knot}
\end{tikzpicture}
},
\end{equation}
whenever we encounter an unlabelled loop (possibly winding around other strands) in a string diagram.

\subsection{Direct Sum Decompositions}
In our proofs we will make frequent use of the following lemmas and notation. We will introduce them in the setting of a ribbon fusion category $\cat{A}$. This section
contains no new results (see \cite{Etingof2002,Bartlett2015}), we give proofs for convenience and later use.

\subsubsection{Dual Bases and Decompositions}
\begin{notation}
Given $i,j,k\in\cat{A}$, we will denote by $B(ij,k)$ a basis for the vector space $\cat{A}(ij,k)$.
\end{notation}

Since $\cat{A}$ is in particular semi-simple, we can, for fixed $i,j$, use this choice of basis $B(ij,k)$ for each $k\in\SA$, to give a direct sum decomposition of $ij$. In other words, we can give a decomposition of the identity on $ij$ as:
\begin{equation}\label{STresid}
\hbox{
\begin{tikzpicture}[baseline=(current  bounding  box.center)]

\node (a) at (-0.25,-2) {$i$};
\node (i) at (0.25,-2) {$j$};
\coordinate (ao) at (-0.25,1);
\coordinate (io) at (0.25,1);

\begin{knot}[clip width=4]

\strand [thick] (a) to [out=90,in=-90] (ao);
\strand [thick] (i) to (io);

\end{knot}
\end{tikzpicture}
}
= \sum_{k\in\cat{O}(\cat{A})}\sum_{\phi\in B(ij,k)}
\hbox{
\begin{tikzpicture}[baseline=(current  bounding  box.center)]

\node (a) at (-0.25,-2) {$i$};
\node (i) at (0.25,-2) {$j$};
\coordinate (ao) at (-0.25,1);
\coordinate (io) at (0.25,1);

\node (phi) at (0,-1) [draw,thick]{$\phi$};
\node (phid) at (0,0.25) [draw,thick]{$\phi^t$};

\node (k) at (-0.25,-0.5) {$k$};

\begin{knot}[clip width=4]
\strand [thick] (a) to [out=90,in=-110] (phi);
\strand [thick] (phi) to [out=90,in=-90] (phid);
\strand [thick] (phid) to [out =110,in=-90] (ao);
\strand [thick] (i) to  [out=90, in =-70] (phi);
\strand [thick] (phid) to [out=70,in=-90] (io);
\end{knot}
\end{tikzpicture}
}
.
\end{equation}

Here the $\phi^t$ are defined below in Definition \ref{STtransposedef}. The pairs $(\phi,\phi^t)$ for a given $k$ are (projection, inclusion)-pairs for subjects of $ij$ isomorphic to the simple object $k$. Choosing the $\phi$ from the basis $B(ij,k)$ ensures we exhaust all $k$-summands of $ij$ without linear dependence.
\begin{df}\label{STtransposedef}
	Let $\phi\in B(ij,k)$ be an element in a basis for $\cat{A}(ij,k)$, for simple objects $i,j,k$. Then a \emph{transpose} of $\phi$ is the morphism $\phi^t$ in a dual basis for $\cat{A}(k,ij)$, with respect to the pairing:
	$$
	\circ:\cat{A}(ij,k)\otimes\cat{A}(k,ij)\rar \cat{A}(k,k)=\C,
	$$ 
	such that $\phi\circ \phi^t=\id_k$ and $\psi\circ \phi^t=0$ for $\psi\in B(ij,k)-\{\phi\}$. As this pairing is non-degenerate (composing a morphism with an arbitrary morphism can only always be zero if the morphism is zero), such a dual basis, and hence transpose, always exist.
\end{df}

\subsubsection{Producing Decompositions from Decompositions}
Picking resolutions of the identities on $ij$ for a fixed $i\in \SA$ and all $j\in\SA$ induces a corresponding resolution of the identity on $k^*i$:
\begin{lem}
Pick, for a fixed $i\in\SA$ and all $j\in\SA$, a resolution of the identity on $ij$ as in Equation \eqref{STresid}. Then, for all $k\in\SA$:
\begin{equation}\label{SThelpfulid}
\hbox{
\begin{tikzpicture}[baseline=(current  bounding  box.center)]

\node (a) at (0.9,-2) {$i$};
\node (k) at (0.1,-2) {$k^*$};

\node (ao) at (0.9,1.5){$i$};
\node (ko) at (0.1,1.5){$k^*$};

\begin{knot}[clip width=4]
\strand [thick] (k) to [out=90,in=-90] (ko);
\strand [thick] (a) to [out=90,in=-90] (ao);
\end{knot}
\end{tikzpicture}
}
=
\sum_{j\in\cat{O}(\cat{A})}\sum_{\phi\in B(ij,k)}\frac{d_j}{d_k}
\hbox{
	\begin{tikzpicture}[baseline=(current  bounding  box.center)]
	
	\node (a) at (0.9,-2) {$i$};
	\node (k) at (0.1,-2) {$k^*$};
	
	\node (ao) at (0.9,1.5){$i$};
	\node (ko) at (0.1,1.5){$k^*$};

	\node (phi) at (1,-0.7) [draw,thick]{$\phi$};
	\node (phid) at (1,0.8) [draw,thick]{$\phi^t$};
	
	\coordinate (lc) at (1.5,-0.8);
	\coordinate (uc) at (1.5, 1);
	
	\node (i) at (1.75,0.5) {$j^*$};
	
	\begin{knot}[clip width=4]
	\strand [thick] (ko) to [out=-90,in=-90] (phid);
	\strand [thick] (phid) to [out=70,in=90] (uc) to [out=-90,in=90] (lc) to [out=-90,in=-70] (phi);
	\strand [thick] (phi) to [out=90,in=90] (k);
	\strand [thick] (a) to [out=90,in=-110] (phi);
	\strand [thick] (phid) to [out=110,in=-90] (ao);
	\end{knot}
	\end{tikzpicture}
}.
\end{equation}
\end{lem}

\begin{proof}
We claim that we can give a direct sum decomposition of $k^*i$, by using for each $j\in \SA$ and $\phi\in B(ij,k)$:
\begin{equation}\label{STnewinclproj}
\hbox{
\begin{tikzpicture}[baseline=(current  bounding  box.center)]

\node (a) at (0.9,-2) {$i$};
\node (k) at (0.1,-2) {$k^*$};

\node (phi) at (1,-0.7) [draw,thick]{$\phi$};

\coordinate (lc) at (1.5,-0.8);
\coordinate (uc) at (1.5, 0);

\node (i) at (1,1) {$j^*$};

\begin{knot}[clip width=4]
\strand [thick] (i) to [out=-90,in=90] (uc) to [out=-90,in=90] (lc) to [out=-90,in=-70] (phi);
\strand [thick] (phi) to [out=90,in=90] (k);
\strand [thick] (a) to [out=90,in=-110] (phi);
\end{knot}
\end{tikzpicture}
}
\tn{ and }\frac{d_j}{d_k}
\hbox{
\begin{tikzpicture}[baseline=(current  bounding  box.center)]

\coordinate (kctr) at (0.1,1);

\node (ao) at (0.9,2){$i$};
\node (ko) at (0.1,2){$k^*$};

\node (phid) at (1,0.8) [draw,thick]{$\phi^t$};

\coordinate (lc) at (1.5,0.5);
\coordinate (uc) at (1.5, 1);

\node (i) at (1,-0.5) {$j^*$};

\begin{knot}[clip width=4]
\strand [thick] (ko) to [out=-90,in=90] (kctr) to [out=-90,in=-90] (phid);
\strand [thick] (phid) to [out=70,in=90] (uc) to [out=-90,in=90] (lc) to [out=-90,in=90] (i);
\strand [thick] (phid) to [out=110,in=-90] (ao);
\end{knot}
\end{tikzpicture}
}
\end{equation}
as projection to and inclusion of $j^*$, respectively. That is, we want to see that for each $j$ (and hence $j^*$) the morphisms from Equation \eqref{STnewinclproj} form, letting $\phi$ range through $B(ij,k)$, a basis for $\cat{A}(k^*i,j^*)$ and a corresponding dual basis for $\cat{A}(j^*,k^*i)$. To see this, we first check that composing a $\phi'$ and a $\phi^t$ along $k^*i$ indeed gives the identity on $j^*$ if and only if $\phi=\phi'$, and zero otherwise:
$$
\frac{d_j}{d_k}
\hbox{
	\begin{tikzpicture}[baseline=(current  bounding  box.center)]
	
	\node (ilabel) at (0.8,-2){$j^*$};
	\coordinate (i) at (1,-2);
	\coordinate (io) at (1,2);
	
	\coordinate (lc) at (1.5,-0.6);
	\coordinate (uc) at (1.5,0.6);
	
	\node (phid) at (1,-0.8) [draw,thick]{$\phi^t$};
	\node (phi) at (1,0.8) [draw,thick]{$\phi'$};
	
	\coordinate (kcu) at (0.5,0.9);
	\coordinate (kcl) at (0.5,-0.9);
	
	\coordinate (ic) at (2,0);
	
	\begin{knot}[clip width=4]
	\strand [thick] (i) to [out=90,in=-90] (lc) to [out=90,in=70] (phid);
	\strand [thick] (phid) to [out=-90,in=-90] (kcl) to [out=90,in=-90] (kcu) to [out=90,in=90] (phi);
	\strand [thick] (phid) to [out=110,in=-110] (phi);
	\strand [thick] (phi) to [out=-70,in=-90] (uc) to [out=90,in=-90] (io) to [out=90,in=90] (ic) to [out=-90,in=-90] (i);
	\end{knot}
	\end{tikzpicture}
}
=
\frac{d_j}{d_k}
\hbox{
\begin{tikzpicture}[baseline=(current  bounding  box.center)]

\node (phid) at (1,-0.8) [draw,thick]{$\phi^t$};
\node (phi) at (1,0.8) [draw,thick]{$\phi'$};

\coordinate (kcu) at (0.5,0.9);
\coordinate (kcl) at (0.5,-0.9);

\begin{knot}[clip width=4]
\strand [thick] (phi) to [out=-70,in=70] (phid);
\strand [thick] (phid) to [out=-90,in=-90] (kcl) to [out=90,in=-90] (kcu) to [out=90,in=90] (phi);
\strand [thick] (phid) to [out=110,in=-110] (phi);
\end{knot}
\end{tikzpicture}
}
=\frac{d_j}{d_k}\delta_{\phi,\phi'}
\hbox{
\begin{tikzpicture}[baseline=(current  bounding  box.center)]

\coordinate (kcu) at (0.5,0.9);
\coordinate (kcl) at (0.5,-0.9);

\coordinate (ic) at (0.7,0);

\node (k) at (0.9,0){$k$};

\begin{knot}[clip width=4]
\strand [thick] (ic) to [out=-90,in=-90] (kcl) to [out=90,in=-90] (kcu) to [out=90,in=90] (ic);
\end{knot}
\end{tikzpicture}
}
=
\delta_{\phi,\phi'}
d_j,
$$
where the first identity uses the snake identities from Equation \eqref{STsnakes}, and in the second equality we used that composing $\phi'$ and $\phi^t$ along $ij$ gives the identity on $k$ if $\phi=\phi'$ and zero otherwise, by Definition \ref{STtransposedef}. As this is the trace of an endomorphism of $j^*$, and $j^*$ is simple, this shows that $\phi'$ and $\phi^t$ compose to the identity on $j^*$ if and only if $\phi=\phi'$, and to zero otherwise. 

To finish the argument, we just need show that, letting $\phi$ run through $B(ij,k)$, the first morphisms from Equation \eqref{STnewinclproj} form a basis for $\cat{A}(k^*i,j^*)$. But $\cat{A}(k^*i,j^*)\cong \cat{A}(ij,k)$, along the map that takes any $\phi\in \cat{A}(ij,k)$ to the first morphism in Equation \eqref{STnewinclproj}. This shows the morphisms from Equation \eqref{STnewinclproj} indeed form a complete and linearly independent set of (projection, inclusion)-pairs for each $j^*$. As $j^*$ indexes through all isomorphism classes of simple objects in $\cat{A}$, this gives a direct sum decomposition of $k^*i$.
\end{proof}

Similarly, we have:
\begin{lem}\label{STdirectsumversion}
	Pick, for fixed $j$ and all $i$ in $\SA$ a resolution of the identity as in Equation \eqref{STresid}. Then:
$$	
\hbox{
	\begin{tikzpicture}[baseline=(current  bounding  box.center)]
	
	\node (a) at (0.9,-2) {$j^*$};
	\node (k) at (0.1,-2) {$k$};
	
	\node (ao) at (0.9,1.5){$j^*$};
	\node (ko) at (0.1,1.5){$k$};

	\begin{knot}[clip width=4]
	\strand [thick] (k) to [out=90,in=-90] (ko);
	\strand [thick] (a) to [out=90,in=-90] (ao);
	\end{knot}
	\end{tikzpicture}
}
=
\sum_{i\in\cat{O}(\cat{A})}\sum_{\phi\in B(ij,k)}\frac{d_i}{d_k}
\hbox{
	\begin{tikzpicture}[baseline=(current  bounding  box.center)]
	
	\node (k) at (0,-2) {$k$};
	\node (jd) at (0.5,-2) {$j^*$};
	
	\node (ko) at (0,1.5){$k$};
	\node (jdo) at (0.5,1.5){$j^*$};

	\node (phid) at (0,-0.7) [draw,thick]{$\phi^t$};
	\node (phi) at (0,0.8) [draw,thick]{$\phi$};
	
	\coordinate (lc) at (0.5,-0.6);
	\coordinate (uc) at (0.5, 0.6);

	\begin{knot}[clip width=4]
	\strand [thick] (ko) to [out=-90,in=90] (phi);
	\strand [thick] (phid) to [out=70,in=90] (lc) to [out=-90,in=90] (jd);
	\strand [thick] (phid) to [out=-90,in=90] (k);
	\strand [thick] (jdo) to [out=-90,in=90] (uc) to [out=-90,in=-70] (phi);
	\strand [thick] (phid) to [out=110,in=-110] (phi);
	\end{knot}
	\end{tikzpicture}
}.
$$
\end{lem}
\begin{proof}
	The proof is analogous to the proof of the previous lemma.
\end{proof}

\subsection{Idempotents and Subobjects}\label{STidemnotation}
\subsubsection{Notation for Associated Subobjects}
Let $\cat{Z}$ again be an idempotent complete category (fusion categories are in particular idempotent). That is, for every $z\in\cat{Z}$ and $f\in \End(z)$ such that $f^2=f$ there exists $z_f\in \cat{Z}$, together with $i:z_f\hookrightarrow z$ and $p:z\twoheadrightarrow z_f$ satisfying $pi=\id_{z_f}$ and $ip=f$. Graphically, we will express this by using:
$$
i=
\hbox{
\begin{tikzpicture}[baseline=(current  bounding  box.center)]

\node (cf) at (0,-0.5) {$z_f$};

\node (tr) at (0.095,0.15) {$\bigtriangledown^f$};
\coordinate (trd) at (0,0);
\coordinate (tru) at (0,0.2);
\node (c) at (0,0.6) {$z$};

\begin{knot}[clip width=4]
\strand [thick] (cf) to (trd);
\strand [thick] (tru) to (c);
\end{knot}
\end{tikzpicture}
}
,\quad p= 
\hbox{
\begin{tikzpicture}[baseline=(current  bounding  box.center)]

\node (cf) at (0,-0.5) {$z$};

\node (tr) at (0.1,0.09) {$\bigtriangleup_f$};
\coordinate (trd) at (0,0.02);
\coordinate (tru) at (0,0.25);
\node (c) at (0,0.6) {$z_f$};

\begin{knot}[clip width=4]
\strand [thick] (cf) to (trd);
\strand [thick] (tru) to (c);
\end{knot}
\end{tikzpicture}
},
$$
with conditions
$$
\hbox{
\begin{tikzpicture}[baseline=(current  bounding  box.center)]

\node (cf) at (0,-0.5) {$z_f$};

\node (tr) at (0.095,0.1) {$\bigtriangledown_f$};
\coordinate (trd) at (0,0);
\coordinate (tru) at (0,0.2);

\node (tr1) at (0.1,0.69) {$\bigtriangleup_f$};
\coordinate (trd1) at (0,0.62);
\coordinate (tru1) at (0,0.85);
\node (c1) at (0,1.2) {$z_f$};

\begin{knot}[clip width=4]
\strand [thick] (cf) to (trd);
\strand [thick] (tru) to (trd1);
\strand [thick] (tru1) to (c1);
\end{knot}
\end{tikzpicture}
}
=
\hbox{
\begin{tikzpicture}[baseline=(current  bounding  box.center)]

\node (cf) at (0,-0.5) {$z_f$};
\node (c1) at (0,1.2) {$z_f$};

\begin{knot}[clip width=4]
\strand [thick] (cf) to (c1);
\end{knot}
\end{tikzpicture}
}
\tn{ and }
\hbox{
\begin{tikzpicture}[baseline=(current  bounding  box.center)]

\node (cf) at (0,-0.5) {$z$};

\node (tr) at (0.1,0.09) {$\bigtriangleup_f$};
\coordinate (trd) at (0,0.02);
\coordinate (tru) at (0,0.25);

\node (tr1) at (0.095,0.7) {$\bigtriangledown_f$};
\coordinate (trd1) at (0,0.6);
\coordinate (tru1) at (0,0.8);
\node (c1) at (0,1.2) {$z$};

\begin{knot}[clip width=4]
\strand [thick] (cf) to (trd);
\strand [thick] (tru) to (trd1);
\strand [thick] (tru1) to (c1);
\end{knot}
\end{tikzpicture}
}
=
\hbox{
\begin{tikzpicture}[baseline=(current  bounding  box.center)]

\node (cf) at (0,-0.5) {$z$};
\node (c1) at (0,1.2) {$z$};
\node (f) at (0,0.3) [draw]{$f$};

\begin{knot}[clip width=4]
\strand [thick] (cf) to (f) to (c1);
\end{knot}
\end{tikzpicture}
}.
$$

We will refer to the object $z_f$ as the subobject associated to $f$.

\subsubsection{Comparing Idempotents}
The following lemma will be useful later on:
\begin{lem}\label{STconjugateidems}
	Let $z,z'$ be objects in an idempotent complete category $\cat{Z}$, and let $f:z\rar z$ and $f':z'\rar z'$ be idempotents, denote their associated projections, inclusions and subobjects by $(p,i,z_f)$ and $(p',i',z_f')$, respectively. Suppose that $g:z\rar z'$ is an isomorphism such that $f'=gfg^{-1}$, then $p'gi:z_f\rar z_f'$ is an isomorphism.
\end{lem}
\begin{proof}
	We claim that the inverse of $p'gi$ is $pg^{-1}i'$. To see this, we compute:
	$$
	p'gipg^{-1}i'=p'gfg^{-1}i'=p'f'i'=p'i'p'i'=\id_{z_f'}.
	$$
	The other composite is similarly seen to be the identity.
\end{proof}

\section{The Symmetric Tensor Product}\label{STsymtenssector}
We will now proceed with constructing the symmetric tensor product on $\dcentcat{A}$.
To do this, we first define an idempotent we will use to pick out a
subobject of the usual tensor product on $\dcentcat{A}$. We then equip this subobject
with a convenient braiding, and define the result to be the symmetric tensor
product. We then show how to extend this definition to the morphisms of $\dcentcat{A}$
and check that the result satisfies the axioms of a symmetric monoidal structure
on $\dcentcat{A}$.

\subsection{A Useful Idempotent}

\subsubsection{Definition of the Idempotent}\label{STdefofidemp}
Recall that $\cat{A}$ is a symmetric ribbon fusion category. Let $z,z'\in\dcentcat{A}$ and suppose $z=(a,\beta)$ and $z'=(a',\beta')$. In defining the symmetric tensor product, we will use the following idempotent to pick out a subobject of $z\otimes_c z'$:
\begin{equation}\label{STPIzzdef}
\Pi_{z,z'}:=
\hbox{
\begin{tikzpicture}[baseline=(current  bounding  box.center)]
\coordinate (west) at (-1,0);
\coordinate (north) at (0,0.5);
\coordinate (east) at (1,0);
\coordinate (south) at (0,-0.4);
\node (a) at (-0.5,-1.2) {$z$};
\node (b) at (0.5,-1.2) {$z'$};
\coordinate (ao) at (-0.5,1);
\coordinate (bo) at (0.5,1);

\begin{knot}[clip width=4]
\strand [blue, thick] (west)
to [out=90,in=-180] (north)
to [out=0,in=90] (east)
to [out=-90,in=0] (south)
to [out=-180,in=-90] (west);
\strand [thick] (a) to (ao);
\strand [thick] (b) to (bo);
\flipcrossings{1,4}
\end{knot}
\end{tikzpicture}
}
=\sum_{i\in\cat{O}(\cat{A})}\frac{d_i}{D}
\hbox{
	\begin{tikzpicture}[baseline=(current  bounding  box.center)]
	\coordinate (west) at (-1,0);
	\coordinate (north) at (0,0.5);
	\coordinate (east) at (1,0);
	\coordinate (south) at (0,-0.4);
	\node (a) at (-0.5,-1.2) {$z$};
	\node (b) at (0.5,-1.2) {$z'$};
	\coordinate (ao) at (-0.5,1);
	\coordinate (bo) at (0.5,1);
	
	\node (i) at (-1.25,0){$i$};
	
	\begin{knot}[clip width=4]
	\strand [thick] (west)
	to [out=90,in=-180] (north)
	to [out=0,in=90] (east)
	to [out=-90,in=0] (south)
	to [out=-180,in=-90] (west);
	\strand [thick] (a) to (ao);
	\strand [thick] (b) to (bo);
	\flipcrossings{1,4}
	\end{knot}
	\end{tikzpicture}
},
\end{equation}

where the equality spells out the notation from Equation \eqref{STringdef}. It is worth emphasising here that the crossings where the $z$ or $z'$ strand passes over an $i$ strands represents $\beta_{i}$ or $\beta'_i$, respectively, whereas the crossings where the $z$ or $z'$ goes under an $i$ strand represent $s_{i,z}$ and $s_{i,z'}$, respectively.

\begin{lem}\label{STringswitch}
The morphism $\Pi_{z,z'}$ from Equation \eqref{STPIzzdef} satisfies
$$
\Pi_{z,z'}=
\hbox{
\begin{tikzpicture}[baseline=(current  bounding  box.center)]
\coordinate (west) at (-1,0);
\coordinate (north) at (0,0.5);
\coordinate (east) at (1,0);
\coordinate (south) at (0,-0.4);
\node (a) at (-0.5,-1.2) {$z$};
\node (b) at (0.5,-1.2) {$z'$};
\coordinate (ao) at (-0.5,1);
\coordinate (bo) at (0.5,1);

\begin{knot}[clip width=4]
\strand [blue, thick] (west)
to [out=90,in=-180] (north)
to [out=0,in=90] (east)
to [out=-90,in=0] (south)
to [out=-180,in=-90] (west);
\strand [thick] (a) to (ao);
\strand [thick] (b) to (bo);
\flipcrossings{1,4}
\end{knot}
\end{tikzpicture}
}=
\hbox{
\begin{tikzpicture}[baseline=(current  bounding  box.center)]
\coordinate (west) at (-1,0);
\coordinate (north) at (0,0.5);
\coordinate (east) at (1,0);
\coordinate (south) at (0,-0.4);
\node (a) at (-0.5,-1.2) {$z$};
\node (b) at (0.5,-1.2) {$z'$};
\coordinate (ao) at (-0.5,1);
\coordinate (bo) at (0.5,1);

\begin{knot}[clip width=4]
\strand [blue, thick] (west)
to [out=90,in=-180] (north)
to [out=0,in=90] (east)
to [out=-90,in=0] (south)
to [out=-180,in=-90] (west);
\strand [thick] (a) to (ao);
\strand [thick] (b) to (bo);
\flipcrossings{2,3}
\end{knot}
\end{tikzpicture}
},
$$

for all $z,z'\in\dcentcat{A}$.
\end{lem}

\begin{proof}
We compute:
$$
\hbox{
\begin{tikzpicture}[baseline=(current  bounding  box.center)]
\coordinate (west) at (-1,0);
\coordinate (north) at (0,0.5);
\coordinate (east) at (1,0);
\coordinate (south) at (0,-0.4);
\node (a) at (-0.5,-1.2) {$z$};
\node (b) at (0.5,-1.2) {$z'$};
\coordinate (ao) at (-0.5,1);
\coordinate (bo) at (0.5,1);

\begin{knot}[clip width=4]
\strand [blue, thick] (west)
to [out=90,in=-180] (north)
to [out=0,in=90] (east)
to [out=-90,in=0] (south)
to [out=-180,in=-90] (west);
\strand [thick] (a) to (ao);
\strand [thick] (b) to (bo);
\flipcrossings{1,4}
\end{knot}
\end{tikzpicture}
}
=
\hbox{
	\begin{tikzpicture}[baseline=(current  bounding  box.center)]
	\coordinate (west) at (-1.5,0);
	\coordinate (north) at (-0.2,0.5);
	\coordinate (east) at (1.5,0);
	\coordinate (south) at (0,-0.4);
	\node (a) at (-0.75,-1.6) {$z$};
	\node (b) at (0.75,-1.6) {$z'$};
	\coordinate (ao) at (-0.75,1);
	\coordinate (bo) at (0.75,1);
	
	\coordinate (es) at (0.4,-1.0);

	\begin{knot}[clip width=4,clip radius=3pt]
	\strand [blue, thick] (west)
	to [out=90,in=-180] (north)
	to [out=0,in=180] (es)
	to [out=0,in=90] (east);
	\strand[thick,blue] (east)
	to [out=-90,in=0] (south)
	to [out=-180,in=-90] (west);
	\strand [thick] (a) to (ao);
	\strand [thick] (b) to (bo);
	\flipcrossings{3,6}
	\end{knot}
	\end{tikzpicture}
}
=
\hbox{
\begin{tikzpicture}[baseline=(current  bounding  box.center)]
\coordinate (west) at (-1,0);
\coordinate (north) at (0,-0.5);
\coordinate (east) at (1,0);
\coordinate (south) at (0,0.4);
\node (a) at (-0.5,-1.2) {$z$};
\node (b) at (0.5,-1.2) {$z'$};
\coordinate (ao) at (-0.5,1);
\coordinate (bo) at (0.5,1);

\begin{knot}[clip width=4,consider self intersections,clip radius=3pt]
\strand [blue, thick] (west)
to [out=90,in=-180] (north)
to [out=0,in=90] (east)
to [out=-90,in=0] (south)
to [out=-180,in=-90] (west);
\strand [thick] (a) to (ao);
\strand [thick] (b) to (bo);
\flipcrossings{1,3}
\end{knot}
\end{tikzpicture}
}
=
\hbox{
\begin{tikzpicture}[baseline=(current  bounding  box.center)]
\coordinate (west) at (-1,0);
\coordinate (north) at (0,0.5);
\coordinate (east) at (1,0);
\coordinate (south) at (0,-0.4);
\node (a) at (-0.5,-1.2) {$z$};
\node (b) at (0.5,-1.2) {$z'$};
\coordinate (ao) at (-0.5,1);
\coordinate (bo) at (0.5,1);

\begin{knot}[clip width=4]
\strand [blue, thick] (west)
to [out=90,in=-180] (north)
to [out=0,in=90] (east)
to [out=-90,in=0] (south)
to [out=-180,in=-90] (west);
\strand [thick] (a) to (ao);
\strand [thick] (b) to (bo);
\flipcrossings{2,3}
\end{knot}
\end{tikzpicture}
}.
$$

The first step pulls down the top strand. The objects on the loop are objects in $\cat{A}\subset\dcentcat{A}$, so crossings of loop with itself correspond to symmetries in $\cat{A}$. So for these crossings we are in the situation of Equation \eqref{STtransparancy}, and can pass those strands through each other. Having done this, we can pull the top strand further down to get the second equality. Then we use that the self-crossings give rise to a twist (see Equation \eqref{STbulletdef}) and we can push these along the loop to meet. As the twist squares to $1$ in a symmetric ribbon fusion category, see the discussion below Equation \eqref{STbulletdef}, the last equality follows.
\end{proof}

We claim that $\Pi_{z,z'}$ is an idempotent, we will prove this below in Lemma \ref{STidempotent}. This will be a consequence of another property, Lemma \ref{STcloaking}, we examine first. 

\subsubsection{Cloaking}
The idempotent $\Pi_{z,z'}$ has a very useful property, a phenomenon called \emph{cloaking}. This lemma is a corollary of \cite[Lemma 7.1]{Bartlett2015a}\footnote{In the paper \cite{Bartlett2015a}, cloaking is phrased as taking place within a solid torus with an incoming and outgoing boundary component. To get from this result to the one here, imagine thickening the ring to a solid torus, giving the torus a boundary on each side, and passing the $a$ strand through it.}. We reprove it here for convenience of the reader.

\begin{lem}\label{STcloaking}
Let $z,z'\in\dcentcat{A}$ and $a\in \cat{A}$. Then the following identity holds:
$$
\hbox{
\begin{tikzpicture}[baseline=(current  bounding  box.center)]
\coordinate (west) at (-0.5,-0.3);
\coordinate (north) at (0,0);
\coordinate (east) at (0.5,-0.3);
\coordinate (south) at (0,-0.6);
\node (a) at (-0.75,-1.5) {$a$};
\node (b) at (-0.25,-1.5) {$z$};
\node (c) at (0.25,-1.5) {$z'$};
\coordinate (ao) at (0.75,1);
\coordinate (bo) at (-0.25,1);
\coordinate (co) at (0.25,1);

\coordinate (lc) at (-0.75,-0.2);

\begin{knot}[clip width=4]
\strand [blue, thick] (west)
to [out=90,in=-180] (north)
to [out=0,in=90] (east)
to [out=-90,in=0] (south)
to [out=-180,in=-90] (west);
\strand [thick] (a) to [out=90,in=-90] (lc) to [out=90,in=-90] (ao);
\strand [thick] (b) to (bo);
\strand [thick] (c) to (co);
\flipcrossings{5,1,4}
\end{knot}
\end{tikzpicture}
}=
\hbox{
\begin{tikzpicture}[baseline=(current  bounding  box.center)]
\coordinate (west) at (-0.5,-0.3);
\coordinate (north) at (0,0);
\coordinate (east) at (0.5,-0.3);
\coordinate (south) at (0,-0.6);
\node (a) at (-0.75,-2) {$a$};
\node (b) at (-0.25,-2) {$z$};
\node (c) at (0.25,-2) {$z'$};
\coordinate (ao) at (0.75,0.5);
\coordinate (bo) at (-0.25,0.5);
\coordinate (co) at (0.25,0.5);

\coordinate (rc) at (0.75,-0.4);

\begin{knot}[clip width=4]
\strand [blue, thick] (west)
to [out=90,in=-180] (north)
to [out=0,in=90] (east)
to [out=-90,in=0] (south)
to [out=-180,in=-90] (west);
\strand [thick] (a) to [out=90,in=-90] (rc) to [out=90,in=-90] (ao);
\strand [thick] (b) to (bo);
\strand [thick] (c) to (co);
\flipcrossings{6,1,4}
\end{knot}
\end{tikzpicture}
}.
$$
\end{lem}

\begin{proof}
For each summand $i$ of the loop, we decompose the identity on $ai$, like in Equation \eqref{STresid}. Inserting this resolution of the identity at the leftmost part of the loop, and pushing the morphisms along the loop to the other side, we obtain:
\begin{equation}\label{STstepcloak}
\hbox{
\begin{tikzpicture}[baseline=(current  bounding  box.center)]
\coordinate (west) at (-0.5,-0.3);
\coordinate (north) at (0,0);
\coordinate (east) at (0.5,-0.3);
\coordinate (south) at (0,-0.6);
\node (a) at (-1,-1.5) {$a$};
\node (b) at (-0.25,-1.5) {$z$};
\node (c) at (0.25,-1.5) {$z'$};
\coordinate (ao) at (1,1);
\coordinate (bo) at (-0.25,1);
\coordinate (co) at (0.25,1);

\coordinate (lc) at (-1,-0.2);

\begin{knot}[clip width=4]
\strand [blue, thick] (west)
to [out=90,in=-180] (north)
to [out=0,in=90] (east)
to [out=-90,in=0] (south)
to [out=-180,in=-90] (west);
\strand [thick] (a) to [out=90,in=-90] (lc) to [out=90,in=-90] (ao);
\strand [thick] (b) to (bo);
\strand [thick] (c) to (co);
\flipcrossings{5,1,4}
\end{knot}
\end{tikzpicture}
}=
\sum_{i,k\in\cat{O}(\cat{A})}\sum_{\phi\in B(ai,k)}\frac{t_i}{D}
\hbox{
\begin{tikzpicture}[baseline=(current  bounding  box.center)]
\coordinate (west) at (-0.75,0.5);
\coordinate (north) at (0,1);
\coordinate (south) at (0,0);
\node (a) at (-1,-2) {$a$};
\node (b) at (-0.25,-2) {$z$};
\node (c) at (0.25,-2) {$z'$};
\coordinate (ao) at (1,2.5);
\coordinate (bo) at (-0.25,2.5);
\coordinate (co) at (0.25,2.5);

\node (phi) at (1,-0.5) [draw,thick]{$\phi$};
\node (phid) at (1,1.5) [draw,thick]{$\phi^t$};

\coordinate (eu) (0.75,2.5);
\coordinate (eastd) (0.75,-0.5);

\coordinate (left) at (-1,-1.5);
\coordinate (right) at (0.8,-1);

\coordinate (lc) at (1.5,-0.6);
\coordinate (uc) at (1.5, 1.7);

\node (k) at (-1,0.5) {$k$};
\node (i) at (1.75,0.5) {$i^*$};

\begin{knot}[clip width=4]
\strand [thick] (west) to [out=90,in=-180] (north) to [out=0,in=-90] (phid);
\strand [thick] (phid) to [out=70,in=90] (uc) to [out=-90,in=90] (lc) to [out=-90,in=-70] (phi);
\strand [thick] (phi) to [out=90,in=0] (eastd) to [out=180,in=0] (south) to [out=-180,in=-90] (west);
\strand [thick] (a) to [out=90,in=90] (left) to [out=90,in=-110] (phi);
\strand [thick] (phid) to [out=110,in=-90] (ao);
\strand [thick] (b) to (bo);
\strand [thick] (c) to (co);
\flipcrossings{1,4,6}
\end{knot}
\end{tikzpicture}
}.
\end{equation}
using Equation \eqref{SThelpfulid} on the rightmost part of this diagram now proves the lemma.
\end{proof}

\subsubsection{Verifying Idempotency}
We still need to check $\Pi_{z,z'}$ is idempotent.
\begin{lem}\label{STidempotent}
The morphism $\Pi_{z,z'}$ in $\dcentcat{A}$ is an idempotent of $z\otimes_c z'$.
\end{lem}
\begin{proof}
We compute
$$
\hbox{
\begin{tikzpicture}[baseline=(current  bounding  box.center)]
\coordinate (west1) at (-1,1.1);
\coordinate (north1) at (0,1.5);
\coordinate (east1) at (1,1.1);
\coordinate (south1) at (0,0.65);

\coordinate (west) at (-1,-0.2);
\coordinate (north) at (0,0.1);
\coordinate (east) at (1,-0.2);
\coordinate (south) at (0,-0.55);

\node (a) at (-0.5,-1.4) {$z$};
\node (b) at (0.5,-1.4) {$z'$};

\coordinate (ao) at (-0.5,2);
\coordinate (bo) at (0.5,2);

\begin{knot}[clip width=4]
\strand [blue, thick] (west1)
to [out=90,in=-180] (north1)
to [out=0,in=90] (east1)
to [out=-90,in=0] (south1);
\strand [thick, blue] (south1)
to [out=-180,in=-90] (west1);
\strand [blue, thick] (west)
to [out=90,in=-180] (north)
to [out=0,in=90] (east)
to [out=-90,in=0] (south)
to [out=-180,in=-90] (west);

\strand [thick] (a) to (ao);
\strand [thick] (b) to (bo);
\flipcrossings{1,3,5,8}
\end{knot}
\end{tikzpicture}
}
=
\hbox{
\begin{tikzpicture}[baseline=(current  bounding  box.center)]
\coordinate (west1) at (-1,1.1);
\coordinate (north1) at (0,1.5);
\coordinate (east1) at (1,1.1);
\coordinate (south1) at (0,0.65);

\coordinate (west) at (-1,-0.2);
\coordinate (north) at (0,0.1);
\coordinate (east) at (1,-0.2);
\coordinate (south) at (0,-0.55);

\node (a) at (-0.5,-1.4) {$z$};
\node (b) at (0.5,-1.4) {$z'$};

\coordinate (ao) at (-0.5,2);
\coordinate (bo) at (0.5,2);

\begin{knot}[clip width=4]
\strand [blue, thick] (west1)
to [out=90,in=-180] (north1)
to [out=0,in=90] (east1)
to [out=-90,in=0] (south1);
\strand [thick, blue] (south1)
to [out=-180,in=-90] (west1);
\strand [blue, thick] (west)
to [out=90,in=-180] (north)
to [out=0,in=90] (east)
to [out=-90,in=0] (south)
to [out=-180,in=-90] (west);

\strand [thick] (a) to (ao);
\strand [thick] (b) to (bo);
\flipcrossings{2,4,5,8}
\end{knot}
\end{tikzpicture}
}
=
\hbox{
\begin{tikzpicture}[baseline=(current  bounding  box.center)]
\coordinate (west1) at (-1.5,-0.2);
\coordinate (north1) at (0,0.8);
\coordinate (east1) at (1.5,-0.2);
\coordinate (south1) at (0,-1);

\coordinate (west) at (-1,-0.2);
\coordinate (north) at (0,0.1);
\coordinate (east) at (1,-0.2);
\coordinate (south) at (0,-0.55);

\node (a) at (-0.5,-2.2) {$z$};
\node (b) at (0.5,-2.2) {$z'$};

\coordinate (ao) at (-0.5,1.5);
\coordinate (bo) at (0.5,1.5);

\begin{knot}[clip width=4]
\strand [blue, thick] (west1)
to [out=90,in=-180] (north1)
to [out=0,in=90] (east1)
to [out=-90,in=0] (south1);
\strand [thick, blue] (south1)
to [out=-180,in=-90] (west1);
\strand [blue, thick] (west)
to [out=90,in=-180] (north)
to [out=0,in=90] (east)
to [out=-90,in=0] (south)
to [out=-180,in=-90] (west);

\strand [thick] (a) to (ao);
\strand [thick] (b) to (bo);
\flipcrossings{1,4,5,8}
\end{knot}
\end{tikzpicture}
}
,
$$

where we used Lemma \ref{STringswitch} in the first step, to switch out under and over crossings on the top ring. The cloaking from Lemma \ref{STcloaking} allows us in the second to pass the lower part of the top ring across the lower ring. Now, we use that the loops are transparent (see Definition \ref{STtransparentdef}) to each other, as they are sums over objects of $\cat{A}\subset\dcentcat{A}$. This allows us to pull the larger loop out towards the right of the diagram, until it is completely separate from the rest of the diagram. This loop then evaluates to $1$, leaving us with the string diagram representation of $\Pi_{z,z'}$. This finishes the proof.
\end{proof}

\subsubsection{The Associated Subobject}
Given $z, z'\in\dcentcat{A}$, the idempotent $\Pi_{z,z'}$ from Lemma \ref{STidempotent} has an associated subobject that we will denote by $z\otimes_\Pi z'\in\dcentcat{A}$. Using the notation discussed in Section \ref{STidemnotation}, we introduce:
\begin{equation}\label{STincprojprop}
\hbox{
\begin{tikzpicture}[baseline=(current  bounding  box.center)]

\node (truw) at (-0.3,1.2){$z$};
\node (true) at (0.3,1.2){$z'$};

\node (tr1) at (0,2.2) {$\bigtriangleup$};
\coordinate (trd1w) at (-0.1,2.12);
\coordinate (trd1e) at (0.1,2.12);
\coordinate (tru1) at (0,2.35);
\node (c1) at (0,3) {$z\otimes_\Pi z'$};

\begin{knot}[clip width=4]
\strand [thick] (true) to [out=90, in=-70] (trd1e);
\strand [thick] (tru1) to (c1);
\strand [thick] (truw) to [out=90,in=-110] (trd1w);
\end{knot}
\end{tikzpicture}
}
\tn{ and }
\hbox{
\begin{tikzpicture}[baseline=(current  bounding  box.center)]

\node (cf) at (0,-0.5){$z\otimes_\Pi z'$};

\node (tr) at (0,0.1) {$\bigtriangledown$};
\coordinate (trd) at (0,0);
\coordinate (truw) at (-0.1,00.2);
\coordinate (true) at (0.1,0.2);

\node (trd1w) at (-0.3,1.12){$z$};
\node (trd1e) at (0.3,1.12){$z'$};

\begin{knot}[clip width=4]
\strand [thick] (cf) to (trd);
\strand [thick] (true) to [out=70, in=-90] (trd1e);
\strand [thick] (truw) to [out=110,in=-90] (trd1w);
\end{knot}
\end{tikzpicture}
},
\end{equation}
satisfying
\begin{equation}\label{STinclprojprops}
\hbox{
\begin{tikzpicture}[baseline=(current  bounding  box.center)]

\node (tr1) at (0,0.2) {$\bigtriangleup$};
\coordinate (trd1w2) at (-0.1,0.12);
\coordinate (trd1e2) at (0.1,0.12);
\coordinate (tru1) at (0,0.35);

\node (tr) at (0,1.6) {$\bigtriangledown$};
\coordinate (trd) at (0,1.5);
\coordinate (truw) at (-0.1,1.7);
\coordinate (true) at (0.1,1.7);

\node (trd1w) at (-0.3,2.2){$z$};
\node (trd1e) at (0.3,2.2){$z'$};

\node (a) at (-0.3,-0.5){$z$};
\node (b) at (0.3,-0.5){$z'$};

\begin{knot}[clip width=4]
\strand [thick] (tru1) to (trd);
\strand [thick] (true) to [out=70, in=-90] (trd1e);
\strand [thick] (truw) to [out=110,in=-90] (trd1w);
\strand [thick] (a) to [out=90,in=-110] (trd1w2);
\strand [thick] (b) to [out=90,in=-70] (trd1e2);
\end{knot}
\end{tikzpicture}
}
=
\hbox{
\begin{tikzpicture}[baseline=(current  bounding  box.center)]
\coordinate (west) at (-0.5,-0.3);
\coordinate (north) at (0,0);
\coordinate (east) at (0.5,-0.3);
\coordinate (south) at (0,-0.6);
\node (a) at (-0.25,-1.5) {$z$};
\node (b) at (0.25,-1.5) {$z'$};
\coordinate (ao) at (-0.25,1);
\coordinate (bo) at (0.25,1);

\begin{knot}[clip width=4]
\strand [blue, thick] (west)
to [out=90,in=-180] (north)
to [out=0,in=90] (east)
to [out=-90,in=0] (south)
to [out=-180,in=-90] (west);
\strand [thick] (a) to [out=90,in=-90] (ao);
\strand [thick] (b) to (bo);
\flipcrossings{1,4}
\end{knot}
\end{tikzpicture}
}
\tn{ and }
\hbox{
\begin{tikzpicture}[baseline=(current  bounding  box.center)]

\node (cf) at (0,-0.5){$z\otimes_\Pi z'$};

\node (tr) at (0,0.1) {$\bigtriangledown$};
\coordinate (trd) at (0,0);
\coordinate (truw) at (-0.1,00.2);
\coordinate (true) at (0.1,0.2);

\node (tr1) at (0,2.2) {$\bigtriangleup$};
\coordinate (trd1w) at (-0.1,2.12);
\coordinate (trd1e) at (0.1,2.12);
\coordinate (tru1) at (0,2.35);
\node (c1) at (0,3) {};

\begin{knot}[clip width=4]
\strand [thick] (cf) to (trd);
\strand [thick] (true) to [out=70, in=-70] (trd1e);
\strand [thick] (tru1) to (c1);
\strand [thick] (truw) to [out=110,in=-110] (trd1w);
\end{knot}
\end{tikzpicture}
}
=
\hbox{
\begin{tikzpicture}[baseline=(current  bounding  box.center)]

\node (inc) at (-0.5,-0.4){$z\otimes_\Pi z'$};

\coordinate (outg) at (-0.5,2.8);

\begin{knot}[clip width=4]
\strand [thick] (inc)  to [out=90, in =-90] (outg);
\end{knot}
\end{tikzpicture}
}.
\end{equation}
We have suppressed the labelling of the triangles by the idempotent $\Pi_{z,z'}$, and will henceforth use unlabelled triangles to denote the inclusions and projections for $\Pi_{z,z'}$.

The subobject associated to $\Pi_{z,z'}$ has the crucial property that the half-braidings associated to both factors agree, as is expressed by the following lemma.

\begin{lem}\label{STkeyprop}
Let $z, z'\in \dcentcat{A}$, then we have, with the notation as above:
$$
\hbox{
\begin{tikzpicture}[baseline=(current  bounding  box.center)]

\node (inc) at (-0.6,0.2){$a$};

\node (truw) at (-0.2,00.2){$z$};
\node (true) at (0.2,0.2){$z'$};

\node (tr1) at (0,2.2) {$\bigtriangleup$};
\coordinate (trd1w) at (-0.1,2.12);
\coordinate (trd1e) at (0.1,2.12);
\coordinate (tru1) at (0,2.35);
\node (c1) at (0,3) {};

\coordinate (rc) at (0.6,2);
\coordinate (outg) at (0.6,2.8);

\begin{knot}[clip width=4]
\strand [thick] (true) to [out=90, in=-70] (trd1e);
\strand [thick] (tru1) to (c1);
\strand [thick] (inc)  to [out=90, in =-90] (rc) to [out=90, in=-90] (outg);
\strand [thick] (truw) to [out=90,in=-110] (trd1w);
\end{knot}
\end{tikzpicture}
}
=
\hbox{
\begin{tikzpicture}[baseline=(current  bounding  box.center)]

\node (inc) at (-0.6,0.2){$a$};

\node (truw) at (-0.2,00.2){$z$};
\node (true) at (0.2,0.2){$z'$};

\node (tr1) at (0,2.2) {$\bigtriangleup$};
\coordinate (trd1w) at (-0.1,2.12);
\coordinate (trd1e) at (0.1,2.12);
\coordinate (tru1) at (0,2.35);
\node (c1) at (0,3) {};

\coordinate (rc) at (0.6,2);
\coordinate (outg) at (0.6,2.8);

\begin{knot}[clip width=4]
\strand [thick] (truw) to [out=90,in=-110] (trd1w);
\strand [thick] (tru1) to (c1);
\strand [thick] (inc)  to [out=90, in=-90] (rc) to [out=90, in=-90] (outg);
\strand [thick] (true) to [out=90, in=-70] (trd1e);
\flipcrossings{}
\end{knot}
\end{tikzpicture}
}
\mbox{ and }
\hbox{
\begin{tikzpicture}[baseline=(current  bounding  box.center)]

\coordinate (cf) at (0,-0.5);
\node (inc) at (-0.6,-0.4){$a$};

\coordinate  (rc) at (-0.6,0);

\node (tr) at (0,0.1) {$\bigtriangledown$};
\coordinate (trd) at (0,0);
\coordinate (truw) at (-0.1,00.2);
\coordinate (true) at (0.1,0.2);

\node (trd1w) at (-0.2,2.12){$z$};
\node (trd1e) at (0.2,2.12){$z'$};

\coordinate (outg) at (0.6,2.12);

\begin{knot}[clip width=4]
\strand [thick] (cf) to (trd);
\strand [thick] (true) to [out=70, in=-90] (trd1e);
\strand [thick] (inc)  to [out=90, in =-90] (rc) to [out=90, in =-90] (outg);
\strand [thick] (truw) to [out=110,in=-90] (trd1w);
\end{knot}
\end{tikzpicture}
}
=
\hbox{
\begin{tikzpicture}[baseline=(current  bounding  box.center)]

\coordinate (cf) at (0,-0.5);
\node (inc) at (-0.6,-0.4){$a$};

\coordinate  (rc) at (-0.6,0);

\node (tr) at (0,0.1) {$\bigtriangledown$};
\coordinate (trd) at (0,0);
\coordinate (truw) at (-0.1,00.2);
\coordinate (true) at (0.1,0.2);

\node (trd1w) at (-0.2,2.12){$z$};
\node (trd1e) at (0.2,2.12){$z'$};

\coordinate (outg) at (0.6,2.12);

\begin{knot}[clip width=4]
\strand [thick] (cf) to (trd);
\strand [thick] (true) to [out=70, in=-90] (trd1e);
\strand [thick] (inc)  to [out=90, in =-90] (rc) to [out=90, in =-90] (outg);
\strand [thick] (truw) to [out=110,in=-90] (trd1w);
\flipcrossings{1,2}
\end{knot}
\end{tikzpicture}
},
$$
for any $a\in\cat{A}$. 

\end{lem}
\begin{proof}
We prove one of the relations, the other is similar. Using both the conditions (Equation \ref{STinclprojprops}) on the projection and inclusion in the first identity to get the ring, we see that:
\begin{equation}\label{STtralala}
\hbox{
\begin{tikzpicture}[baseline=(current  bounding  box.center)]

\coordinate (cf) at (0,-0.5);
\node (inc) at (-0.6,-0.4){$a$};

\node (truw) at (-0.2,-0.4){$z$};
\node (true) at (0.2,-0.4){$z'$};

\node (tr1) at (0,2.2) {$\bigtriangleup$};
\coordinate (trd1w) at (-0.1,2.12);
\coordinate (trd1e) at (0.1,2.12);
\coordinate (tru1) at (0,2.35);
\node (c1) at (0,3) {};

\coordinate (outg) at (0.5,2.8);

\begin{knot}[clip width=4]
\strand [thick] (truw) to [out=90,in=-110] (trd1w);
\strand [thick] (tru1) to (c1);
\strand [thick] (inc)  to [out=90, in=-90] (outg);
\strand [thick] (true) to [out=90, in=-70] (trd1e);
\flipcrossings{}
\end{knot}
\end{tikzpicture}
}
=
\hbox{
\begin{tikzpicture}[baseline=(current  bounding  box.center)]

\coordinate (cf) at (0,-0.5);
\node (inc) at (-0.6,-0.4){$a$};

\node (truw) at (-0.2,-0.4){$z$};
\node (true) at (0.2,-0.4){$z'$};

\coordinate (west) at (-0.6,1.4);
\coordinate (north) at (0,1.6);
\coordinate (east) at (0.6,1.4);
\coordinate (south) at (0, 1.2);

\node (tr1) at (0,2.2) {$\bigtriangleup$};
\coordinate (trd1w) at (-0.1,2.12);
\coordinate (trd1e) at (0.1,2.12);
\coordinate (tru1) at (0,2.35);
\node (c1) at (0,3) {};

\coordinate (outg) at (1,2.8);
\coordinate (rctrl) at (1,1.1);
\coordinate (lctrl) at (-0.6,0.3);

\coordinate (rwall) at (0.3,1.05);

\begin{knot}[clip width=4]
\strand [blue, thick] (west) to [out=90,in=180] (north) to [out=0,in=90] (east) to [out=-90,in=0] (south) to [out=180,in=-90] (west);
\strand [thick] (truw) to [out=90,in=-110] (trd1w);
\strand [thick] (tru1) to (c1);
\strand [thick] (inc) to [out=90, in =-90] (lctrl) to [out=90, in =-90] (rctrl) to (outg);
\strand [thick] (true)to [out=90,in=-70] (trd1e);
\flipcrossings{1,4}
\end{knot}
\end{tikzpicture}
}
=
\hbox{
\begin{tikzpicture}[baseline=(current  bounding  box.center)]

\coordinate (cf) at (0,-0.5);
\node (inc) at (-0.6,-0.4){$a$};

\node (truw) at (-0.2,-0.4){$z$};
\node (true) at (0.2,-0.4){$z'$};

\coordinate (west) at (-0.6,0.8);
\coordinate (north) at (0,1);
\coordinate (east) at (0.6,0.8);
\coordinate (south) at (0, 0.6);

\node (tr1) at (0,2.2) {$\bigtriangleup$};
\coordinate (trd1w) at (-0.1,2.12);
\coordinate (trd1e) at (0.1,2.12);
\coordinate (tru1) at (0,2.35);
\node (c1) at (0,3) {};

\coordinate (outg) at (0.5,2.8);
\coordinate (rctrl) at (0.5,1.7);
\coordinate (lctrl) at (-1,0.9);

\coordinate (rwall) at (0.3,1.05);

\begin{knot}[clip width=4]
\strand [blue, thick] (west) to [out=90,in=180] (north) to [out=0,in=90] (east) to [out=-90,in=0] (south) to [out=180,in=-90] (west);
\strand [thick] (truw) to [out=90,in=-110] (trd1w);
\strand [thick] (tru1) to (c1);
\strand [thick] (inc) to [out=90, in =-90] (lctrl) to [out=90, in =-90] (rctrl) to (outg);
\strand [thick] (true)to [out=90,in=-70] (trd1e);
\flipcrossings{1,4}
\end{knot}
\end{tikzpicture}
}
=
\hbox{
\begin{tikzpicture}[baseline=(current  bounding  box.center)]

\coordinate (cf) at (0,-0.5);
\node (inc) at (-0.6,-0.4){$a$};

\node (truw) at (-0.2,-0.4){$z$};
\node (true) at (0.2,-0.4){$z'$};

\coordinate (west) at (-0.6,1.4);
\coordinate (north) at (0,1.6);
\coordinate (east) at (0.6,1.4);
\coordinate (south) at (0, 1.2);

\node (tr1) at (0,2.2) {$\bigtriangleup$};
\coordinate (trd1w) at (-0.1,2.12);
\coordinate (trd1e) at (0.1,2.12);
\coordinate (tru1) at (0,2.35);
\node (c1) at (0,3) {};

\coordinate (outg) at (1,2.8);
\coordinate (rctrl) at (1,1.1);
\coordinate (lctrl) at (-0.6,0.3);

\coordinate (rwall) at (0.3,1.05);

\begin{knot}[clip width=4]
\strand [blue, thick] (west) to [out=90,in=180] (north) to [out=0,in=90] (east) to [out=-90,in=0] (south) to [out=180,in=-90] (west);
\strand [thick] (truw) to [out=90,in=-110] (trd1w);
\strand [thick] (tru1) to (c1);
\strand [thick] (inc) to [out=90, in =-90] (lctrl) to [out=90, in =-90] (rctrl) to (outg);
\strand [thick] (true)to [out=90,in=-70] (trd1e);
\flipcrossings{1,4,5,6}
\end{knot}
\end{tikzpicture}
}
=
\hbox{
\begin{tikzpicture}[baseline=(current  bounding  box.center)]

\coordinate (cf) at (0,-0.5);
\node (inc) at (-0.6,-0.4){$a$};

\node (truw) at (-0.2,-0.4){$z$};
\node (true) at (0.2,-0.4){$z'$};

\node (tr1) at (0,2.2) {$\bigtriangleup$};
\coordinate (trd1w) at (-0.1,2.12);
\coordinate (trd1e) at (0.1,2.12);
\coordinate (tru1) at (0,2.35);
\node (c1) at (0,3) {};

\coordinate (outg) at (0.5,2.8);

\begin{knot}[clip width=4]
\strand [thick] (truw) to [out=90,in=-110] (trd1w);
\strand [thick] (tru1) to (c1);
\strand [thick] (inc)  to [out=90, in=-90] (outg);
\strand [thick] (true) to [out=90, in=-70] (trd1e);
\flipcrossings{1,2}
\end{knot}
\end{tikzpicture}
}
,
\end{equation}
using the fact that the loop is transparent to the $a$-strand (as they are both labelled by objects of $\cat{A}$) in the second identity to push the loop down past the strand. The third equality using cloaking from Lemma \ref{STcloaking} to flip the crossings on the $a$-strand, while pushing the loop back up. The final equality comes from reverse of the step done in the first.
\end{proof}

\subsection{The Symmetric Tensor Product on Objects}
\subsubsection{Definition on Objects}

\begin{df}\label{STsymtensobj}
Let $z,z'\in\dcentcat{A}$, and write $\Phi:\dcentcat{A}\rar\cat{A}$ for the forgetful functor (cf. Definition \ref{STdrinfeldcenterdef}). The \emph{symmetric tensor product} $z\otimes_s z'\in \dcentcat{A}$ of $z$ and $z'$ is the object $(\Phi(z\otimes_{\Pi} z'),\beta)$, where $z\otimes_{\Pi} z'$ is the subobject associated to $\Pi_{z,z'}$, and $\beta$ is the half-braiding with components, for $a\in\cat{A}$:
\begin{equation}\label{SThalfbraid}
\beta_a=
\hbox{
\begin{tikzpicture}[baseline=(current  bounding  box.center)]

\node (cf) at (0,-0.5){$z\otimes_s z'$};
\node (inc) at (-0.75,-0.5){$a$};

\node (c1) at (0,3) {};

\coordinate (outg) at (0.75,2.8);

\begin{knot}[clip width=4]
\strand [thick] (cf) to (c1);
\strand [thick] (inc)  to [out=90, in =-90] (outg);
\end{knot}
\end{tikzpicture}
}
:=
\hbox{
\begin{tikzpicture}[baseline=(current  bounding  box.center)]

\node (cf) at (0,-0.5){$z\otimes_s z'$};
\node (inc) at (-0.75,-0.5){$a$};

\node (tr) at (0,0.1) {$\bigtriangledown$};
\coordinate (trd) at (0,0);
\coordinate (truw) at (-0.1,00.2);
\coordinate (true) at (0.1,0.2);

\node (tr1) at (0,2.2) {$\bigtriangleup$};
\coordinate (trd1w) at (-0.1,2.12);
\coordinate (trd1e) at (0.1,2.12);
\coordinate (tru1) at (0,2.35);
\node (c1) at (0,3) {};

\coordinate (outg) at (0.75,2.8);

\begin{knot}[clip width=4]
\strand [thick] (cf) to (trd);
\strand [thick] (true) to [out=70, in=-70] (trd1e);
\strand [thick] (tru1) to (c1);
\strand [thick] (inc)  to [out=90, in =-90] (outg);
\strand [thick] (truw) to [out=110,in=-110] (trd1w);
\end{knot}
\end{tikzpicture}
}
=
\hbox{
\begin{tikzpicture}[baseline=(current  bounding  box.center)]

\node (cf) at (0,-0.5){$z\otimes_s z'$};
\node (inc) at (-0.75,-0.5){$a$};

\node (tr) at (0,0.1) {$\bigtriangledown$};
\coordinate (trd) at (0,0);
\coordinate (truw) at (-0.1,00.2);
\coordinate (true) at (0.1,0.2);

\node (tr1) at (0,2.2) {$\bigtriangleup$};
\coordinate (trd1w) at (-0.1,2.12);
\coordinate (trd1e) at (0.1,2.12); 
\coordinate (tru1) at (0,2.35);
\node (c1) at (0,3) {};

\coordinate (outg) at (0.5,2.8);

\begin{knot}[clip width=4]
\strand [thick] (cf) to (trd);
\strand [thick] (truw) to [out=110,in=-110] (trd1w);
\strand [thick] (tru1) to (c1);
\strand [thick] (inc)  to [out=90, in=-90] (outg);
\strand [thick] (true) to [out=70, in=-70] (trd1e);
\flipcrossings{}
\end{knot}
\end{tikzpicture}
} ,
\end{equation}
where the equality is a consequence of Lemma \ref{STkeyprop}.
\end{df}

We observe that the $\beta_{a}$ indeed satisfy the hexagon equation (see Definition \ref{STdrinfeldcenterdef} of the Drinfeld centre), which in this case reads $\beta_{aa'}=(\beta_{a}\otimes\id_{a'})\circ(\id_a\otimes\beta_{a'})$:
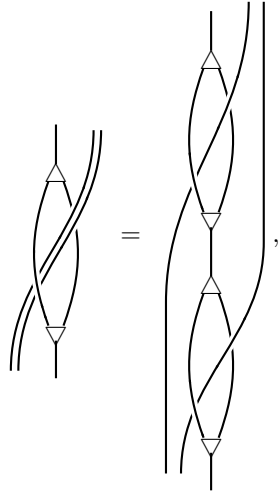
$$
\hbox{
\begin{tikzpicture}[baseline=(current  bounding  box.center)]

\coordinate (cf) at (0,-0.5);
\coordinate (inc) at (-0.5,-0.4);
\coordinate (inc1) at (-0.6,-0.4);

\node (tr) at (0,0.1) {$\bigtriangledown$};
\coordinate (trd) at (0,0);
\coordinate (truw) at (-0.1,00.2);
\coordinate (true) at (0.1,0.2);

\node (tr1) at (0,2.2) {$\bigtriangleup$};
\coordinate (trd1w) at (-0.1,2.12);
\coordinate (trd1e) at (0.1,2.12); 
\coordinate (tru1) at (0,2.35);
\node (c1) at (0,3) {};

\coordinate (outg) at (0.6,2.8);
\coordinate (outg1) at (0.5,2.8);

\begin{knot}[clip width=4]
\strand [thick] (cf) to (trd);
\strand [thick] (truw) to [out=110,in=-110] (trd1w);
\strand [thick] (tru1) to (c1);
\strand [thick] (inc)  to [out=90, in=-90] (outg);
\strand [thick] (inc1)  to [out=90, in=-90] (outg1);
\strand [thick] (true) to [out=70, in=-70] (trd1e);
\flipcrossings{}
\end{knot}
\end{tikzpicture}
}
=
\hbox{
\begin{tikzpicture}[baseline=(current  bounding  box.center)]

\coordinate (cf) at (0,-0.5);
\node (inc) at (-0.4,-0.4){};

\node (tr) at (0,0.1) {$\bigtriangledown$};
\coordinate (trd) at (0,0);
\coordinate (truw) at (-0.1,00.2);
\coordinate (true) at (0.1,0.2);

\node (tr1) at (0,2.2) {$\bigtriangleup$};
\coordinate (trd1w) at (-0.1,2.12);
\coordinate (trd1e) at (0.1,2.12); 
\coordinate (tru1) at (0,2.35);

\coordinate (rc) at (0.7,2.8);
\coordinate (outg) at (0.7,6);

\node (inc1) at (-0.6,-0.4){};
\coordinate (lc) at (-0.6,2);

\node (tr2) at (0,3.1) {$\bigtriangledown$};
\coordinate (trd2) at (0,3);
\coordinate (truw2) at (-0.1,3.2);
\coordinate (true2) at (0.1,3.2);

\node (tr12) at (0,5.2) {$\bigtriangleup$};
\coordinate (trd1w2) at (-0.1,5.12);
\coordinate (trd1e2) at (0.1,5.12); 
\coordinate (tru12) at (0,5.35);
\node (c1) at (0,6) {};

\coordinate (outg1) at (0.5,6);

\begin{knot}[clip width=4]
\strand [thick] (cf) to (trd);
\strand [thick] (truw) to [out=110,in=-110] (trd1w);
\strand [thick] (inc)  to [out=90, in=-90] (rc) to [out=90, in=-90] (outg);
\strand [thick] (true) to [out=70, in=-70] (trd1e);
\strand [thick] (tru1) to (trd2);
\strand [thick] (truw2) to [out=110,in=-110] (trd1w2);
\strand [thick] (tru12) to (c1);
\strand [thick] (inc1)  to [out=90, in=-90] (lc) to [out=90, in=-90] (outg1);
\strand [thick] (true2) to [out=70, in=-70] (trd1e2);
\flipcrossings{}
\end{knot}
\end{tikzpicture}
},
$$
using Equation \ref{STinclprojprops} and cloaking (Lemma \ref{STcloaking}).

Lemma \ref{STkeyprop} ensures this definition of the half-braiding for $z\otimes_s z'$ does not depend on a choice between $z$ and $z'$. It should be noted that that the inclusion and projection for $\Pi_{z,z'}$ do not commute with the half-braiding (so in particular $z\otimes_s z'$ is not a subobject of $z\otimes_c z'$ in $\dcentcat{A}$, but $z\otimes_{\Pi} z'$ is). Instead we have the following relation that we will call \emph{slicing}. Note that this lemma expresses equalities between morphisms in $\cat{A}$.

\begin{lem}[Slicing]\label{STslicing}
The half-braiding for $z\otimes_s z'$ and the inclusion and projection maps for $\Pi_{z,z'}$ interact as follows:
$$
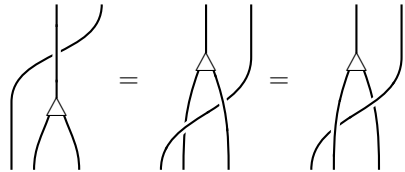
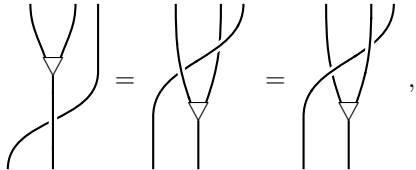
\hbox{
\begin{tikzpicture}[baseline=(current  bounding  box.center)]

\coordinate (c) at (-0.6,1);
\coordinate (a) at (-0.3,1);
\coordinate (b) at (0.3,1);

\coordinate (ctrl) at (-0.6,1.9);

\node (tr1) at (0,1.8) {$\bigtriangleup$};
\coordinate (trd1w) at (-0.1,1.72);
\coordinate (trd1e) at (0.1,1.72);
\coordinate (tru1) at (0,1.95);
\coordinate (c1) at (0,3.2);

\coordinate (co) at (0.6,3.2);

\begin{knot}[clip width=4]
\strand [thick] (c) to [out=90,in=-90] (ctrl) to [out=90,in=-90] (co);
\strand [thick] (b) to [out=90, in=-70] (trd1e);
\strand [thick] (tru1) to (c1);
\strand [thick] (a) to [out=90,in=-110] (trd1w);
\flipcrossings{1}
\end{knot}
\end{tikzpicture}
}
=
\hbox{
\begin{tikzpicture}[baseline=(current  bounding  box.center)]

\coordinate (c) at (-0.6,0.8);
\coordinate (a) at (-0.3,0.8);
\coordinate (b) at (0.3,0.8);

\coordinate (ctrl) at (0.6,2.3);

\node (tr1) at (0,2.2) {$\bigtriangleup$};
\coordinate (trd1w) at (-0.1,2.12);
\coordinate (trd1e) at (0.1,2.12);
\coordinate (tru1) at (0,2.35);
\coordinate (c1) at (0,3);

\coordinate (co) at (0.6,3);

\begin{knot}[clip width=4]
\strand [thick] (c) to [out=90,in=-90] (ctrl) to [out=90,in=-90] (co);
\strand [thick] (b) to [out=90, in=-70] (trd1e);
\strand [thick] (tru1) to (c1);
\strand [thick] (a) to [out=90,in=-110] (trd1w);
\flipcrossings{1}
\end{knot}
\end{tikzpicture}
}
=
\hbox{
\begin{tikzpicture}[baseline=(current  bounding  box.center)]

\coordinate (c) at (-0.6,0.8);
\coordinate (a) at (-0.3,0.8);
\coordinate (b) at (0.3,0.8);

\coordinate (ctrl) at (0.6,2.3);

\node (tr1) at (0,2.2) {$\bigtriangleup$};
\coordinate (trd1w) at (-0.1,2.12);
\coordinate (trd1e) at (0.1,2.12);
\coordinate (tru1) at (0,2.35);
\coordinate (c1) at (0,3);

\coordinate (co) at (0.6,3);

\begin{knot}[clip width=4]
\strand [thick] (c) to [out=90,in=-90] (ctrl) to [out=90,in=-90] (co);
\strand [thick] (b) to [out=90, in=-70] (trd1e);
\strand [thick] (tru1) to (c1);
\strand [thick] (a) to [out=90,in=-110] (trd1w);
\flipcrossings{2}
\end{knot}
\end{tikzpicture}
}
$$
and
$$
\rotatebox{180}{
\begin{tikzpicture}[baseline=(current  bounding  box.center)]

\coordinate (c) at (-0.6,1);
\coordinate (a) at (-0.3,1);
\coordinate (b) at (0.3,1);

\coordinate (ctrl) at (-0.6,1.9);

\node (tr1) at (0,1.8) {$\bigtriangleup$};
\coordinate (trd1w) at (-0.1,1.72);
\coordinate (trd1e) at (0.1,1.72);
\coordinate (tru1) at (0,1.95);
\coordinate (c1) at (0,3.2);

\coordinate (co) at (0.6,3.2);

\begin{knot}[clip width=4]
\strand [thick] (c) to [out=90,in=-90] (ctrl) to [out=90,in=-90] (co);
\strand [thick] (b) to [out=90, in=-70] (trd1e);
\strand [thick] (tru1) to (c1);
\strand [thick] (a) to [out=90,in=-110] (trd1w);
\flipcrossings{1}
\end{knot}
\end{tikzpicture}
}
=
\rotatebox{180}{
\begin{tikzpicture}[baseline=(current  bounding  box.center)]

\coordinate (c) at (-0.6,0.8);
\coordinate (a) at (-0.3,0.8);
\coordinate (b) at (0.3,0.8);

\coordinate (ctrl) at (0.6,2.3);

\node (tr1) at (0,2.2) {$\bigtriangleup$};
\coordinate (trd1w) at (-0.1,2.12);
\coordinate (trd1e) at (0.1,2.12);
\coordinate (tru1) at (0,2.35);
\coordinate (c1) at (0,3);

\coordinate (co) at (0.6,3);

\begin{knot}[clip width=4]
\strand [thick] (c) to [out=90,in=-90] (ctrl) to [out=90,in=-90] (co);
\strand [thick] (b) to [out=90, in=-70] (trd1e);
\strand [thick] (tru1) to (c1);
\strand [thick] (a) to [out=90,in=-110] (trd1w);
\flipcrossings{1}
\end{knot}
\end{tikzpicture}
}
=
\rotatebox{180}{
\begin{tikzpicture}[baseline=(current  bounding  box.center)]

\coordinate (c) at (-0.6,0.8);
\coordinate (a) at (-0.3,0.8);
\coordinate (b) at (0.3,0.8);

\coordinate (ctrl) at (0.6,2.3);

\node (tr1) at (0,2.2) {$\bigtriangleup$};
\coordinate (trd1w) at (-0.1,2.12);
\coordinate (trd1e) at (0.1,2.12);
\coordinate (tru1) at (0,2.35);
\coordinate (c1) at (0,3);

\coordinate (co) at (0.6,3);

\begin{knot}[clip width=4]
\strand [thick] (c) to [out=90,in=-90] (ctrl) to [out=90,in=-90] (co);
\strand [thick] (b) to [out=90, in=-70] (trd1e);
\strand [thick] (tru1) to (c1);
\strand [thick] (a) to [out=90,in=-110] (trd1w);
\flipcrossings{2}
\end{knot}
\end{tikzpicture}
},
$$
where the diagonal strand is labelled by an object of $\cat{A}$.
\end{lem}

\begin{proof}
Plugging in the definition of the half-braiding (Equation \eqref{SThalfbraid}) as the first equality, we have: 
$$
\hbox{
\begin{tikzpicture}[baseline=(current  bounding  box.center)]

\coordinate (c) at (-0.6,1);
\coordinate (a) at (-0.3,1);
\coordinate (b) at (0.3,1);

\coordinate (ctrl) at (-0.6,1.9);

\node (tr1) at (0,1.8) {$\bigtriangleup$};
\coordinate (trd1w) at (-0.1,1.72);
\coordinate (trd1e) at (0.1,1.72);
\coordinate (tru1) at (0,1.95);
\coordinate (c1) at (0,3.2);

\coordinate (co) at (0.6,3.2);

\begin{knot}[clip width=4]
\strand [thick] (c) to [out=90,in=-90] (ctrl) to [out=90,in=-90] (co);
\strand [thick] (b) to [out=90, in=-70] (trd1e);
\strand [thick] (tru1) to (c1);
\strand [thick] (a) to [out=90,in=-110] (trd1w);
\flipcrossings{1}
\end{knot}
\end{tikzpicture}
}
=
\hbox{
\begin{tikzpicture}[baseline=(current  bounding  box.center)]

\coordinate (c) at (-0.6,1);
\coordinate (a) at (-0.3,1);
\coordinate (b) at (0.3,1);

\coordinate (ctrl) at (-0.6,2.7);
\coordinate (rctrl) at (0.6,3.6);

\node (tr1) at (0,1.8) {$\bigtriangleup$};
\coordinate (trd1w) at (-0.1,1.72);
\coordinate (trd1e) at (0.1,1.72);
\coordinate (tru1) at (0,1.95);

\node (tr2) at (0,2.3) {$\bigtriangledown$};
\coordinate (trd2) at (0,2.2);
\coordinate (truw2) at (-0.1,2.4);
\coordinate (true2) at (0.1,2.4);

\node (tr3) at (0,3.7) {$\bigtriangleup$};
\coordinate (trdw3) at (-0.1,3.62);
\coordinate (trde3) at (0.1,3.62);
\coordinate (tru3) at (0,3.85);
\coordinate (c1) at (0,4);

\coordinate (co) at (0.6,4);

\begin{knot}[clip width=4]
\strand [thick] (c) to [out=90,in=-90] (ctrl) to [out=90,in=-90] (rctrl) to [out=90,in=-90] (co);
\strand [thick] (b) to [out=90, in=-70] (trd1e);
\strand [thick] (tru1) to (trd2);
\strand [thick] (truw2) to [out=110,in=-110] (trdw3);
\strand [thick] (true2) to [out=70,in=-70] (trde3);
\strand [thick] (tru3) to (c1);
\strand [thick] (a) to [out=90,in=-110] (trd1w);
\flipcrossings{1}
\end{knot}
\end{tikzpicture}
}
=
\hbox{
\begin{tikzpicture}[baseline=(current  bounding  box.center)]

\coordinate (c) at (-0.6,1);
\coordinate (a) at (-0.3,1);
\coordinate (b) at (0.3,1);

\coordinate (ctrl) at (-0.6,1.7);
\coordinate (rctrl) at (0.6,2.6);

\coordinate (west) at (-0.5,2.9);
\coordinate (north) at (0,3.2);
\coordinate (east) at (0.5,2.9);
\coordinate (south) at (0,2.6);

\node (tr3) at (0,3.7) {$\bigtriangleup$};
\coordinate (trdw3) at (-0.1,3.62);
\coordinate (trde3) at (0.1,3.62);
\coordinate (tru3) at (0,3.85);
\coordinate (c1) at (0,4);

\coordinate (co) at (0.6,4);

\begin{knot}[clip width=4]
\strand [blue, thick] (west) to [out=90,in=180] (north) to [out=0,in=90] (east) to [out=-90,in=0] (south) to [out=180,in=-90] (west);
\strand [thick] (c) to [out=90,in=-90] (ctrl) to [out=90,in=-90] (rctrl) to [out=90,in=-90] (co);
\strand [thick] (b) to [out=90, in=-70] (trde3);
\strand [thick] (tru3) to (c1);
\strand [thick] (a) to [out=90,in=-110] (trdw3);
\flipcrossings{3,2,4}
\end{knot}
\end{tikzpicture}
}
=
\hbox{
\begin{tikzpicture}[baseline=(current  bounding  box.center)]

\coordinate (c) at (-0.6,0.8);
\coordinate (a) at (-0.3,0.8);
\coordinate (b) at (0.3,0.8);

\coordinate (ctrl) at (0.6,2.3);

\node (tr1) at (0,2.2) {$\bigtriangleup$};
\coordinate (trd1w) at (-0.1,2.12);
\coordinate (trd1e) at (0.1,2.12);
\coordinate (tru1) at (0,2.35);
\coordinate (c1) at (0,3);

\coordinate (co) at (0.6,3);

\begin{knot}[clip width=4]
\strand [thick] (c) to [out=90,in=-90] (ctrl) to [out=90,in=-90] (co);
\strand [thick] (b) to [out=90, in=-70] (trd1e);
\strand [thick] (tru1) to (c1);
\strand [thick] (a) to [out=90,in=-110] (trd1w);
\flipcrossings{1}
\end{knot}
\end{tikzpicture}
},
$$
where in the second equality we used Equation \eqref{STincprojprop} to combine the inclusion-projection pair to a ring, followed by an application of cloaking to pass the loop up at the cost of swapping crossings on the strand. The loop cancels with the projection, like in the last step of Equation \eqref{STtralala}.
\end{proof}

\subsubsection{Symmetry of the Symmetric Tensor Product}
The symmetric tensor product is indeed symmetric:
\begin{lem}\label{STsymlemma}
The symmetry in $\cat{A}$ induces an isomorphism between $z\otimes_s z'$ and $z \otimes_s z'$. That is, using the triangle notation for the inclusions and projections,
\begin{equation}\label{STsymmorph}
\hbox{
\begin{tikzpicture}[baseline=(current  bounding  box.center)]

\node (cf) at (0,-0.5){$z\otimes_s z'$};

\node (tr) at (0,0.1) {$\bigtriangledown$};
\coordinate (trd) at (0,0);
\coordinate (truw) at (-0.1,00.2);
\coordinate (true) at (0.1,0.2);

\node (tr1) at (0,1.2) {$\bigtriangleup$};
\coordinate (trd1w) at (-0.1,1.12);
\coordinate (trd1e) at (0.1,1.12);
\coordinate (tru1) at (0,1.35);
\node (c1) at (0,2) {$z'\otimes_s z$};

\begin{knot}[clip width=4]
\strand [thick] (cf) to (trd);
\strand [thick] (true) to [out=70, in=-110] (trd1w);
\strand [thick] (tru1) to (c1);
\end{knot}
\draw [thick] (truw) to [out=110,in=-70] (trd1e);
\end{tikzpicture}
}
\mbox{ and }
\hbox{
\begin{tikzpicture}[baseline=(current  bounding  box.center)]

\node (cf) at (0,-0.5){$z'\otimes_s z$};

\node (tr) at (0,0.1) {$\bigtriangledown$};
\coordinate (trd) at (0,0);
\coordinate (truw) at (-0.1,00.2);
\coordinate (true) at (0.1,0.2);

\node (tr1) at (0,1.2) {$\bigtriangleup$};
\coordinate (trd1w) at (-0.1,1.12);
\coordinate (trd1e) at (0.1,1.12);
\coordinate (tru1) at (0,1.35);
\node (c1) at (0,2) {$z\otimes_s z'$};

\begin{knot}[clip width=4]
\strand [thick] (cf) to (trd);
\strand [thick] (true) to [out=70, in=-110] (trd1w);
\strand [thick] (tru1) to (c1);
\end{knot}
\draw [thick] (truw) to [out=110,in=-70] (trd1e);
\end{tikzpicture}
}
\end{equation}
are mutually inverse morphisms in $\dcentcat{A}$.
\end{lem}

\begin{proof}
We will first establish that the symmetry morphisms are mutually inverse in $\cat{A}$, then we will prove they lift to morphisms in $\dcentcat{A}$. Consider the composite 
$$
\hbox{
\begin{tikzpicture}[baseline=(current  bounding  box.center)]

\node (cf) at (0,-0.5){$z\otimes_s z'$};
\node (tr) at (0,0.1) {$\bigtriangledown$};
\coordinate (trd) at (0,0);
\coordinate (truw) at (-0.1,00.2);
\coordinate (true) at (0.1,0.2);
\node (tr1) at (0,1.2) {$\bigtriangleup$};
\coordinate (trd1w) at (-0.1,1.12);
\coordinate (trd1e) at (0.1,1.12);
\coordinate (tru1) at (0,1.35);

\node (tra) at (0,1.6) {$\bigtriangledown$};
\coordinate (trda) at (0,1.5);
\coordinate (truwa) at (-0.1,1.7);
\coordinate (truea) at (0.1,1.7);
\node (tr1a) at (0,2.7) {$\bigtriangleup$};
\coordinate (trd1wa) at (-0.1,2.62);
\coordinate (trd1ea) at (0.1,2.62);
\coordinate (tru1a) at (0,2.85);
\node (c1) at (0,3.5) {$z\otimes_s z'$};

\begin{knot}[clip width=4]
\strand [thick] (cf) to (trd);
\strand [thick] (true) to [out=70, in=-110] (trd1w);
\strand [thick] (truea) to [out=70, in=-110] (trd1wa);
\strand [thick] (tru1a) to (c1);
\end{knot}
\draw [thick] (truw) to [out=110,in=-70] (trd1e);
\draw [thick] (truwa) to [out=110,in=-70] (trd1ea);
\draw [thick] (tru1) to (trda);
\end{tikzpicture}
}
=
\hbox{
\begin{tikzpicture}[baseline=(current  bounding  box.center)]

\node (cf) at (0,-0.5){$z\otimes_s z'$};
\node (tr) at (0,0.1) {$\bigtriangledown$};
\coordinate (trd) at (0,0);
\coordinate (truw) at (-0.1,00.2);
\coordinate (true) at (0.1,0.2);

\node (tr1a) at (0,2.9) {$\bigtriangleup$};
\coordinate (trd1wa) at (-0.1,2.82);
\coordinate (trd1ea) at (0.1,2.82);
\coordinate (tru1a) at (0,3.05);
\node (c1) at (0,3.5) {$z\otimes_s z'$};

\coordinate (miduw) at (-0.25,1.61);
\coordinate (midue) at (0.25,1.61);

\coordinate (middw) at (-0.25,1.01);
\coordinate (midde) at (0.25,1.01);

\coordinate (west) at (-0.5,1.3);
\coordinate (north) at (0,1.6);
\coordinate (east) at (0.5,1.3);
\coordinate (south) at (0, 1.0);

\begin{knot}[clip width=4]
\strand [blue, thick] (west) to [out=90,in=180] (north) to [out=0,in=90] (east) to [out=-90,in=0] (south) to [out=180,in=-90] (west);
\strand [thick] (cf) to (trd);
\strand [thick] (middw) to [out=90,in=-90] (miduw) to [out=90,in=-70] (trd1ea);
\strand [thick] (truw) to [out=110,in=-90] (midde) to (midue);
\strand [thick] (tru1a) to (c1);
\flipcrossings{1,4}
\end{knot}
\draw [thick] (true) to [out=70,in=-90] (middw);
\draw [thick] (midue) to [out=90,in=-110] (trd1wa);
\end{tikzpicture}
}
=
\hbox{
\begin{tikzpicture}[baseline=(current  bounding  box.center)]

\node (cf) at (0,-0.5){$z\otimes_s z'$};
\node (tr) at (0,0.1) {$\bigtriangledown$};
\coordinate (trd) at (0,0);
\coordinate (truw) at (-0.1,00.2);
\coordinate (true) at (0.1,0.2);

\node (tr1a) at (0,2.9) {$\bigtriangleup$};
\coordinate (trd1wa) at (-0.1,2.82);
\coordinate (trd1ea) at (0.1,2.82);
\coordinate (tru1a) at (0,3.05);
\node (c1) at (0,3.5) {$z\otimes_s z'$};

\coordinate (miduw) at (-0.25,2);
\coordinate (midue) at (0.25,2);

\coordinate (middw) at (-0.25,0.9);
\coordinate (midde) at (0.25,0.9);

\coordinate (west) at (-0.5,1.1);
\coordinate (north) at (-0.1,1.5);
\coordinate (east) at (0.5,1.8);
\coordinate (south) at (0.1, 1.4);

\coordinate (sd) at (0.1,0.3);
\coordinate (nd) at (-0.1,2.6);

\begin{knot}[clip width=4]
\strand [blue, thick] (west) to [out=90,in=180] (north) to [out=0,in=180] (nd) to [out=0,in=90] (east) to [out=-90,in=0] (south) to [out=180,in=0] (sd) to [out=180,in=-90] (west);
\strand [thick] (cf) to (trd);
\strand [thick] (middw) to [out=90,in=-90] (miduw);
\strand [thick] (midde) to (midue);
\strand [thick] (tru1a) to (c1);
\flipcrossings{1,2}
\end{knot}
\draw [thick] (midue) to [out=90,in=-110] (trd1wa);
\draw [thick] (true)  to [out=70, in=-90] (middw);
\draw [thick] (miduw) to [out=90,in=-70] (trd1ea);
\draw [thick] (truw) to [out=110,in=-90] (midde);
\end{tikzpicture}
}
.
$$

Here the unresolved crossings denote the symmetry in $\cat{A}$. The first step comes from replacing the inclusion followed by the projection with the idempotent (cf. Section \ref{STidemnotation}). The second uses the fact that the symmetry in $\cat{A}$ allows us to do Reidemeister moves which involve only the unresolved crossings. We can now swap the strands with the braiding morphisms for $z$ and $z'$, undoing the symmetry crossings between the $z$ and $z'$ strands. We get:
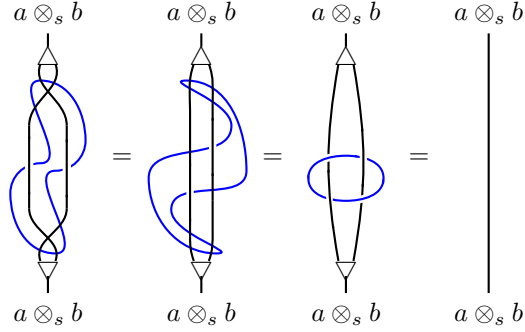
$$
\hbox{
	\begin{tikzpicture}[baseline=(current  bounding  box.center)]
	
	\node (cf) at (0,-0.5){$z\otimes_s z'$};
	\node (tr) at (0,0.1) {$\bigtriangledown$};
	\coordinate (trd) at (0,0);
	\coordinate (truw) at (-0.1,00.2);
	\coordinate (true) at (0.1,0.2);

	\node (tr1a) at (0,2.9) {$\bigtriangleup$};
	\coordinate (trd1wa) at (-0.1,2.82);
	\coordinate (trd1ea) at (0.1,2.82);
	\coordinate (tru1a) at (0,3.05);
	\node (c1) at (0,3.5) {$z\otimes_s z'$};
	
	\coordinate (miduw) at (-0.25,2);
	\coordinate (midue) at (0.25,2);
	
	\coordinate (middw) at (-0.25,0.9);
	\coordinate (midde) at (0.25,0.9);
	
	\coordinate (west) at (-0.5,1.1);
	\coordinate (north) at (-0.1,1.5);
	\coordinate (east) at (0.5,1.8);
	\coordinate (south) at (0.1, 1.4);
	
	\coordinate (sd) at (0.1,0.3);
	\coordinate (nd) at (-0.1,2.6);
	
	\begin{knot}[clip width=4]
	\strand [blue, thick] (west) to [out=90,in=180] (north) to [out=0,in=180] (nd) to [out=0,in=90] (east) to [out=-90,in=0] (south) to [out=180,in=0] (sd) to [out=180,in=-90] (west);
	\strand [thick] (cf) to (trd);
	\strand [thick] (middw) to [out=90,in=-90] (miduw);
	\strand [thick] (midde) to (midue);
	\strand [thick] (tru1a) to (c1);
	\flipcrossings{1,2}
	\end{knot}
	\draw [thick] (midue) to [out=90,in=-110] (trd1wa);
	\draw [thick] (true)  to [out=70, in=-90] (middw);
	\draw [thick] (miduw) to [out=90,in=-70] (trd1ea);
	\draw [thick] (truw) to [out=110,in=-90] (midde);
	\end{tikzpicture}
}
=
\hbox{
	\begin{tikzpicture}[baseline=(current  bounding  box.center)]
	
	\node (cf) at (0,-0.5){$z\otimes_s z'$};
	\node (tr) at (0,0.1) {$\bigtriangledown$};
	\coordinate (trd) at (0,0);
	\coordinate (truw) at (-0.1,00.2);
	\coordinate (true) at (0.1,0.2);

	\node (tr1a) at (0,2.9) {$\bigtriangleup$};
	\coordinate (trd1wa) at (-0.1,2.82);
	\coordinate (trd1ea) at (0.1,2.82);
	\coordinate (tru1a) at (0,3.05);
	\node (c1) at (0,3.5) {$z \otimes_s z'$};
	
	\coordinate (miduw) at (-0.15,2.2);
	\coordinate (midue) at (0.15,2.3);
	
	\coordinate (midw) at (-0.15,1.61);
	\coordinate (mide) at (0.15,1);
	
	\coordinate (middw) at (-0.15,0.53);
	\coordinate (midde) at (0.15,0.6);
	
	\coordinate (west) at (-0.7,1.3);
	\coordinate (north) at (0.4,2);
	\coordinate (east) at (0.6,1.4);
	\coordinate (south) at (-0.4, 0.88);
	
	\coordinate (sd) at (0.2,0.3);
	\coordinate (nd) at (-0.2,2.6);
	
	\begin{knot}[clip width=3.5]
	\strand [blue, thick] (west) to [out=90,in=-90] (north) to [out=90,in=180] (nd) to [out=0,in=90] (east) to [out=-90,in=90] (south) to [out=-90,in=0] (sd) to [out=180,in=-90] (west);
	\strand [thick] (cf) to (trd);
	\strand [thick] (midw) to (middw);
	\strand [thick] (mide) to (midue);
	\strand [thick] (tru1a) to (c1);
	\flipcrossings{1,2}
	\end{knot}
	\draw [thick] (truw) to [out=110,in=-90] (middw);
	\draw [thick] (midue) to [out=90,in=-70] (trd1ea);
	\draw [thick] (midw) to (miduw) to [out=90,in=-110] (trd1wa);
	\draw [thick] (true) to [out=70,in=-90] (midde) to (mide);
	\end{tikzpicture}
}
=
\hbox{
	\begin{tikzpicture}[baseline=(current  bounding  box.center)]
	
	\node (cf) at (0,-0.5){$z \otimes_s z'$};
	\node (tr) at (0,0.1) {$\bigtriangledown$};
	\coordinate (trd) at (0,0);
	\coordinate (truw) at (-0.1,00.2);
	\coordinate (true) at (0.1,0.2);

	\node (tr1a) at (0,2.9) {$\bigtriangleup$};
	\coordinate (trd1wa) at (-0.1,2.82);
	\coordinate (trd1ea) at (0.1,2.82);
	\coordinate (tru1a) at (0,3.05);
	\node (c1) at (0,3.5) {$z \otimes_s z'$};

	\coordinate (west) at (-0.5,1.3);
	\coordinate (north) at (0,1.6);
	\coordinate (east) at (0.5,1.3);
	\coordinate (south) at (0, 1.0);
	
	\begin{knot}[clip width=3.5]
	\strand [blue, thick] (west) to [out=90,in=180] (north) to [out=0,in=90] (east) to [out=-90,in=0] (south) to [out=180,in=-90] (west);
	\strand [thick] (cf) to (trd);
	\strand [thick] (truw) to [out=100,in=-100] (trd1wa);
	\strand [thick] (true) to [out=80,in=-80] (trd1ea);
	\strand [thick] (tru1a) to (c1);
	\flipcrossings{2,3}
	\end{knot}
	\end{tikzpicture}
}
=
\hbox{
	\begin{tikzpicture}[baseline=(current  bounding  box.center)]
	
	\node (cf) at (0,-0.5){$z \otimes_s z'$};
	
	\node (c1) at (0,3.5) {$z \otimes_s z'$};

	\begin{knot}[clip width=3.5]
	\strand [thick] (cf) to (c1);
	\end{knot}
	\end{tikzpicture}
},
$$
where the first equality is the swap, the second is similar to the reverse to Reidemeister move done in the previous equation, and the final is a basic property of the idempotent. A similar argument shows the other composite is also the identity.

We still need to establish that the morphisms are indeed morphisms in $\dcentcat{A}$. That is, we need to show that they commute with the braiding as defined in Equation \eqref{SThalfbraid}. We compute, using slicing (Lemma \ref{STslicing}) and Equation \eqref{STcrossinginteraction} expressing how the crossings interact:
$$
\hbox{
\begin{tikzpicture}[baseline=(current  bounding  box.center)]

\node (i) at (-0.75,-0.5) {$a$};
\coordinate (lc) at (-0.75,1.5);
\coordinate (rc) at (0.75,2.5);
\coordinate (o) at (0.75,3);

\node (cf) at (0,-0.5){$z \otimes_s z'$};

\node (tr) at (0,0.1) {$\bigtriangledown$};
\coordinate (trd) at (0,0);
\coordinate (truw) at (-0.1,00.2);
\coordinate (true) at (0.1,0.2);

\node (tr1) at (0,1.2) {$\bigtriangleup$};
\coordinate (trd1w) at (-0.1,1.12);
\coordinate (trd1e) at (0.1,1.12);
\coordinate (tru1) at (0,1.35);
\node (c1) at (0,3) {$z' \otimes_s z$};

\begin{knot}[clip width=4]
\strand [thick] (cf) to (trd);
\strand [thick] (true) to [out=70, in=-110] (trd1w);
\strand [thick] (tru1) to (c1);
\strand [thick] (i) to [out=90,in=-90] (lc) to [out=90,in=-90] (rc) to [out=90,in=-90] (o);
\end{knot}
\draw [thick] (truw) to [out=110,in=-70] (trd1e);
\end{tikzpicture}
}
=
\hbox{
\begin{tikzpicture}[baseline=(current  bounding  box.center)]

\node (i) at (-0.75,-0.5) {$a$};
\coordinate (lc) at (-0.75,1.25);
\coordinate (rc) at (0.75,2);
\coordinate (o) at (0.75,3);

\node (cf) at (0,-0.5){$z \otimes_s z'$};

\node (tr) at (0,0.1) {$\bigtriangledown$};
\coordinate (trd) at (0,0);
\coordinate (truw) at (-0.1,00.2);
\coordinate (true) at (0.1,0.2);

\node (tr1) at (0,2.2) {$\bigtriangleup$};
\coordinate (trd1w) at (-0.1,2.12);
\coordinate (trd1e) at (0.1,2.12);
\coordinate (tru1) at (0,2.35);
\node (c1) at (0,3) {$z' \otimes_s z$};

\begin{knot}[clip width=4]
\strand [thick] (cf) to (trd);
\strand [thick] (true) to [out=70, in=-110] (trd1w);
\strand [thick] (tru1) to (c1);
\strand [thick] (i) to [out=90,in=-90] (lc) to [out=90,in=-90] (rc) to [out=90,in=-90] (o);
\end{knot}
\draw [thick] (truw) to [out=110,in=-70] (trd1e);
\end{tikzpicture}
}
=
\hbox{
\begin{tikzpicture}[baseline=(current  bounding  box.center)]

\node (i) at (-0.75,-0.5) {$a$};
\coordinate (lc) at (-0.75,0.25);
\coordinate (rc) at (0.75,1.1);
\coordinate (o) at (0.75,3);

\node (cf) at (0,-0.5){$z \otimes_s z'$};

\node (tr) at (0,0.1) {$\bigtriangledown$};
\coordinate (trd) at (0,0);
\coordinate (truw) at (-0.1,00.2);
\coordinate (true) at (0.1,0.2);

\node (tr1) at (0,2.2) {$\bigtriangleup$};
\coordinate (trd1w) at (-0.1,2.12);
\coordinate (trd1e) at (0.1,2.12);
\coordinate (tru1) at (0,2.35);
\node (c1) at (0,3) {$z' \otimes_s z$};

\begin{knot}[clip width=4]
\strand [thick] (cf) to (trd);
\strand [thick] (true) to [out=70, in=-110] (trd1w);
\strand [thick] (tru1) to (c1);
\strand [thick] (i) to [out=90,in=-90] (lc) to [out=90,in=-90] (rc) to [out=90,in=-90] (o);
\end{knot}
\draw [thick] (truw) to [out=110,in=-70] (trd1e);
\end{tikzpicture}
}
=
\hbox{
\begin{tikzpicture}[baseline=(current  bounding  box.center)]

\node (i) at (-0.75,-0.5) {$a$};
\coordinate (lc) at (-0.75,0);
\coordinate (rc) at (0.75,0.75);
\coordinate (o) at (0.75,3);

\node (cf) at (0,-0.5){$z \otimes_s z'$};

\node (tr) at (0,1.1) {$\bigtriangledown$};
\coordinate (trd) at (0,1);
\coordinate (truw) at (-0.1,1.2);
\coordinate (true) at (0.1,1.2);

\node (tr1) at (0,2.2) {$\bigtriangleup$};
\coordinate (trd1w) at (-0.1,2.12);
\coordinate (trd1e) at (0.1,2.12);
\coordinate (tru1) at (0,2.35);
\node (c1) at (0,3) {$z' \otimes_s z$};

\begin{knot}[clip width=4]
\strand [thick] (cf) to (trd);
\strand [thick] (true) to [out=70, in=-110] (trd1w);
\strand [thick] (tru1) to (c1);
\strand [thick] (i) to [out=90,in=-90] (lc) to [out=90,in=-90] (rc) to [out=90,in=-90] (o);
\end{knot}
\draw [thick] (truw) to [out=110,in=-70] (trd1e);
\end{tikzpicture}
},
$$
where the first step is slicing, the second step uses the Reidemeister moves for the braiding, and the final step is slicing again. This is what we wanted to show.
\end{proof}

\subsubsection{Associativity}
Before we discuss the associators, it is helpful to examine what at a triple product $(z \otimes_s z')\otimes_s z''$ looks like for $z,z',z''\in \dcentcat{A}$.

\begin{lem}\label{STtripleproduct}
	The triple products $(z \otimes_s z')\otimes_s z''$ and $z\otimes_s (z'\otimes_s z'')$ have as underlying object the subobject associated to the idempotent
	\begin{equation}\label{STtripleidem}
	\hbox{
		\begin{tikzpicture}[baseline=(current  bounding  box.center)]
		\coordinate (west) at (-0.8,0);
		\coordinate (north) at (-0.15,0.2);
		\coordinate (east) at (0.3,0);
		\coordinate (south) at (-0.15,-0.2);
		
		\coordinate (west1) at (-0.3,1);
		\coordinate (north1) at (0.15,1.2);
		\coordinate (east1) at (0.8,1);
		\coordinate (south1) at (0.15,0.8);

		\node (a) at (-0.5,-1.2) {$z$};
		\node (b) at (0,-1.2) {$z'$};
		\node (c) at (0.5,-1.2){$z''$};
		\coordinate (ao) at (-0.5,2);
		\coordinate (bo) at (0,2);
		\coordinate (co) at (0.5,2);
		
		\begin{knot}[clip width=4,clip radius = 3pt]
		\strand [blue, thick] (west)
		to [out=90,in=-180] (north)
		to [out=0,in=90] (east)
		to [out=-90,in=0] (south)
		to [out=-180,in=-90] (west);
		\strand [blue, thick] (west1)
		to [out=90,in=-180] (north1)
		to [out=0,in=90] (east1)
		to [out=-90,in=0] (south1)
		to [out=-180,in=-90] (west1);
		\strand [thick] (a) to (ao);
		\strand [thick] (b) to (bo);
		\strand [thick] (c) to (co);
		\flipcrossings{2,3,6,7}
		\end{knot}
		\end{tikzpicture}
	},
	\end{equation}
	interpreted as endomorphism of $(zz')z''$ and $z(z'z'')$, respectively, using the (suppressed) associators.
\end{lem}

\begin{proof}
	By definition, the underlying object of $(z \otimes_s z')\otimes_s z''$ is the subobject associated to the idempotent
	$$
	
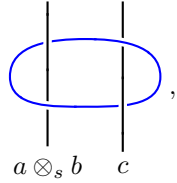
\hbox{
		\begin{tikzpicture}[baseline=(current  bounding  box.center)]
		\coordinate (west) at (-1,0);
		\coordinate (north) at (0,0.5);
		\coordinate (east) at (1,0);
		\coordinate (south) at (0,-0.4);
		\node (a) at (-0.5,-1.2) {$z \otimes_s z'$};
		\node (b) at (0.5,-1.2) {$z''$};
		\coordinate (ao) at (-0.5,1);
		\coordinate (bo) at (0.5,1);
		
		\begin{knot}[clip width=4]
		\strand [blue, thick] (west)
		to [out=90,in=-180] (north)
		to [out=0,in=90] (east)
		to [out=-90,in=0] (south)
		to [out=-180,in=-90] (west);
		\strand [thick] (a) to (ao);
		\strand [thick] (b) to (bo);
		\flipcrossings{1,4}
		\end{knot}
		\end{tikzpicture}
	},
	$$
	where the overcrossing on the strand $z \otimes_s z'$ corresponds to Equation \eqref{SThalfbraid}. Spelling this out, we get:
	
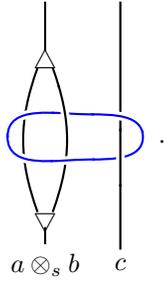
\begin{equation}\label{STtripleidemp}
	\hbox{
		\begin{tikzpicture}[baseline=(current  bounding  box.center)]
		
		\node (ab) at (0,-0.5){$z \otimes_s z'$};

		\node (tr) at (0,0.1) {$\bigtriangledown$};
		\coordinate (trd) at (0,0);
		\coordinate (truw) at (-0.1,00.2);
		\coordinate (true) at (0.1,0.2);
		
		\coordinate (west) at (-0.5,1.2);
		\coordinate (north) at (0.6,1.5);
		\coordinate (east) at (1.3,1.2);
		\coordinate (south) at (0.6,0.9);
		
		\node (tr1) at (0,2.2) {$\bigtriangleup$};
		\coordinate (trd1w) at (-0.1,2.12);
		\coordinate (trd1e) at (0.1,2.12);
		\coordinate (tru1) at (0,2.35);
		\coordinate (c1) at (0,3);
		
		\node (c) at (1,-0.5){$z''$};
		\coordinate (co) at (1,3); 
		
		\begin{knot}[clip width=4]
		\strand [blue, thick] (west)
		to [out=90,in=-180] (north)
		to [out=0,in=90] (east)
		to [out=-90,in=0] (south)
		to [out=-180,in=-90] (west);
		\strand [thick] (ab) to (trd);
		\strand [thick] (true) to [out=70, in=-70] (trd1e);
		\strand [thick] (tru1) to (c1);
		\strand [thick] (c)  to [out=90, in =-90] (co);
		\strand [thick] (truw) to [out=110,in=-110] (trd1w);
		\flipcrossings{4,1}
		\end{knot}
		\end{tikzpicture}
	}=
		\hbox{
		\begin{tikzpicture}[baseline=(current  bounding  box.center)]
		
		\node (ab) at (0,-0.5){$z \otimes_s z'$};

		\node (tr) at (0,0.1) {$\bigtriangledown$};
		\coordinate (trd) at (0,0);
		\coordinate (truw) at (-0.1,00.2);
		\coordinate (true) at (0.1,0.2);
		
		\coordinate (west) at (0,1.2);
		\coordinate (north) at (0.6,1.5);
		\coordinate (east) at (1.3,1.2);
		\coordinate (south) at (0.6,0.9);
		
		\node (tr1) at (0,2.2) {$\bigtriangleup$};
		\coordinate (trd1w) at (-0.1,2.12);
		\coordinate (trd1e) at (0.1,2.12);
		\coordinate (tru1) at (0,2.35);
		\coordinate (c1) at (0,3);
		
		\node (c) at (1,-0.5){$z''$};
		\coordinate (co) at (1,3); 
		
		\begin{knot}[clip width=4]
			\strand [blue, thick] (west)
			to [out=90,in=-180] (north)
			to [out=0,in=90] (east)
			to [out=-90,in=0] (south)
			to [out=-180,in=-90] (west);
			\strand [thick] (ab) to (trd);
			\strand [thick] (true) to [out=70, in=-70] (trd1e);
			\strand [thick] (tru1) to (c1);
			\strand [thick] (c)  to [out=90, in =-90] (co);
			\strand [thick] (truw) to [out=110,in=-110] (trd1w);
			\flipcrossings{4,1}
		\end{knot}
	\end{tikzpicture}
	}
	\end{equation}
	for the idempotent that picks out the underlying object of $(z \otimes_s z')\otimes_s z''$. We now claim that
	$$
	
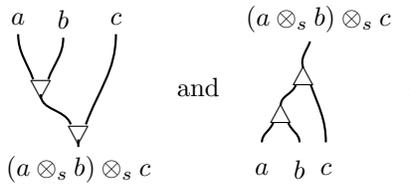
\hbox{
		\begin{tikzpicture}[baseline=(current  bounding  box.center)]

		\node (inc) at (0.5,1){$(z \otimes_s z')\otimes_s z''$};
		
		\node (tr) at (0,2.1) {$\bigtriangledown$};
		\coordinate (trd) at (0,2);
		\coordinate (truw) at (-0.1,2.2);
		\coordinate (true) at (0.1,2.2);

		\node (tr3) at (0.5,1.5) {$\bigtriangledown$};
		\coordinate (trd3) at (0.5,1.4);
		\coordinate (truw3) at (0.4,1.6);
		\coordinate (true3) at (0.6,1.6);

		\node (ao) at (-0.3,3){$z$};
		\node (bo) at (0.3,3){$z'$};
		\node (co) at (1,3){$z''$}; 
		
		\begin{knot}[clip width=4]
		\strand [thick] (truw3) to [out=110,in=-90] (trd);
		\strand [thick] (true) to [out=70, in=-90] (bo);
		\strand [thick] (true3) to [out=70,in=-90] (co);
		\strand [thick] (inc)  to [out=90, in =-90] (trd3);
		\strand [thick] (truw) to [out=110,in=-90] (ao);
		\end{knot}
		\end{tikzpicture}
	}
	\tn{ and }
	\hbox{
		\begin{tikzpicture}[baseline=(current  bounding  box.center)]
		
		\node (a) at (-0.25,0.5)[minimum height=20pt]{$z$};
		\node (b) at (0.25,0.5)[minimum height=20pt]{$z'$};
		\node (c) at (0.6,0.5)[minimum height=20pt]{$z''$};

		\node (tr13) at (0,1.2) {$\bigtriangleup$};
		\coordinate (trd1w3) at (-0.1,1.12);
		\coordinate (trd1e3) at (0.1,1.12);
		\coordinate (tru13) at (0,1.35);
		
		\node (tr1) at (0.3,1.7) {$\bigtriangleup$};
		\coordinate (trd1w) at (0.2,1.62);
		\coordinate (trd1e) at (0.4,1.62);
		\coordinate (tru1) at (0.3,1.85);
		
		\node (outg) at (0.5,2.5){$(z \otimes_s z')\otimes_s z''$};

		\begin{knot}[clip width=4]
		\strand [thick] (a) to [out=90,in=-110] (trd1w3);
		\strand [thick] (b) to [out=90, in=-70] (trd1e3);
		\strand [thick] (tru13) to [out=90,in=-110] (trd1w);
		\strand [thick] (c)  to [out=90, in =-70] (trd1e);
		\strand [thick] (tru1) to [out=110,in=-110] (outg);
		\end{knot}
		\end{tikzpicture}
	},
	$$
	together exhibit $(z \otimes_s z')\otimes_s z''$ as the subobject associated to the idempotent from Equation \eqref{STtripleidem}, as desired. That is, we want to show these morphisms satisfy the properties from Equation \eqref{STinclprojprops}. From the properties of the inclusions and projections involved, we see that the composition along $zz'z''$ indeed is the identity. Composing along $(z \otimes_s z')\otimes_s z''$, we get:
	\begin{equation}\label{STassocomp}
	\hbox{
		\begin{tikzpicture}[baseline=(current  bounding  box.center)]
		
		\node (a) at (-0.25,0.5)[minimum height=15pt]{$z$};
		\node (b) at (0.25,0.5)[minimum height=15pt]{$z'$};
		\node (c) at (1,0.5)[minimum height=15pt]{$z''$};

		\node (tr14) at (0,1.2) {$\bigtriangleup$};
		\coordinate (trd1w4) at (-0.1,1.12);
		\coordinate (trd1e4) at (0.1,1.12);
		\coordinate (tru14) at (0,1.35);
		
		\node (tr15) at (0.3,1.7) {$\bigtriangleup$};
		\coordinate (trd1w5) at (0.2,1.62);
		\coordinate (trd1e5) at (0.4,1.62);
		\coordinate (tru15) at (0.3,1.85);
		
		\node (tr3) at (0.5,2.5) {$\bigtriangledown$};
		\coordinate (trd3) at (0.5,2.4);
		\coordinate (truw3) at (0.4,2.6);
		\coordinate (true3) at (0.6,2.6);
		
		\node (tr) at (0,3.1) {$\bigtriangledown$};
		\coordinate (trd) at (0,3);
		\coordinate (truw) at (-0.1,3.2);
		\coordinate (true) at (0.1,3.2);
		
		\node (ao) at (-0.3,4){$z$};
		\node (bo) at (0.3,4){$z'$};
		\node (co) at (1,4){$z''$};

		\begin{knot}[clip width=4]
		\strand [thick] (a) to [out=90,in=-110] (trd1w4);
		\strand [thick] (b) to [out=90, in=-70] (trd1e4);
		\strand [thick] (tru14) to [out=90,in=-110] (trd1w5);
		\strand [thick] (c)  to [out=90, in =-70] (trd1e5);
		\strand [thick] (tru15) to [out=90,in=-90] (trd3);
		\strand [thick] (truw3) to [out=110,in=-90] (trd);
		\strand [thick] (true) to [out=70, in=-90] (bo);
		\strand [thick] (true3) to [out=70,in=-90] (co);
		\strand [thick] (truw) to [out=110,in=-90] (ao);
		\end{knot}
		\end{tikzpicture}
	}
=
\hbox{
	\begin{tikzpicture}[baseline=(current  bounding  box.center)]
	
	\node (a) at (-0.25,0.5){$z$};
	\node (b) at (0.25,0.5){$z'$};
	\node (c) at (1,0.5){$z''$};

	\node (tr14) at (0,1.2) {$\bigtriangleup$};
	\coordinate (trd1w4) at (-0.1,1.12);
	\coordinate (trd1e4) at (0.1,1.12);
	\coordinate (tru14) at (0,1.35);
	
	\node (tr3) at (0,1.6) {$\bigtriangledown$};
	\coordinate (trd3) at (0,1.5);
	\coordinate (truw3) at (-0.1,1.7);
	\coordinate (true3) at (0.1,1.7);
	
	\coordinate (west) at (-0.1,2.7);
	\coordinate (north) at (0.6,3);
	\coordinate (east) at (1.3,2.7);
	\coordinate (south) at (0.6,2.4);
	
	\node (tr15) at (0,3.7) {$\bigtriangleup$};
	\coordinate (trd1w5) at (-0.1,3.62);
	\coordinate (trd1e5) at (0.1,3.62);
	\coordinate (tru15) at (0,3.85);
	
	\node (tr) at (0,4) {$\bigtriangledown$};
	\coordinate (trd) at (0,3.9);
	\coordinate (truw) at (-0.1,4.1);
	\coordinate (true) at (0.1,4.1);
	
	\node (ao) at (-0.3,4.6){$z$};
	\node (bo) at (0.3,4.6){$z'$};
	\node (co) at (1,4.6){$z''$};

	\begin{knot}[clip width=4]
	\strand [blue, thick] (west) to [out=90,in=-180] (north) to [out=0,in=90] (east) to [out=-90,in=0] (south) to [out=-180,in=-90] (west);
	\strand [thick] (a) to [out=90,in=-110] (trd1w4);
	\strand [thick] (b) to [out=90, in=-70] (trd1e4);
	\strand [thick] (tru14) to [out=90,in=-90] (trd3);
	\strand [thick] (truw3) to [out=110,in=-110] (trd1w5);
	\strand [thick] (true3) to [out=70,in=-70] (trd1e5);
	\strand [thick] (tru15) to [out=90,in=-90] (trd);
	\strand [thick] (c)  to [out=90, in =-90] (co);
	\strand [thick] (true) to [out=70, in=-90] (bo);
	\strand [thick] (truw) to [out=110,in=-90] (ao);
	\flipcrossings{1,4}
	\end{knot}
	\end{tikzpicture}
}
=
\hbox{
	\begin{tikzpicture}[baseline=(current  bounding  box.center)]
	\coordinate (west) at (-0.8,0);
	\coordinate (north) at (-0.15,0.2);
	\coordinate (east) at (0.3,0);
	\coordinate (south) at (-0.15,-0.2);
	
	\coordinate (west1) at (-0.3,1);
	\coordinate (north1) at (0.15,1.2);
	\coordinate (east1) at (0.8,1);
	\coordinate (south1) at (0.15,0.8);
	
	\coordinate (west2) at (-0.8,2);
	\coordinate (north2) at (-0.15,2.2);
	\coordinate (east2) at (0.3,2);
	\coordinate (south2) at (-0.15,1.8);
	
	\node (a) at (-0.5,-1) {$z$};
	\node (b) at (0,-1) {$z'$};
	\node (c) at (0.5,-1){$z''$};
	\coordinate (ao) at (-0.5,2.9);
	\coordinate (bo) at (0,2.9);
	\coordinate (co) at (0.5,2.9);
	
	\begin{knot}[clip width=4, clip radius = 3pt]
	\strand [blue, thick] (west)
	to [out=90,in=-180] (north)
	to [out=0,in=90] (east)
	to [out=-90,in=0] (south)
	to [out=-180,in=-90] (west);
	\strand [blue, thick] (west1)
	to [out=90,in=-180] (north1)
	to [out=0,in=90] (east1)
	to [out=-90,in=0] (south1)
	to [out=-180,in=-90] (west1);
		\strand [blue, thick] (west2)
	to [out=90,in=-180] (north2)
	to [out=0,in=90] (east2)
	to [out=-90,in=0] (south2)
	to [out=-180,in=-90] (west2);
	\strand [thick] (a) to (ao);
	\strand [thick] (b) to (bo);
	\strand [thick] (c) to (co);
	\flipcrossings{2,3,6,7,11,10}
	\end{knot}
	\end{tikzpicture}
}
	=
	\hbox{
		\begin{tikzpicture}[baseline=(current  bounding  box.center)]
		\coordinate (west) at (-0.8,0);
		\coordinate (north) at (-0.15,0.2);
		\coordinate (east) at (0.3,0);
		\coordinate (south) at (-0.15,-0.2);
		
		\coordinate (west1) at (-0.3,1);
		\coordinate (north1) at (0.15,1.2);
		\coordinate (east1) at (0.8,1);
		\coordinate (south1) at (0.15,0.8);

		\node (a) at (-0.5,-1.2) {$z$};
		\node (b) at (0,-1.2) {$z'$};
		\node (c) at (0.5,-1.2){$z''$};
		\coordinate (ao) at (-0.5,2);
		\coordinate (bo) at (0,2);
		\coordinate (co) at (0.5,2);
		
		\begin{knot}[clip width=4, clip radius = 3pt]
		\strand [blue, thick] (west)
		to [out=90,in=-180] (north)
		to [out=0,in=90] (east)
		to [out=-90,in=0] (south)
		to [out=-180,in=-90] (west);
		\strand [blue, thick] (west1)
		to [out=90,in=-180] (north1)
		to [out=0,in=90] (east1)
		to [out=-90,in=0] (south1)
		to [out=-180,in=-90] (west1);
		\strand [thick] (a) to (ao);
		\strand [thick] (b) to (bo);
		\strand [thick] (c) to (co);
		\flipcrossings{2,3,6,7}
		\end{knot}
		\end{tikzpicture}
	},
	\end{equation}
	where the first two steps come from combining inclusion and projections to idempotents, using Equation \eqref{STtripleidemp} in the first. In the last step we used that the rings are transparent to each other, so we can slide the top one down to meet the bottom one, and idempotent, so the two rings straddling the same strands combine to one. The argument showing  $z\otimes_s (z'\otimes_s z'')$ corresponds to the idempotent from Equation \eqref{STtripleidem} is analogous.
\end{proof}

\begin{lem}\label{STasso}
	The associators of $\cat{A}$ induce isomorphisms between $(z \otimes_s z')\otimes_s z''$ and $z \otimes_s (z' \otimes_s z'')$ for all $z,z',z''\in\dcentcat{A}$.
\end{lem}

\begin{proof}
	From Lemma \ref{STtripleproduct}, we know that that the triple products have underlying objects that are the subobjects associated to idempotents that are conjugate to each other along the associators $\alpha:(zz')z''\rar z(z'z'')$. This means we are in the situation of Lemma \ref{STconjugateidems} and the associators will induce isomorphisms between these subobjects. We still have show that these isomorphisms are compatible with the half-braidings, i.e. that the induced morphisms are indeed in $\dcentcat{A}$. To do this, we check that, explicitly inserting the associator $\alpha$ for this proof:
	$$
	\hbox{
		\begin{tikzpicture}[baseline=(current  bounding  box.center)]
		
		\node (tr3) at (0.5,0.7) {$\bigtriangledown$};
		\coordinate (trd3) at (0.5,0.6);
		\coordinate (truw3) at (0.4,0.8);
		\coordinate (true3) at (0.6,0.8);
		
		\node (tr) at (0,1.3) {$\bigtriangledown$};
		\coordinate (trd) at (0,1.2);
		\coordinate (truw) at (-0.1,1.4);
		\coordinate (true) at (0.1,1.4);
		
		\node (tr14) at (0.3,2.5) {$\bigtriangleup$};
		\coordinate (trd1w4) at (0.2,2.42);
		\coordinate (trd1e4) at (0.4,2.42);
		\coordinate (tru14) at (0.3,2.65);
		
		\node (tr15) at (0,3) {$\bigtriangleup$};
		\coordinate (trd1w5) at (-0.1,2.92);
		\coordinate (trd1e5) at (0.1,2.92);
		\coordinate (tru15) at (0,3.15);
		
		\node (alpha) at (0.15,2) [draw, minimum width=1cm]{$\alpha$};
		
		\coordinate (outg) at (0,4);
		\coordinate (inc) at (0.5,-0.5);
		
		\coordinate (w) at (0,-0.5);
		\coordinate (rc) at (1,0.6);
		\coordinate (wo) at (0.5,4);
		
		\begin{knot}[clip width=4]
		\strand [thick] (w) to [out=90,in=-90] (rc) to [out=90,in=-90] (wo);
		\strand [thick] (truw) to [out=110,in=-90] (alpha.-150);
		\strand [thick] (alpha.150) to [out=90,in=-110] (trd1w5);
		\strand [thick] (true) to [out=70,in=-90] (alpha);
		\strand [thick] (alpha) to [out=90, in=-110] (trd1w4);
		\strand [thick] (tru14) to [out=90,in=-70] (trd1e5);
		\strand [thick] (true3) to [out=70, in=-90] (alpha.-30);
		\strand [thick] (alpha.30) to [out=90,in =-70] (trd1e4);
		\strand [thick] (tru15) to [out=90,in=-90] (outg);
		\strand [thick] (truw3) to [out=110,in=-90] (trd);
		\strand [thick] (inc) to [out=90,in=-90] (trd3);
		\flipcrossings{1}
		\end{knot}
		\end{tikzpicture}
	}
	=
	\hbox{
	\begin{tikzpicture}[baseline=(current  bounding  box.center)]
	
	\node (tr3) at (0.5,0) {$\bigtriangledown$};
	\coordinate (trd3) at (0.5,-0.1);
	\coordinate (truw3) at (0.4,0.1);
	\coordinate (true3) at (0.6,0.1);
	
	\node (tr) at (0,1.3) {$\bigtriangledown$};
	\coordinate (trd) at (0,1.2);
	\coordinate (truw) at (-0.1,1.4);
	\coordinate (true) at (0.1,1.4);
	
	\node (tr14) at (0.3,2.5) {$\bigtriangleup$};
	\coordinate (trd1w4) at (0.2,2.42);
	\coordinate (trd1e4) at (0.4,2.42);
	\coordinate (tru14) at (0.3,2.65);
	
	\node (tr15) at (0,3) {$\bigtriangleup$};
	\coordinate (trd1w5) at (-0.1,2.92);
	\coordinate (trd1e5) at (0.1,2.92);
	\coordinate (tru15) at (0,3.15);
	
	\node (alpha) at (0.15,2) [draw, minimum width=1cm]{$\alpha$};
	
	\coordinate (outg) at (0,4);
	\coordinate (inc) at (0.5,-0.5);
	
	\coordinate (w) at (0,-0.5);
	\coordinate (lc) at (0,0.2);
	\coordinate (rc) at (0.8,1.8);
	\coordinate (wo) at (0.5,4);
	
	\begin{knot}[clip width=4]
	\strand [thick] (w) to [out=90,in=-90] (lc) to [out=90,in=-90] (rc) to [out=90,in=-90] (wo);
	\strand [thick] (truw) to [out=110,in=-90] (alpha.-150);
	\strand [thick] (alpha.150) to [out=90,in=-110] (trd1w5);
	\strand [thick] (true) to [out=70,in=-90] (alpha);
	\strand [thick] (alpha) to [out=90, in=-110] (trd1w4);
	\strand [thick] (tru14) to [out=90,in=-70] (trd1e5);
	\strand [thick] (true3) to [out=70, in=-90] (alpha.-30);
	\strand [thick] (alpha.30) to [out=90,in =-70] (trd1e4);
	\strand [thick] (tru15) to [out=90,in=-90] (outg);
	\strand [thick] (truw3) to [out=110,in=-90] (trd);
	\strand [thick] (inc) to [out=90,in=-90] (trd3);
	\flipcrossings{2}
	\end{knot}
	\end{tikzpicture}
}
	=
	\hbox{
	\begin{tikzpicture}[baseline=(current  bounding  box.center)]
	
	\node (tr3) at (0.5,0) {$\bigtriangledown$};
	\coordinate (trd3) at (0.5,-0.1);
	\coordinate (truw3) at (0.4,0.1);
	\coordinate (true3) at (0.6,0.1);
	
	\node (tr) at (0,0.5) {$\bigtriangledown$};
	\coordinate (trd) at (0,0.4);
	\coordinate (truw) at (-0.1,0.6);
	\coordinate (true) at (0.1,0.6);
	
	\node (tr14) at (0.3,2.5) {$\bigtriangleup$};
	\coordinate (trd1w4) at (0.2,2.42);
	\coordinate (trd1e4) at (0.4,2.42);
	\coordinate (tru14) at (0.3,2.65);
	
	\node (tr15) at (0,3) {$\bigtriangleup$};
	\coordinate (trd1w5) at (-0.1,2.92);
	\coordinate (trd1e5) at (0.1,2.92);
	\coordinate (tru15) at (0,3.15);
	
	\node (alpha) at (0.15,2) [draw, minimum width=1cm]{$\alpha$};
	
	\coordinate (outg) at (0,4);
	\coordinate (inc) at (0.5,-0.5);
	
	\coordinate (w) at (0,-0.5);
	\coordinate (lc) at (-0.5,0.6);
	\coordinate (rc) at (0.8,1.7);
	\coordinate (wo) at (0.5,4);
	
	\begin{knot}[clip width=4,clip radius=3pt]
	\strand [thick] (w) to [out=90,in=-90] (lc) to [out=90,in=-90] (rc) to [out=90,in=-90] (wo);
	\strand [thick] (truw) to [out=110,in=-90] (alpha.-150);
	\strand [thick] (alpha.150) to [out=90,in=-110] (trd1w5);
	\strand [thick] (true) to [out=70,in=-90] (alpha);
	\strand [thick] (alpha) to [out=90, in=-110] (trd1w4);
	\strand [thick] (tru14) to [out=90,in=-70] (trd1e5);
	\strand [thick] (true3) to [out=70, in=-90] (alpha.-30);
	\strand [thick] (alpha.30) to [out=90,in =-70] (trd1e4);
	\strand [thick] (tru15) to [out=90,in=-90] (outg);
	\strand [thick] (truw3) to [out=110,in=-90] (trd);
	\strand [thick] (inc) to [out=90,in=-90] (trd3);
	\flipcrossings{2}
	\end{knot}
	\end{tikzpicture}
}
=
	\hbox{
	\begin{tikzpicture}[baseline=(current  bounding  box.center)]
	
	\node (tr3) at (0.5,0) {$\bigtriangledown$};
	\coordinate (trd3) at (0.5,-0.1);
	\coordinate (truw3) at (0.4,0.1);
	\coordinate (true3) at (0.6,0.1);
	
	\node (tr) at (0,0.5) {$\bigtriangledown$};
	\coordinate (trd) at (0,0.4);
	\coordinate (truw) at (-0.1,0.6);
	\coordinate (true) at (0.1,0.6);
	
	\node (tr14) at (0.3,2.5) {$\bigtriangleup$};
	\coordinate (trd1w4) at (0.2,2.42);
	\coordinate (trd1e4) at (0.4,2.42);
	\coordinate (tru14) at (0.3,2.65);
	
	\node (tr15) at (0,3) {$\bigtriangleup$};
	\coordinate (trd1w5) at (-0.1,2.92);
	\coordinate (trd1e5) at (0.1,2.92);
	\coordinate (tru15) at (0,3.15);
	
	\node (alpha) at (0.15,1.1) [draw, minimum width=1cm]{$\alpha$};
	
	\coordinate (outg) at (0,4);
	\coordinate (inc) at (0.5,-0.5);
	
	\coordinate (w) at (0,-0.5);
	\coordinate (lc) at (-0.5,1.5);
	\coordinate (rc) at (0.8,2.3);
	\coordinate (wo) at (0.5,4);
	
	\begin{knot}[clip width=4,clip radius=3pt]
	\strand [thick] (w) to [out=90,in=-90] (lc) to [out=90,in=-90] (rc) to [out=90,in=-90] (wo);
	\strand [thick] (truw) to [out=110,in=-90] (alpha.-150);
	\strand [thick] (alpha.150) to [out=90,in=-110] (trd1w5);
	\strand [thick] (true) to [out=70,in=-90] (alpha);
	\strand [thick] (alpha) to [out=90, in=-110] (trd1w4);
	\strand [thick] (tru14) to [out=90,in=-70] (trd1e5);
	\strand [thick] (true3) to [out=70, in=-90] (alpha.-30);
	\strand [thick] (alpha.30) to [out=90,in =-70] (trd1e4);
	\strand [thick] (tru15) to [out=90,in=-90] (outg);
	\strand [thick] (truw3) to [out=110,in=-90] (trd);
	\strand [thick] (inc) to [out=90,in=-90] (trd3);
	\flipcrossings{2}
	\end{knot}
	\end{tikzpicture}
}
	=
	\hbox{
	\begin{tikzpicture}[baseline=(current  bounding  box.center)]
	
	\node (tr3) at (0.5,0) {$\bigtriangledown$};
	\coordinate (trd3) at (0.5,-0.1);
	\coordinate (truw3) at (0.4,0.1);
	\coordinate (true3) at (0.6,0.1);
	
	\node (tr) at (0,0.5) {$\bigtriangledown$};
	\coordinate (trd) at (0,0.4);
	\coordinate (truw) at (-0.1,0.6);
	\coordinate (true) at (0.1,0.6);
	
	\node (tr14) at (0.3,2) {$\bigtriangleup$};
	\coordinate (trd1w4) at (0.2,1.92);
	\coordinate (trd1e4) at (0.4,1.92);
	\coordinate (tru14) at (0.3,2.15);
	
	\node (tr15) at (0,2.4) {$\bigtriangleup$};
	\coordinate (trd1w5) at (-0.1,2.32);
	\coordinate (trd1e5) at (0.1,2.32);
	\coordinate (tru15) at (0,2.55);
	
	\node (alpha) at (0.15,1.5) [draw, minimum width=1cm]{$\alpha$};
	
	\coordinate (outg) at (0,4);
	\coordinate (inc) at (0.5,-0.5);
	
	\coordinate (w) at (0,-0.5);
	\coordinate (lc) at (-0.5,2.2);
	\coordinate (wo) at (0.5,4);
	
	\begin{knot}[clip width=4,clip radius=3pt]
	\strand [thick] (w) to [out=90,in=-90] (lc) to [out=90,in=-90] (wo);
	\strand [thick] (truw) to [out=110,in=-90] (alpha.-150);
	\strand [thick] (alpha.150) to [out=90,in=-110] (trd1w5);
	\strand [thick] (true) to [out=70,in=-90] (alpha);
	\strand [thick] (alpha) to [out=90, in=-110] (trd1w4);
	\strand [thick] (tru14) to [out=90,in=-70] (trd1e5);
	\strand [thick] (true3) to [out=70, in=-90] (alpha.-30);
	\strand [thick] (alpha.30) to [out=90,in =-70] (trd1e4);
	\strand [thick] (tru15) to [out=90,in=-90] (outg);
	\strand [thick] (truw3) to [out=110,in=-90] (trd);
	\strand [thick] (inc) to [out=90,in=-90] (trd3);
	\flipcrossings{1}
	\end{knot}
	\end{tikzpicture}
},
	$$
	where we made repeated use of slicing (Lemma \ref{STslicing}) to bring the crossing strand up. To pass the braiding past the associator, we have used the naturality of the braiding. 
\end{proof}

\subsubsection{Unit}
In order for $\otimes_s$ to define a monoidal structure on $\dcentcat{A}$, we need to provide a monoidal unit. This unit needs to come with unitors, natural isomorphisms witnessing that tensoring with this object is the identity. For more details on monoidal categories, see for example \cite{Joyal1986}.

\begin{df}\label{STunitdef}
The \emph{symmetric unit} $\mathbb{I}_s\dcentcat{A}$ is the object $\sum_{i\in \cat{O}(\cat{A})} ii^*$, equipped with the half braiding:
\begin{equation}\label{STunithalfbraid}
\hbox{
	\begin{tikzpicture}[baseline=(current  bounding  box.center)]
	\node (i) at (0,0) {$\mathbb{I}_s$};
	\coordinate (ip) at (-0.1,0.27);
	\coordinate (o) at (0,3);
	\coordinate (op) at (-0.1,3);
	
	\node (a) at (-0.5,0){$a$};
	\coordinate (ao) at (0.5,3);

	\begin{knot}[clip width=4]
	\strand [thick,blue] (i) to (o);
	\strand [thick,blue] (ip) to (op);
	\strand [thick] (a) to [out=90,in=-90] (ao);
	\end{knot}
	
	\end{tikzpicture}
}
:=
\sum_{i,j\in\cat{O}(\cat{A})}\sum_{\phi\in B(ai,j)}
\hbox{
\begin{tikzpicture}[baseline=(current  bounding  box.center)]
\node (i) at (0,0) {$i$};
\node (o) at (0,3){$j$};

\node (id) at (0.5,0){$i^*$};
\node (do) at (0.5,3){$j^*$};

\node (a) at (-0.5,0){$a$};
\node (phi) at (0,1) [draw]{$\phi$};

\node (phid) at (0.5,2)[draw]{$\phi^*$};
\node (ao) at (1,3){$a$};

\begin{knot}[clip width=4]
\strand [thick] (i) to (phi);
\strand [thick] (phi) to (o);
\strand [thick] (a) to [out=90,in=-110] (phi);
\strand [thick] (id) to (phid);
\strand [thick] (phid) to (do);
\strand [thick] (phid) to [out=-60,in=-90] (ao);
\end{knot}

\end{tikzpicture}
}.
\end{equation}
The double strand will henceforth be used to denote the identity on $\mathbb{I}_s$. In the above formula $\phi^*$ denotes 
$$
\hbox{
\begin{tikzpicture}[baseline=(current  bounding  box.center)]

\node (a) at (-0.25,-3) {$i^*$};
\node (i) at (0.25,-3) {$a^*$};

\node (phi) at (0,-1) [draw,thick]{$\phi^*$};

\node (k) at (0,1) {$j^*$};

\begin{knot}[clip width=4]
\strand [thick] (a) to [out=90,in=-110] (phi);
\strand [thick] (phi) to [out=90,in=-90] (k);
\strand [thick] (i) to  [out=90, in =-70] (phi);
\end{knot}
\end{tikzpicture}
}
:=
\hbox{
	\begin{tikzpicture}[baseline=(current  bounding  box.center)]

	\node (i) at (-0.25,-3) {$i^*$};
	\node (a) at (0.25,-3) {$a^*$};
	
	\coordinate (ac) at (0.9,-0.6);
	\coordinate (ic) at (0.6,-0.8);
	
	\node (phi) at (0,-1) [draw,thick]{$\phi^t$};
	
	\node (k) at (0,1) {$j^*$};
	\coordinate (kc) at (-0.6,-1.3);
	
	\begin{knot}[clip width=4]
	\strand [thick] (a) to [out=90,in=-90] (ac) to [out=90,in=110] (phi);
	\strand [thick] (phi) to [out=-90,in=-90] (kc) to [out=90,in=-90] (k);
	\strand [thick] (i) to [in=-90,out=90] (ic) to [out=90, in =70] (phi);
	\end{knot}
	\end{tikzpicture}
},
$$
and $\phi^t$ was introduced in Definition \ref{STtransposedef}.
\end{df}

We will show that this object acts as the monoidal unit for the symmetric tensor product. The left unitor is built from evaluation morphisms
\begin{equation}\label{STleftunitor}
\hbox{
	\begin{tikzpicture}[baseline=(current  bounding  box.center)]
	\node (inc) at (0,0.5){$\mathbb{I}_s\otimes_s z$};
	\node (bo) at (0,3){$z$};
	\node (tr) at (0,1.1) {$\bigtriangledown$};
	\coordinate (trd) at (0,1);
	\coordinate (truw) at (-0.15,1.2);
	\coordinate (true) at (0.1,1.2);
	\coordinate (trum) at (-0.1,1.2);
	
	\coordinate (ctrl) at (0.175,2.3);
	
	\begin{knot}[clip width=4,clip radius=3pt]
	\strand [thick] (inc) to (trd);
	\strand [thick,blue] (trum) to [out=120, in=0] (ctrl) to [out=180,in=120] (truw);
	\strand [thick] (true) to [out=60,in =-90] (bo);
	\flipcrossings{2}
	\end{knot}
	\end{tikzpicture}
}
:=
\sum\limits_{i\in\cat{O}(\cat{A})}
\hbox{
	\begin{tikzpicture}[baseline=(current  bounding  box.center)]
	\node (inc) at (0,0.5){$\mathbb{I}_s\otimes_s z$};
	\node (bo) at (0,3){$z$};
	\node (il) at (-0.35,1.4){$i$};
	\node (tr) at (0,1.1) {$\bigtriangledown$};
	\coordinate (trd) at (0,1);
	\coordinate (truw) at (-0.15,1.2);
	\coordinate (true) at (0.1,1.2);
	\coordinate (trum) at (-0.1,1.2);
	
	\coordinate (ctrl) at (0.175,2.3);
	
	\begin{knot}[clip width=4,clip radius = 3pt]
	\strand [thick] (inc) to (trd);
	\strand [thick] (trum) to [out=115, in=0] (ctrl) to [out=180,in=120] (truw);
	\strand [thick] (true) to [out=60,in =-90] (bo);
	\flipcrossings{2}
	\end{knot}
	\end{tikzpicture}
},
\end{equation}
where the double strand coming out of the inclusion on the left hand side denotes the identity on the object $\mathbb{I}_s$ (c.f. the convention made above). The right unitor is obtained by reflecting the above diagram in a vertical line.
\begin{equation}\label{STrightunitor}
\hbox{
	\begin{tikzpicture}[baseline=(current  bounding  box.center)]
	\node (inc) at (0,0.5){$ z\otimes_s\mathbb{I}_s$};
	\node (bo) at (0,3){$z$};
	\node (tr) at (0,1.1) {$\bigtriangledown$};
	\coordinate (trd) at (0,1);
	\coordinate (truw) at (0.15,1.2);
	\coordinate (true) at (-0.1,1.2);
	\coordinate (trum) at (0.1,1.2);
	
	\coordinate (ctrl) at (-0.175,2.3);
	
	\begin{knot}[clip width=4,clip radius=3pt]
	\strand [thick] (inc) to (trd);
	\strand [thick,blue] (trum) to [out=60, in=180] (ctrl) to [out=0,in=60] (truw);
	\strand [thick] (true) to [out=120,in =-90] (bo);
	\flipcrossings{2}
	\end{knot}
	\end{tikzpicture}
}
:=
\sum\limits_{i\in\cat{O}(\cat{A})}
\hbox{
	\begin{tikzpicture}[baseline=(current  bounding  box.center)]
	\node (inc) at (0,0.5){$z\otimes_s \mathbb{I}_s$};
	\node (bo) at (0,3){$z$};
	\node (il) at (0.35,1.4){$i$};
	\node (tr) at (0,1.1) {$\bigtriangledown$};
	\coordinate (trd) at (0,1);
	\coordinate (truw) at (0.15,1.2);
	\coordinate (true) at (-0.1,1.2);
	\coordinate (trum) at (0.1,1.2);

	\coordinate (ctrl) at (-0.175,2.3);

	\begin{knot}[clip width=4,clip radius=3pt]
	\strand [thick] (inc) to (trd);
	\strand [thick] (trum) to [out=60, in=180] (ctrl) to [out=0,in=60] (truw);
	\strand [thick] (true) to [out=120,in =-90] (bo);
	\flipcrossings{2}
	\end{knot}
	\end{tikzpicture}
}.
\end{equation}

We claim, and prove below in Lemma \ref{STunitisos}, that the left unitor has an inverse given by:
\begin{equation}\label{STunitorinverses}
\hbox{
	\begin{tikzpicture}[baseline=(current  bounding  box.center)]
	\node (inc) at (0,2){$\mathbb{I}_s\otimes_s b$};
	\node (bo) at (0,-0.7){$b$};

	\node (tr1n) at (0,1.2) {$\bigtriangleup$};
	\coordinate (trd) at (0,1.35);
	\coordinate (truw) at (-0.15,1.12);
	\coordinate (true) at (0.1,1.12);
	\coordinate (trum) at (-0.1,1.12);
	
	\coordinate (ctrl) at (0.4,0.0);
	
	\begin{knot}[clip width=4, clip radius = 3pt]
	\strand [thick] (inc) to (trd);
	\strand [thick,blue] (trum) to [out=-120, in=0] (ctrl) to [out=180,in=-120] (truw);
	\strand [thick] (true) to [out=-55,in =90] (bo);
	\flipcrossings{2}
	\end{knot}
	\end{tikzpicture}
}
:=
\sum\limits_{i\in\cat{O}(\cat{A})}\frac{d_i}{D}
\hbox{
	\begin{tikzpicture}[baseline=(current  bounding  box.center)]
	\node (inc) at (0,2){$\mathbb{I}_s\otimes_s b$};
	\node (bo) at (0,-0.7){$b$};
	
	\node (il) at (-0.35,0.9){$i$};
	
	\node (tr1n) at (0,1.2) {$\bigtriangleup$};
	\coordinate (trd) at (0,1.35);
	\coordinate (truw) at (-0.15,1.12);
	\coordinate (true) at (0.1,1.12);
	\coordinate (trum) at (-0.1,1.12);
	
	\coordinate (ctrl) at (0.4,0.0);
	
	\begin{knot}[clip width=4,clip radius = 3pt]
	\strand [thick] (inc) to (trd);
	\strand [thick] (trum) to [out=-115, in=0] (ctrl) to [out=180,in=-120] (truw);
	\strand [thick] (true) to [out=-55,in =90] (bo);
	\flipcrossings{2}
	\end{knot}
	\end{tikzpicture}
},
\end{equation}
and the inverse for the right unitor is correspondingly given by reflecting the above diagram in a vertical line. To prove these statements, and to show that this indeed gives the monoidal unit, we will make use of the following property we will refer to as \emph{snapping}:

\begin{lem}[Snapping]\label{STsnapping}
	For any $c\in\dcentcat{A}$ we have:
	$$
	\hbox{
		\begin{tikzpicture}[baseline=(current  bounding  box.center)]
		\node (i) at (0,0) {$\mathbb{I}_s$};
		\coordinate (ip) at (-0.1,0.27);
		\coordinate (o) at (0,3);
		\coordinate (op) at (-0.1,3);
		
		\node (a) at (0.3,0){$z''$};
		\coordinate (ao) at (0.3,3);
		
		\coordinate (west) at (-0.4,1.6);
		\coordinate (north) at (0.2,1.8);
		\coordinate (east) at (0.5,1.6);
		\coordinate (south) at (0.2,1.4);

		\begin{knot}[clip width=4,clip radius =3pt]
		\strand [thick,blue] (i) to (o);
		\strand [thick,blue] (ip) to (op);
		\strand [thick] (a) to [out=90,in=-90] (ao);
		\strand [blue, thick] (west) to [out=90,in=-180] (north) to [out=0,in=90] (east) to [out=-90,in=0] (south) to [out=-180,in=-90] (west);
		\flipcrossings{2,4,5}
		\end{knot}
		\end{tikzpicture}
	}
	=
	\hbox{
		\begin{tikzpicture}[baseline=(current  bounding  box.center)]
		\node (i) at (0,0) {$\mathbb{I}_s$};
		\coordinate (ip) at (-0.1,0.27);
		\coordinate (o) at (0,3);
		\coordinate (op) at (-0.1,3);
		
		\node (a) at (0.3,0){$z''$};
		\coordinate (ao) at (0.3,3);
		
		\coordinate (uc) at (0.5,2);
		\coordinate (lc) at (0.5,1.2);
		
		\begin{knot}[clip width=4,clip radius=4pt]
		\strand [thick,blue] (i) to [out=90,in=-90] (lc) to [out=90,in=90] (ip);
		\strand [thick,blue] (o) to [out=-90,in=90] (uc) to [out=-90,in=-90] (op);
		\strand [thick] (a) to [out=90,in=-90] (ao);
		\flipcrossings{2,4,5}
		\end{knot}
		\end{tikzpicture}
	}.
	$$
\end{lem}

\begin{proof}
	Unpacking the definition of the half-braiding on $\mathbb{I}_s$, we get:
	\begin{equation}
	\hbox{
		\begin{tikzpicture}[baseline=(current  bounding  box.center)]
		\coordinate (truw) at (-0.1,0.2);
		\coordinate (true) at (0.3,0.2);
		\coordinate (trum) at (0,0.2);
		
		\coordinate (truw1) at (-0.1,3.12);
		\coordinate (true1) at (0.3,3.12);
		\coordinate (trum1) at (0,3.12);

		\coordinate (west) at (-0.4,1.6);
		\coordinate (north) at (0.2,1.8);
		\coordinate (east) at (0.6,1.6);
		\coordinate (south) at (0.2,1.4);
		
		\begin{knot}[clip width=4,clip radius=2pt]
		\strand [blue, thick] (west) to [out=90,in=-180] (north) to [out=0,in=90] (east) to [out=-90,in=0] (south) to [out=-180,in=-90] (west);
		\strand [thick,blue] (trum) to [out=90,in=-90] (trum1);
		\strand [thick,blue] (truw) to [out=90, in=-90] (truw1);
		\strand [thick] (true) to [out=90,in =-90] (true1);
		\flipcrossings{1,3,6}
		\end{knot}
		\end{tikzpicture}
	}
	=\sum\limits_{i,j,k\in\cat{O}(\cat{A})}\sum\limits_{\phi\in B(ki,j)}\frac{d_k}{D}
	\hbox{
		\begin{tikzpicture}[baseline=(current  bounding  box.center)]

		\node (il) at (-0.4,0.5){$i$};
		\node (jl) at (-0.6,2.3){$j$};
		\node (kl) at (-0.8,1.4){$k$};

		\coordinate (truw) at (-0.2,0.2);
		\coordinate (true) at (0.5,0.2);
		\coordinate (trum) at (-0.1,0.2);
		
		\coordinate (truw1) at (-0.2,3.3);
		\coordinate (true1) at (0.5,3.3);
		\coordinate (trum1) at (-0.1,3.3);
		
		\node (phi) at (-0.5,1.75)[draw]{$\phi$};
		\node (phid) at (0,2.35)[draw]{$\phi^*$};

		\coordinate (south) at (0, 1);
		\coordinate (east) at (1,1.3);
		
		\begin{knot}[clip width=4]
		\strand [thick] (phi) to [out=-110,in=-90] (east);
		\strand [thick] (east) to [out=90,in=-70] (phid);
		\strand [thick] (trum) to [out=90,in=-100] (phid);
		\strand [thick] (truw) to [out=90, in=-70] (phi);
		\strand [thick] (phi) to [out=90,in=-90] (truw1);
		\strand [thick] (phid) to [out=90,in=-90] (trum1);
		\strand [thick] (true) to [out=90,in =-90] (true1);
		\flipcrossings{3}
		\end{knot}
		\end{tikzpicture}
	}.
	\end{equation}
	
	We can manipulate the summands on the right hand side, using Equation \eqref{STtransparancy} to bring the middle incoming strand to the front, and then pushing $\phi^*$ to the right passing over the straight strand, to get:
	$$
	\hbox{
		\begin{tikzpicture}[baseline=(current  bounding  box.center)]
		\node (il) at (-0.4,0.5){$i$};
		\node (jl) at (-0.6,2.3){$j$};
		\node (kl) at (-0.8,1.35){$k$};

		\coordinate (truw) at (-0.1,0);
		\coordinate (true) at (0.3,0);
		\coordinate (trum) at (0,0);
		
		\coordinate (truw1) at (-0.1,3.12);
		\coordinate (true1) at (0.3,3.12);
		\coordinate (trum1) at (0,3.12);
		
		\node (phi) at (-0.5,1.75)[draw]{$\phi$};
		\node (phid) at (1,1.5)[draw]{$\phi^*$};

		\coordinate (south) at (0, 0.9);
		\coordinate (east) at (1,1.3);
		
		\begin{knot}[clip width=4,clip radius=3pt]
		\strand [thick] (phi) to [out=-110,in=-70] (phid);
		\strand [thick] (trum) to [out=90,in=-100] (phid);
		\strand [thick] (truw) to [out=90, in=-70] (phi);
		\strand [thick] (phi) to [out=90,in=-90] (truw1);
		\strand [thick] (phid) to [out=90,in=-90] (trum1);
		\strand [thick] (true) to [out=90,in =-90] (true1);
		\flipcrossings{1,3}
		\end{knot}
		\end{tikzpicture}
	}
	=
	\hbox{
		\begin{tikzpicture}[baseline=(current  bounding  box.center)]
		\node (il) at (-0.35,0.2){$i$};
		\node (jl) at (-0.3,3.6){$j$};
		\node (kl) at (-0.35,1.6){$k$};
		
		\coordinate (truw) at (-0.1,0);
		\coordinate (true) at (0.5,0);
		\coordinate (trum) at (0,0);
		
		\coordinate (truw1) at (-0.1,3.8);
		\coordinate (true1) at (0.5,3.8);
		\coordinate (trum1) at (0,3.8);
		
		\node (phi) at (-0.25,3.1)[draw]{$\phi$};
		\node (phid) at (-1,1.5)[draw]{$\phi^*$};
		
		\coordinate (o) at (0,4);
		
		\coordinate (east) at (0.9,1.1);
		\coordinate (lc) at (0.9,2.5);
		\coordinate (rc) at (-0.5,1.3);
		
		\begin{knot}[clip width=4]
		\strand [thick] (phi) to [out=-110,in=90] (rc) to [out=-90,in=-70] (phid);
		\strand [thick] (trum) to [out=90,in=-90] (east) to [out=90,in=-100] (phid);
		\strand [thick] (truw) to [out=90, in=-70] (phi);
		\strand [thick] (phi) to [out=90,in=-90] (truw1);
		\strand [thick] (phid) to [out=90,in=-40] (lc) to [out=140,in=-90] (trum1);
		\strand [thick] (true) to [out=90,in =-90] (true1);
		\flipcrossings{1,5,6,4}
		\end{knot}
		\fill (rc) circle[radius=2pt];
		\end{tikzpicture}
	}
	=
	\hbox{
		\begin{tikzpicture}[baseline=(current  bounding  box.center)]
		\node (il) at (-0.35,0){$i$};
		\node (jl) at (-0.3,3.6){$j$};
		\node (kl) at (-0.3,1.8){$k$};
		
		\coordinate (truw) at (-0.1,-0.3);
		\coordinate (true) at (0.4,-0.3);
		\coordinate (trum) at (0,-0.3);
		
		\coordinate (truw1) at (-0.1,3.62);
		\coordinate (true1) at (0.4,3.62);
		\coordinate (trum1) at (0,3.62);
		
		\node (phi) at (-0.25,2.3)[draw]{$\phi$};
		\node (phid) at (-0.9,0.9)[draw]{$\phi^t$};

		\coordinate (east) at (0.9,0.9);
		\coordinate (lc) at (0.9,2.5);
		\coordinate (rc) at (-0.5,1.3);
		\coordinate (llc) at (-1.4,0.8);
		\coordinate (luc) at (-0.6,3);
		
		\coordinate (north) at (-0.1,3.4);
		\coordinate (below) at (-0.1,0);
		
		\begin{knot}[clip width=4,clip radius=3pt]
		\strand [thick] (phi) to [out=-110,in=90] (phid);
		\strand [thick] (trum) to [out=90,in=-90] (east) to [out=90,in=60] (phid);
		\strand [thick] (truw) to [out=90,in=-90] (below) to [out=90, in=-70] (phi);
		\strand [thick] (phi) to [out=90,in=-90] (north) to [out=90,in=-90] (truw1);
		\strand [thick] (phid) to [out=-90,in=-90] (llc) to [out=90,in=-120] (luc) to [out=60,in=-90] (lc) to [out=90,in=-90] (trum1);
		\strand [thick] (true) to [out=90,in =-90] (true1);
		\flipcrossings{3,5,4}
		\end{knot}
		\fill (north) circle[radius=2pt];
		\fill (below) circle[radius=2pt];
		\end{tikzpicture}
		}.
	$$
To get the first equality, we pushed $\phi^*$ all the way to the left, passing behind the straight strand and in front of the $k$ strand. A self-intersection gave a twist (see Equation \eqref{STbulletdef}) on the $k$ strand while doing this. The third equality plugs in the definition of $\phi^*$, and the equality:
	\begin{equation}\label{STtwisttrick}
		\hbox{
		\begin{tikzpicture}[baseline=(current  bounding  box.center)]
		
		\node (phi) at (0,1)[draw]{$\phi$};
		
		\coordinate (wc) at (-0.2,0.5);
		\coordinate (ec) at (0.2,0.5);
		
		\coordinate (nc) at (0,1.5);
		
		\node (k) at (-0.3,0){$k$};
		\node (i) at (0.3,0){$i$};
		
		\node (j) at (0,2){$j$};
		
		\begin{knot}[clip width=4,clip radius=3pt]
		\strand [thick] (k) to [out=90,in=-110] (wc) to [out=70,in=-110] (phi);
		\strand [thick] (i) to [out=90,in=-70] (phi);
		\strand [thick] (phi) to [out=90,in=-90] (nc) to [out=90,in=-90] (j);
		\end{knot}
		\fill (wc) circle[radius=2pt];
		\end{tikzpicture}
	}
	=
			\hbox{
		\begin{tikzpicture}[baseline=(current  bounding  box.center)]
		
		\node (phi) at (0,1)[draw]{$\phi$};
		
		\coordinate (wc) at (-0.2,0.5);
		\coordinate (ec) at (0.2,0.5);
		
		\coordinate (nc) at (0,1.5);
		
		\node (k) at (-0.3,0){$k$};
		\node (i) at (0.3,0){$i$};
		
		\node (j) at (0,2){$j$};
		
		\begin{knot}[clip width=4,clip radius=3pt]
		\strand [thick] (k) to [out=90,in=-110] (phi);
		\strand [thick] (i) to [out=90,in=-70] (ec) to [out=110,in=-70] (phi);
		\strand [thick] (phi) to [out=90,in=-90] (nc) to [out=90,in=-90] (j);
		\end{knot}
		\fill (ec) circle[radius=2pt];
		\fill (nc) circle[radius=2pt];
		\end{tikzpicture}
	}.
	\end{equation}
	This equality follows from the naturality of the twist, together with the fact that in a symmetric fusion category the twist is a monoidal automorphism of the identity functor that squares to $1$. We can now  perform the sum over $\phi$ and $k$ using Lemma \ref{STdirectsumversion} to obtain:
	$$
	\hbox{
		\begin{tikzpicture}[baseline=(current  bounding  box.center)]
		\coordinate (truw) at (-0.1,0.2);
		\coordinate (true) at (0.3,0.2);
		\coordinate (trum) at (0,0.2);
		
		\coordinate (truw1) at (-0.1,3.12);
		\coordinate (true1) at (0.3,3.12);
		\coordinate (trum1) at (0,3.12);

		\coordinate (west) at (-0.4,1.6);
		\coordinate (north) at (0.2,1.8);
		\coordinate (east) at (0.6,1.6);
		\coordinate (south) at (0.2,1.4);
		
		\begin{knot}[clip width=4,clip radius=2pt]
		\strand [blue, thick] (west) to [out=90,in=-180] (north) to [out=0,in=90] (east) to [out=-90,in=0] (south) to [out=-180,in=-90] (west);
		\strand [thick,blue] (trum) to [out=90,in=-90] (trum1);
		\strand [thick,blue] (truw) to [out=90, in=-90] (truw1);
		\strand [thick] (true) to [out=90,in =-90] (true1);
		\flipcrossings{1,3,6}
		\end{knot}
		\end{tikzpicture}
	}
	=
	\sum\limits_{i,j\in\cat{O}(\cat{A})}\frac{d_j}{D}
	\hbox{
		\begin{tikzpicture}[baseline=(current  bounding  box.center)]
		\node (il) at (-0.35,0){$i$};
		\node (jl) at (-0.3,3.6){$j$};

		\coordinate (truw) at (-0.1,-0.3);
		\coordinate (true) at (0.4,-0.3);
		\coordinate (trum) at (0,-0.3);
		
		\coordinate (truw1) at (-0.1,3.62);
		\coordinate (true1) at (0.4,3.62);
		\coordinate (trum1) at (0,3.62);
		
		\coordinate (phi) at (-0.25,2.1);
		\coordinate (phid) at (-0.5,0.9);
				
		\coordinate (east) at (0.9,0.9);
		\coordinate (lc) at (0.9,2.5);
		\coordinate (rc) at (0,1.3);
		\coordinate (llc) at (-1,2);
		\coordinate (luc) at (-0.6,3);
		
		\coordinate (north) at (-0.1,3.4);
		\coordinate (below) at (-0.1,0);

		\begin{knot}[clip width=4,clip radius=3pt]
		\strand [thick] (trum) to [out=90,in=-90] (east) to [out=90,in=-90] (phid);
		\strand [thick] (truw) to [out=90, in=-90] (below) to [out=90, in=-90] (rc) to [out=90, in=70] (phid);
		\strand [thick] (truw1) to [out=-90,in=90] (north) to [out=-90,in=70] (phi);
		\strand [thick] (phi) to [out=-110,in=-90] (llc) to [out=90,in=-120] (luc) to [out=60,in=-90] (lc) to [out=90,in=-90] (trum1);
		\strand [thick] (true) to [out=90,in =-90] (true1);
		\flipcrossings{3,5,,4}
		\end{knot}
		\fill (north) circle[radius=2pt];
		\fill (below) circle[radius=2pt];
		\end{tikzpicture}
	}
	=\sum\limits_{i,j\in\cat{O}(\cat{A})}\frac{d_j}{D}
	\hbox{
		\begin{tikzpicture}[baseline=(current  bounding  box.center)]
		\coordinate (i) at (0,0);
		\coordinate (ip) at (-0.1,0);
		\coordinate (o) at (0,3);
		\coordinate (op) at (-0.1,3);
		
		\coordinate (a) at (0.3,0);
		\coordinate (ao) at (0.3,3);
		
		\coordinate (uc) at (0.5,2);
		\coordinate (lc) at (0.5,1.2);
		
		\coordinate (ka) at (-0.1,0);

		\node (il) at (-0.2,0.4){$i$};
		\node (jl) at (-0.2,2.7){$j$};

		\begin{knot}[clip width=4,clip radius=4pt]
		\strand [thick] (i) to [out=90,in=-90] (lc) to [out=90,in=90] (ip);
		\strand [thick] (o) to [out=-90,in=90] (uc) to [out=-90,in=-90] (op);
		\strand [thick] (a) to [out=90,in=-90] (ao);
		\flipcrossings{2,4,5}
		\end{knot}
		\end{tikzpicture}
	}
	=
	\hbox{
	\begin{tikzpicture}[baseline=(current  bounding  box.center)]
	\coordinate (i) at (0,0);
	\coordinate (ip) at (-0.1,0);
	\coordinate (o) at (0,3);
	\coordinate (op) at (-0.1,3);
	
	\coordinate (a) at (0.3,0);
	\coordinate (ao) at (0.3,3);
	
	\coordinate (uc) at (0.5,2);
	\coordinate (lc) at (0.5,1.2);
	
	\coordinate (ka) at (-0.1,0);
	\begin{knot}[clip width=4,clip radius=4pt]
	\strand [thick,blue] (i) to [out=90,in=-90] (lc) to [out=90,in=90] (ip);
	\strand [thick,blue] (o) to [out=-90,in=90] (uc) to [out=-90,in=-90] (op);
	\strand [thick] (a) to [out=90,in=-90] (ao);
	\flipcrossings{2,4,5}
	\end{knot}
	\end{tikzpicture}
	},
	$$
	where in the second equality we cancelled twists with self-intersections.
\end{proof}

The object $\mathbb{I}_s$ does indeed act as the unit for the symmetric tensor product on $\dcentcat{A}$:
\begin{lem}\label{STunitisos}
The symmetric tensor product $\mathbb{I}_s\otimes_s z$ of $\mathbb{I}_s$ with any object $z\in\dcentcat{A}$ is isomorphic to $z$ along the morphism given in Equation \eqref{STleftunitor}. Similarly, $z\otimes_s\mathbb{I}_s\cong z$ along the morphism given in Equation \eqref{STrightunitor}.
\end{lem}
\begin{proof}
We first prove that the morphisms from Equations \eqref{STleftunitor} and \eqref{STunitorinverses} are inverse to each other. We then establish they are morphisms in $\dcentcat{A}$, a priori they might not commute with the chosen half-braidings. Composing the morphisms along $\mathbb{I}_s\otimes_s z$, we see we need to check that:
\begin{equation}\label{STringcomesoff}
\hbox{
\begin{tikzpicture}[baseline=(current  bounding  box.center)]

\node (bo) at (0,3){$z$};

\node (tr) at (0,1.1) {$\bigtriangledown$};
\coordinate (trd) at (0,1);
\coordinate (truw) at (-0.1,1.2);
\coordinate (true) at (0.1,1.2);
\coordinate (trum) at (0,1.2);

\coordinate (ctrl) at (0.175,2.3);

\node (b) at (0,-1.7){$z$};

\node (tr1n) at (0,0.2) {$\bigtriangleup$};
\coordinate (trd1) at (0,0.35);
\coordinate (truw1) at (-0.1,0.12);
\coordinate (true1) at (0.1,0.12);
\coordinate (trum1) at (0,0.12);

\coordinate (ctrl1) at (0.3,-0.9);

\begin{knot}[clip width=4,clip radius = 3pt]
\strand [thick,blue] (trum) to [out=90, in=0] (ctrl) to [out=180,in=120] (truw);
\strand [thick] (true) to [out=60,in =-90] (bo);
\strand [thick] (trd1) to (trd);
\strand [thick,blue] (trum1) to [out=-90, in=0] (ctrl1) to [out=180,in=-120] (truw1);
\strand [thick] (true1) to [out=-60,in =90] (b);
\flipcrossings{2,4}
\end{knot}
\end{tikzpicture}
}
=
\hbox{
	\begin{tikzpicture}[baseline=(current  bounding  box.center)]
	
	\node (bo) at (0,3){$z$};
	
	\coordinate (truw) at (-0.1,1.2);
	\coordinate (trum) at (0,1.2);
	
	\coordinate (ctrl) at (0.175,2.3);

	\node (b) at (0,-1.7){$z$};
	
	\coordinate (west) at (-0.4,0.6);
	\coordinate (north) at (0.2,0.8);
	\coordinate (east) at (0.6,0.6);
	\coordinate (south) at (0.2,0.4);

	\coordinate (truw1) at (-0.1,0.12);;
	\coordinate (trum1) at (0,0.12);
	
	\coordinate (ctrl1) at (.3,-0.9);
	
	\begin{knot}[clip width=4,clip radius = 3pt]
	\strand [thick,blue] (trum) to [out=90, in=0] (ctrl) to [out=180,in=90] (truw) to [out=-90,in=90] (truw1) to [out=-90, in=180] (ctrl1) to [out=0,in=-90] (trum1) to [out=90,in=-90] (trum);
	\strand [thick] (b) to [out=80,in =-80] (bo);
	\strand [blue, thick] (west) to [out=90,in=-180] (north) to [out=0,in=90] (east) to [out=-90,in=0] (south) to [out=-180,in=-90] (west);
	\flipcrossings{2,3,6,7,8}
	\end{knot}
	\end{tikzpicture}
}
=
\hbox{
	\begin{tikzpicture}[baseline=(current  bounding  box.center)]
	
	\node (bo) at (0,3){$z$};
	
	\coordinate (truw) at (-0.1,1.4);
	\coordinate (trum) at (0,1.4);
	
	\coordinate (ctrl) at (0.175,2.3);

	\node (b) at (0,-1.7){$z$};
	
	\coordinate (north) at (0.5,0.7);
	\coordinate (south) at (0.5,0.5);

	\coordinate (truw1) at (-0.1,-0.10);;
	\coordinate (trum1) at (0,-0.10);
	
	\coordinate (ctrl1) at (.3,-0.9);
	
	\begin{knot}[clip width=4,clip radius = 3pt]
	\strand [thick,blue] (trum) to [out=90, in=0] (ctrl) to [out=180,in=90] (truw) to [out=-90,in=180] (north) to [out=0,in=-90] (trum); 
	\strand [thick,blue] (truw1) to [out=-90, in=180] (ctrl1) to [out=0,in=-90] (trum1) to [out=90,in=0] (south) to [out=180,in=90] (truw1);
	\strand [thick] (b) to [out=80,in =-80] (bo);
	\flipcrossings{2,3,5,8}
	\end{knot}
	\end{tikzpicture}
}
=
\hbox{
	\begin{tikzpicture}[baseline=(current  bounding  box.center)]
	
	\node (bo) at (0,3){$z$};

	\node (b) at (0,-1.7){$z$};
	
	\begin{knot}[clip width=4,clip radius = 3pt]
	\strand [thick] (b) to [out=80,in =-80] (bo);
	\end{knot}
	\end{tikzpicture}
},
\end{equation}
where we used that projection follow by inclusion gives the idempotent, followed by snapping (Lemma \ref{STsnapping}), and in the last step the rings come off and evaluate to $1$. For the other composition, along $z$, note that, using Equation \eqref{STinclprojprops}:
$$
\hbox{
\begin{tikzpicture}[baseline=(current  bounding  box.center)]
\node (ib) at (0,0){$\mathbb{I}_s\otimes_s z$};
\coordinate (ibo) at (0,2);

\begin{knot}[clip width=4]
\strand [thick] (ib) to (ibo);
\end{knot}
\end{tikzpicture}
}
=
\hbox{
	\begin{tikzpicture}[baseline=(current  bounding  box.center)]
	\node (inc) at (0,-0.5){$\mathbb{I}_s\otimes_s z$};
	\node (bl) at (0.3,2.6){$z$};
	
	\node (tr) at (0,0.1) {$\bigtriangledown$};
	\coordinate (trd) at (0,0);
	\coordinate (truw) at (-0.1,0.2);
	\coordinate (true) at (0.1,0.2);
	\coordinate (trum) at (0,0.2);
	
	\node (tr1n) at (0,3.2) {$\bigtriangleup$};
	\coordinate (trd1) at (0,3.35);
	\coordinate (truw1) at (-0.1,3.12);
	\coordinate (true1) at (0.1,3.12);
	\coordinate (trum1) at (0,3.12);
	
	\coordinate (o) at (0,3.5);
	
	\coordinate (west) at (-0.4,1.6);
	\coordinate (north) at (0.2,1.8);
	\coordinate (east) at (0.6,1.6);
	\coordinate (south) at (0.2,1.4);
	
	\begin{knot}[clip width=4,clip radius=2pt]
	\strand [blue, thick] (west) to [out=90,in=-180] (north) to [out=0,in=90] (east) to [out=-90,in=0] (south) to [out=-180,in=-90] (west);
	\strand [thick] (inc) to (trd);
	\strand [thick,blue] (trum) to [out=90,in=-90] (trum1);
	\strand [thick,blue] (truw) to [out=90, in=-90] (truw1);
	\strand [thick] (true) to [out=80,in =-80] (true1);
	\strand [thick] (trd1) to (o);
	\flipcrossings{1,3,6}
	\end{knot}
	\end{tikzpicture}
}
=
\hbox{
	\begin{tikzpicture}[baseline=(current  bounding  box.center)]
	\node (inc) at (0,-1){$\mathbb{I}_s\otimes_s z$};
	\node (bl) at (0.3,3.5){$z$};
	
	\node (tr) at (0,-0.4) {$\bigtriangledown$};
	\coordinate (trd) at (0,0.-0.5);
	\coordinate (truw) at (-0.1,-0.3);
	\coordinate (true) at (0.1,-0.3);
	\coordinate (trum) at (0,-0.3);
	
	\node (tr1n) at (0,3.7) {$\bigtriangleup$};
	\coordinate (trd1) at (0,3.85);
	\coordinate (truw1) at (-0.1,3.62);
	\coordinate (true1) at (0.1,3.62);
	\coordinate (trum1) at (0,3.62);

	\coordinate (o) at (0,4);
	
	\coordinate (east) at (0.9,0.9);
	\coordinate (lc) at (0.9,2.5);

	\begin{knot}[clip width=4,clip radius=3pt]
	\strand [thick] (inc) to (trd);
	\strand [thick,blue] (trum) to [out=90,in=-90] (east) to [out=90,in=120] (truw);
	\strand [thick,blue] (truw1) to [out=-120,in=-90] (lc) to [out=90,in=-90] (trum1);
	\strand [thick] (true) to [out=60,in =-60] (true1);
	\strand [thick] (trd1) to (o);
	\flipcrossings{2,3}
	\end{knot}
	\end{tikzpicture}
},
$$
using snapping in the last step. 

To see that the morphisms are indeed morphisms in $\dcentcat{A}$, we check that, pushing the crossing strand (which represents an object of $\cat{A}$) further and further up:
$$
\hbox{
	\begin{tikzpicture}[baseline=(current  bounding  box.center)]
	\node (inc) at (0,-0.5){$\mathbb{I}_s\otimes_s z$};
	\node (bo) at (0,3.4){$z$};
	\node (tr) at (0,1.1) {$\bigtriangledown$};
	\coordinate (trd) at (0,1);
	\coordinate (truw) at (-0.1,1.2);
	\coordinate (true) at (0.1,1.2);
	\coordinate (trum) at (0,1.2);
	
	\coordinate (ctrl) at (0.175,2.3);
	
	\coordinate (tra) at (-1,-0.5);
	\coordinate (la) at (0.5,1);
	\coordinate (lala) at (0.5,3.2);
	
	\begin{knot}[clip width=4,clip radius = 3pt]
	\strand [thick] (inc) to (trd);
	\strand [thick] (trum) to [out=115, in=0] (ctrl) to [out=180,in=120] (truw);
	\strand [thick] (true) to [out=60,in =-90] (bo);
	\strand [thick] (tra) to [out=90,in=-90] (la) to [out=90,in=-90] (lala);
	\flipcrossings{3}
	\end{knot}
	\end{tikzpicture}
}
=
\hbox{
	\begin{tikzpicture}[baseline=(current  bounding  box.center)]
	\node (inc) at (0,0.5){$\mathbb{I}_s\otimes_s z$};
	\node (bo) at (0,4.4){$z$};
	\node (tr) at (0,1.1) {$\bigtriangledown$};
	\coordinate (trd) at (0,1);
	\coordinate (truw) at (-0.1,1.2);
	\coordinate (true) at (0.1,1.2);
	\coordinate (trum) at (0,1.2);
	
	\coordinate (ctrl) at (0.175,3);
	
	\coordinate (tra) at (-1,0.5);
	\coordinate (lal) at (-1,0.7);
	\coordinate (la) at (0.5,2.2);
	\coordinate (lala) at (0.5,4.2);
	
	\begin{knot}[clip width=4,clip radius = 3pt]
	\strand [thick] (inc) to (trd);
	\strand [thick] (trum) to [out=115, in=0] (ctrl) to [out=180,in=120] (truw);
	\strand [thick] (true) to [out=60,in =-90] (bo);
	\strand [thick] (tra) to [out=90,in=-90] (lal) to [out=90,in=-90] (la) to [out=90,in=-90] (lala);
	\flipcrossings{2}
	\end{knot}
	\end{tikzpicture}
}
=
\hbox{
	\begin{tikzpicture}[baseline=(current  bounding  box.center)]
	\node (inc) at (0,-0.5){$\mathbb{I}_s\otimes_s z$};
	\node (bo) at (0,3.6){$z$};
	\node (tr) at (0,1.1) {$\bigtriangledown$};
	\coordinate (trd) at (0,1);
	\coordinate (truw) at (-0.1,1.2);
	\coordinate (true) at (0.1,1.2);
	\coordinate (trum) at (0,1.2);
	
	\coordinate (ctrl) at (0.175,2.3);
	
	\coordinate (tra) at (-1,-0.5);
	\coordinate (lal) at (-1,2);
	\coordinate (lala) at (0.5,3.4);
	
	\begin{knot}[clip width=4,clip radius = 3pt]
	\strand [thick] (inc) to (trd);
	\strand [thick] (trum) to [out=115, in=0] (ctrl) to [out=180,in=120] (truw);
	\strand [thick] (true) to [out=60,in =-90] (bo);
	\strand [thick] (tra) to [out=90,in=-90] (lal) to [out=90,in=-90] (lala);
	\flipcrossings{2}
	\end{knot}
	\end{tikzpicture}
},
$$
where the first step is slicing (Lemma \ref{STslicing}), and the second step uses Equation \eqref{STcrossinginteraction}.
The proof that $z\otimes_s\mathbb{I}_s\cong z$ along the specified isomorphisms is analogous.
\end{proof}

For $\mathbb{I}_s$ to be a unit for the symmetric tensor product, the isomorphisms from Lemma \ref{STunitisos} need to satisfy the triangle equality \cite[Equation (IT)]{Joyal1986}, that is:
$$
\begin{tikzcd}
(z\otimes_s \mathbb{I}_s)\otimes_s z' \arrow[rr]\arrow[rd]&			& z\otimes_s(\mathbb{I}_s\otimes_s z')\arrow[ld]\\
& z \otimes_s z'&.
\end{tikzcd}
$$
commutes for all $z,z'\in\dcentcat{A}$, where the downwards maps are the unitor isomorphisms (Equations \eqref{STleftunitor} and \eqref{STrightunitor}) and the top is the associator.

\begin{lem}\label{STtriangleeq}
The isomorphisms from Lemma \ref{STunitisos} satisfy the triangle equality.
\end{lem}
\begin{proof}
	We will show that the clockwise composite $z \otimes_s z' \rar (a\otimes_s \mathbb{I}_s)\otimes_s b\rar a\otimes_s(\mathbb{I}_s\otimes_s b)\rar a\otimes_sb$ is the identity on $z \otimes_s z'$. That is, we are considering the composite of
	$$
	
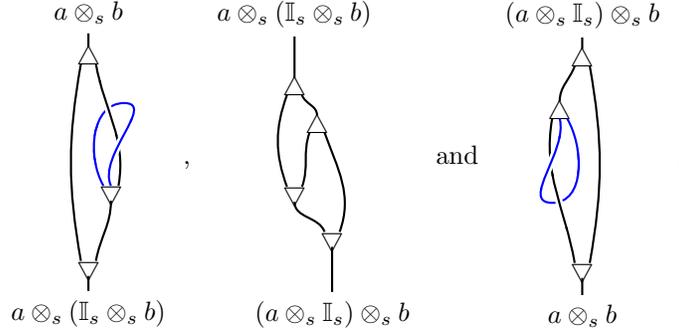
\hbox{
		\begin{tikzpicture}[baseline=(current  bounding  box.center)]
		
		\node (cf) at (0,-0.5){$a\otimes_s(\mathbb{I}_s\otimes_s b)$};
		\node (tr) at (0,0.1) {$\bigtriangledown$};
		\coordinate (trd) at (0,0);
		\coordinate (truw) at (-0.1,00.2);
		\coordinate (true) at (0.1,0.2);

		\node (tr1a) at (0,2.9) {$\bigtriangleup$};
		\coordinate (trd1wa) at (-0.1,2.82);
		\coordinate (trd1ea) at (0.1,2.82);
		\coordinate (tru1a) at (0,3.05);
		\node (c1) at (0,3.5) {$z \otimes_s z'$};
		
		\node (trb) at (0.3,1.1) {$\bigtriangledown$};
		\coordinate (trdb) at (0.3,1);
		\coordinate (truwb) at (0.2,1.2);
		\coordinate (trueb) at (0.4,1.2);
		\coordinate (trumb) at (0.3,1.2);
			
		\coordinate (ctrl) at (0.475,2.3);

		\begin{knot}[clip width=3.5]
		\strand [thick,blue] (trumb) to [out=115, in=0] (ctrl) to [out=180,in=120] (truwb);
		\strand [thick] (cf) to (trd);
		\strand [thick] (truw) to [out=100,in=-100] (trd1wa);
		\strand [thick] (true) to [out=80,in=-90] (trdb);
		\strand [thick] (trueb) to [out=80,in=-80] (trd1ea);
		\strand [thick] (tru1a) to (c1);
		\flipcrossings{2,3}
		\end{knot}
		\end{tikzpicture}
	}
\tn{, }
		\hbox{
		\begin{tikzpicture}[baseline=(current  bounding  box.center)]
		
		\node (tr3) at (0.5,1) {$\bigtriangledown$};
		\coordinate (trd3) at (0.5,0.9);
		\coordinate (truw3) at (0.4,1.1);
		\coordinate (true3) at (0.6,1.1);
		
		\node (tr) at (0,1.6) {$\bigtriangledown$};
		\coordinate (trd) at (0,1.5);
		\coordinate (truw) at (-0.1,1.7);
		\coordinate (true) at (0.1,1.7);
		
		\node (tr14) at (0.3,2.5) {$\bigtriangleup$};
		\coordinate (trd1w4) at (0.2,2.42);
		\coordinate (trd1e4) at (0.4,2.42);
		\coordinate (tru14) at (0.3,2.65);
		
		\node (tr15) at (0,3) {$\bigtriangleup$};
		\coordinate (trd1w5) at (-0.1,2.92);
		\coordinate (trd1e5) at (0.1,2.92);
		\coordinate (tru15) at (0,3.15);

		\node (outg) at (0,4){$a\otimes_s(\mathbb{I}_s\otimes_sb)$};
		\node (inc) at (0.5,0){$(a\otimes_s\mathbb{I}_s)\otimes_sb$};

		\begin{knot}[clip width=4]
		\strand [thick] (truw) to [out=110,in=-110] (trd1w5);
		\strand [thick] (true) to [out=70, in=-110] (trd1w4);
		\strand [thick] (tru14) to [out=90,in=-70] (trd1e5);
		\strand [thick] (true3) to [out=70, in =-70] (trd1e4);
		\strand [thick] (tru15) to [out=90,in=-90] (outg);
		\strand [thick] (truw3) to [out=110,in=-90] (trd);
		\strand [thick] (inc) to [out=90,in=-90] (trd3);
		\end{knot}
		\end{tikzpicture}
	}
\tn{ and }
	\hbox{
		\begin{tikzpicture}[baseline=(current  bounding  box.center)]
		
		\node (cf) at (0,-0.5){$z \otimes_s z'$};
		\node (tr) at (0,0.1) {$\bigtriangledown$};
		\coordinate (trd) at (0,0);
		\coordinate (truw) at (-0.1,00.2);
		\coordinate (true) at (0.1,0.2);

		\node (tr1a) at (0,2.9) {$\bigtriangleup$};
		\coordinate (trd1wa) at (-0.1,2.82);
		\coordinate (trd1ea) at (0.1,2.82);
		\coordinate (tru1a) at (0,3.05);
		\node (c1) at (0,3.5) {$(a\otimes_s\mathbb{I}_s)\otimes_s b$};

		\node (tr1nb) at (-0.3,2.2) {$\bigtriangleup$};
		\coordinate (trdb) at (-0.3,2.35);
		\coordinate (truwb) at (-0.4,2.12);
		\coordinate (trueb) at (-0.2,2.12);
		\coordinate (trumb) at (-0.3,2.12);
			
		\coordinate (ctrl) at (-0.4,1.0);

		\begin{knot}[clip width=3.5]
		\strand [thick,blue] (trumb) to [out=-75, in=180] (ctrl) to [out=0,in=-60] (trueb);
		\strand [thick] (cf) to (trd);
		\strand [thick] (truw) to [out=100,in=-100] (truwb);
		\strand [thick] (true) to [out=80,in=-80] (trd1ea);
		\strand [thick] (trdb) to [out=90,in=-110] (trd1wa);
		\strand [thick] (tru1a) to (c1);
		\flipcrossings{2,3}
		\end{knot}
		\end{tikzpicture}
	}.
	$$
	When composing along the triple symmetric tensor products, we encounter Equation \eqref{STassocomp} and its mirror image. Plugging this in right away and remembering the rings are idempotent, we get
	$$
	
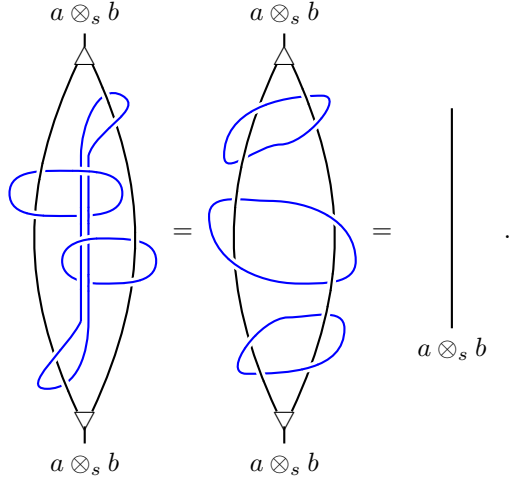
\hbox{
		\begin{tikzpicture}[baseline=(current  bounding  box.center)]
		
		\node (cf) at (0,-0.5){$z \otimes_s z'$};
		\node (tr) at (0,0.1) {$\bigtriangledown$};
		\coordinate (trd) at (0,0);
		\coordinate (truw) at (-0.1,00.2);
		\coordinate (true) at (0.1,0.2);

		\node (tr1a) at (0,4.9) {$\bigtriangleup$};
		\coordinate (trd1wa) at (-0.1,4.82);
		\coordinate (trd1ea) at (0.1,4.82);
		\coordinate (tru1a) at (0,5.05);
		\node (c1) at (0,5.5) {$z \otimes_s z'$};

		\coordinate (trueb) at (0.05,1.4);
		\coordinate (trumb) at (-0.05,1.4);
		\coordinate (truea) at (0.05,3.6);
		\coordinate (truma) at (-0.05,3.6);

		\coordinate (ctrl) at (-0.5,0.5);
		\coordinate (rtc) at (0.5,4.4);

		\coordinate (west) at (-1,3.1);
		\coordinate (north) at (-0.2,3.4);
		\coordinate (east) at (0.5,3.1);
		\coordinate (south) at (-0.2, 2.8);
		
		\coordinate (west1) at (-0.3,2.2);
		\coordinate (north1) at (0.2,2.5);
		\coordinate (east1) at (0.95,2.2);
		\coordinate (south1) at (0.2, 1.9);

		\begin{knot}[clip width=4,clip radius = 3pt]
		\strand [blue, thick] (west) to [out=90,in=180] (north) to [out=0,in=90] (east) to [out=-90,in=0] (south) to [out=180,in=-90] (west);
		\strand [blue, thick] (east1) to [out=-90,in=0] (south1) to [out=180,in=-90](west1) to [out=90,in=180] (north1) to [out=0,in=90]  (east1);
		\strand [thick,blue] (trumb) to [out=-115, in=180] (ctrl) to [out=0,in=-90] (trueb) to [out=90,in=-90] (truea) to [out=75,in=-30] (rtc) to [out=150,in=90] (truma) to [out=-90,in=90] (trumb);
		\strand [thick] (cf) to (trd);
		\strand [thick] (truw) to [out=115,in=-115] (trd1wa);
		\strand [thick] (true) to [out=65,in=-65] (trd1ea);
		\strand [thick] (tru1a) to (c1);
		\flipcrossings{16,17,8,7,4,3,5,12,14}
		\end{knot}
		\end{tikzpicture}
	}
	=
\hbox{
	\begin{tikzpicture}[baseline=(current  bounding  box.center)]
	
	\node (cf) at (0,-0.5){$z \otimes_s z'$};
	\node (tr) at (0,0.1) {$\bigtriangledown$};
	\coordinate (trd) at (0,0);
	\coordinate (truw) at (-0.1,00.2);
	\coordinate (true) at (0.1,0.2);

	\node (tr1a) at (0,4.9) {$\bigtriangleup$};
	\coordinate (trd1wa) at (-0.1,4.82);
	\coordinate (trd1ea) at (0.1,4.82);
	\coordinate (tru1a) at (0,5.05);
	\node (c1) at (0,5.5) {$z \otimes_s z'$};

	\coordinate (trueb) at (0.8,1.3);
	\coordinate (trumb) at (0.05,1.45);
	
	\coordinate (truea) at (-0.05,3.75);
	\coordinate (truma) at (-0.8,3.6);

	\coordinate (ctrl) at (-0.5,0.7);
	\coordinate (rtc) at (0.5,4.4);

	\coordinate (west) at (-1,2.8);
	\coordinate (north) at (-0.2,3);
	\coordinate (east) at (0.95,2.2);
	\coordinate (south) at (0.2, 1.9);

	\begin{knot}[clip width=4,clip radius = 3pt]
	\strand [blue, thick] (west) to [out=90,in=180] (north) to [out=0,in=90] (east) to [out=-90,in=0] (south) to [out=180,in=-90] (west);
	\strand [thick,blue] (trumb) to [out=180, in=180] (ctrl) to [out=0,in=-90] (trueb) to [out=90,in=0] (trumb);
	\strand [thick,blue] (truea) to [out=0,in=0] (rtc) to [out=180,in=90] (truma) to [out=-90,in=180] (truea);
	\strand [thick] (cf) to (trd);
	\strand [thick] (truw) to [out=115,in=-115] (trd1wa);
	\strand [thick] (true) to [out=65,in=-65] (trd1ea);
	\strand [thick] (tru1a) to (c1);
	\flipcrossings{6,8,4,1,12,10}
	\end{knot}
	\end{tikzpicture}
}
	=
	\hbox{
		\begin{tikzpicture}[baseline=(current  bounding  box.center)]

		\node (inc) at (-0.5,-0.4){$z \otimes_s z'$};

		\coordinate (outg) at (-0.5,2.8);
		
		\begin{knot}[clip width=4]
		\strand [thick] (inc)  to [out=90, in =-90] (outg);
		\end{knot}
		\end{tikzpicture}
	}
	.
	$$
	Here the first equality is an application of snapping to the two horizontal rings (Lemma \ref{STsnapping}), the second uses the fact that the rings cancel with the inclusion and projection morphisms.
\end{proof}

\subsection{The Symmetric Tensor Product as a Functor}
We have so far given objectwise definitions of the ingredients needed to define the symmetric tensor product. In this section we will combine these definitions to make the symmetric tensor product into a monoidal structure. The final ingredient needed is a definition of the symmetric tensor product on morphisms.

\subsubsection{Definition on Morphisms}
\begin{df}\label{STsymtensorfunctor}
The \emph{symmetric tensor product}
$$
\otimes_s: \dcentcat{A}\boxtimes\dcentcat{A}\rar \dcentcat{A}
$$
is the assignment defined on objects in Definition \ref{STsymtensobj}. On morphisms $f:z\rar z'$ and $g:y\rar y'$, it is given by 
\begin{equation}\label{STsymtensmorph}
\hbox{
\begin{tikzpicture}[baseline=(current  bounding  box.center)]

\node (cf) at (-1.1,-0.5){$z$};
\node (inc) at (0.3,-0.5){$y$};

\node (f) at (-1.1,1.5)[draw,minimum height=15pt]{$f$};
\node (ot) at (-0.375,1.4){$\mathop{\otimes}\limits_\Vect$};
\node (g) at (0.3,1.5)[draw,,minimum height=15pt]{$g$};

\node (c1) at (-1.1,3) {$z'$};

\node (outg) at (0.3,3){$y'$};

\begin{knot}[clip width=4]
\strand [thick] (cf) to (f) to (c1);
\strand [thick] (inc) to (g) to  (outg);
\end{knot}
\end{tikzpicture}
}
\mapsto
\hbox{
\begin{tikzpicture}[baseline=(current  bounding  box.center)]

\node (cf) at (0,-0.5){$z \otimes_s y$};

\node (f) at (-0.4,1.2)[draw]{$f$};
\node (g) at (0.4,1.2)[draw]{$g$};

\node (tr) at (0,0.1) {$\bigtriangledown$};
\coordinate (trd) at (0,0);
\coordinate (truw) at (-0.1,00.2);
\coordinate (true) at (0.1,0.2);

\node (tr1) at (0,2.2) {$\bigtriangleup$};
\coordinate (trd1w) at (-0.1,2.12);
\coordinate (trd1e) at (0.1,2.12);
\coordinate (tru1) at (0,2.35);
\node (c1) at (0,3) {$z'\otimes_s y'$};

\begin{knot}[clip width=4]
\strand [thick] (cf) to (trd);
\strand [thick] (true) to [out=70, in=-90] (g);
\strand [thick] (g) to [out=90,in=-70] (trd1e);
\strand [thick] (tru1) to (c1);
\strand [thick] (truw) to [out=110,in=-90] (f);
\strand [thick] (f) to [out=90,in=-110] (trd1w);
\end{knot}
\end{tikzpicture}
}.
\end{equation}
\end{df}

\begin{lem}\label{STsymisfunct}
	The assignment from Definition \ref{STsymtensorfunctor} is functorial.
\end{lem}

\begin{proof}
Observe that we have for $f,f'$ and $g,g'$ morphisms in $\dcentcat{A}$:
$$
\hbox{
\begin{tikzpicture}[baseline=(current  bounding  box.center)]

\node (cf) at (0,-0.5){};

\node (f) at (-0.4,1.2)[draw,minimum height=20pt]{$f$};
\node (g) at (0.4,1.2)[draw,minimum height=20pt]{$g$};

\node (tr) at (0,0.1) {$\bigtriangledown$};
\coordinate (trd) at (0,0);
\coordinate (truw) at (-0.1,00.2);
\coordinate (true) at (0.1,0.2);

\node (tr1) at (0,2.2) {$\bigtriangleup$};
\coordinate (trd1w) at (-0.1,2.12);
\coordinate (trd1e) at (0.1,2.12);
\coordinate (tru1) at (0,2.35);

\node (f3) at (-0.4,4.2)[draw,minimum height=20pt]{$f'$};
\node (g3) at (0.4,4.2)[draw,minimum height=20pt]{$g'$};

\node (tr3) at (0,3.1) {$\bigtriangledown$};
\coordinate (trd3) at (0,3);
\coordinate (truw3) at (-0.1,3.2);
\coordinate (true3) at (0.1,3.2);

\node (tr13) at (0,5.2) {$\bigtriangleup$};
\coordinate (trd1w3) at (-0.1,5.12);
\coordinate (trd1e3) at (0.1,5.12);
\coordinate (tru13) at (0,5.35);
\node (c1) at (0,6) {};

\begin{knot}[clip width=4]
\strand [thick] (cf) to (trd);
\strand [thick] (true) to [out=70, in=-90] (g);
\strand [thick] (g) to [out=90,in=-70] (trd1e);
\strand [thick] (tru1) to (trd3);
\strand [thick] (truw) to [out=110,in=-90] (f);
\strand [thick] (f) to [out=90,in=-110] (trd1w);
\strand [thick] (true3) to [out=70, in=-90] (g3);
\strand [thick] (g3) to [out=90,in=-70] (trd1e3);
\strand [thick] (tru13) to (c1);
\strand [thick] (truw3) to [out=110,in=-90] (f3);
\strand [thick] (f3) to [out=90,in=-110] (trd1w3);
\end{knot}
\end{tikzpicture}
}
=
\hbox{
\begin{tikzpicture}[baseline=(current  bounding  box.center)]

\node (cf) at (0,-0.5){};

\node (f) at (-0.4,1.2)[draw,minimum height=20pt]{$f$};
\node (g) at (0.4,1.2)[draw,minimum height=20pt]{$g$};

\node (tr) at (0,0.1) {$\bigtriangledown$};
\coordinate (trd) at (0,0);
\coordinate (truw) at (-0.1,00.2);
\coordinate (true) at (0.1,0.2);

\node (f3) at (-0.4,4.2)[draw,minimum height=20pt]{$f'$};
\node (g3) at (0.4,4.2)[draw,minimum height=20pt]{$g'$};

\node (tr13) at (0,5.2) {$\bigtriangleup$};
\coordinate (trd1w3) at (-0.1,5.12);
\coordinate (trd1e3) at (0.1,5.12);
\coordinate (tru13) at (0,5.35);
\node (c1) at (0,6) {};

\coordinate (west) at (-0.7,2.7);
\coordinate (north) at (0,3);
\coordinate (east) at (0.7,2.7);
\coordinate (south) at (0, 2.4);

\begin{knot}[clip width=4]
\strand [blue, thick] (west) to [out=90,in=180] (north) to [out=0,in=90] (east) to [out=-90,in=0] (south) to [out=180,in=-90] (west);
\strand [thick] (cf) to (trd);
\strand [thick] (true) to [out=70, in=-90] (g);
\strand [thick] (g) to (g3);
\strand [thick] (truw) to [out=110,in=-90] (f);
\strand [thick] (f) to (f3);
\strand [thick] (g3) to [out=90,in=-70] (trd1e3);
\strand [thick] (tru13) to (c1);
\strand [thick] (f3) to [out=90,in=-110] (trd1w3);
\flipcrossings{2,3}
\end{knot}
\end{tikzpicture}
}
=
\hbox{
\begin{tikzpicture}[baseline=(current  bounding  box.center)]

\node (cf) at (0,-0.5){};

\node (f) at (-0.4,1.2)[draw,minimum height=20pt]{$f$};
\node (g) at (0.4,1.2)[draw,minimum height=20pt]{$g$};

\node (tr) at (0,0.1) {$\bigtriangledown$};
\coordinate (trd) at (0,0);
\coordinate (truw) at (-0.1,00.2);
\coordinate (true) at (0.1,0.2);

\node (f3) at (-0.4,2)[draw,minimum height=20pt]{$f'$};
\node (g3) at (0.4,2)[draw,minimum height=20pt]{$g'$};

\node (tr13) at (0,5.2) {$\bigtriangleup$};
\coordinate (trd1w3) at (-0.1,5.12);
\coordinate (trd1e3) at (0.1,5.12);
\coordinate (tru13) at (0,5.35);
\node (c1) at (0,6) {};

\coordinate (west) at (-0.7,3.7);
\coordinate (north) at (0,4);
\coordinate (east) at (0.7,3.7);
\coordinate (south) at (0, 3.4);

\begin{knot}[clip width=4]
\strand [blue, thick] (west) to [out=90,in=180] (north) to [out=0,in=90] (east) to [out=-90,in=0] (south) to [out=180,in=-90] (west);
\strand [thick] (cf) to (trd);
\strand [thick] (true) to [out=70, in=-90] (g);
\strand [thick] (g) to (g3);
\strand [thick] (truw) to [out=110,in=-90] (f);
\strand [thick] (f) to (f3);
\strand [thick] (g3) to [out=90,in=-70] (trd1e3);
\strand [thick] (tru13) to (c1);
\strand [thick] (f3) to [out=90,in=-110] (trd1w3);
\flipcrossings{2,3}
\end{knot}
\end{tikzpicture}
}
=
\hbox{
\begin{tikzpicture}[baseline=(current  bounding  box.center)]

\node (cf) at (0,-0.5){};

\node (f) at (-0.4,1.2)[draw,minimum height=20pt]{$f$};
\node (g) at (0.4,1.2)[draw,minimum height=20pt]{$g$};

\node (tr) at (0,0.1) {$\bigtriangledown$};
\coordinate (trd) at (0,0);
\coordinate (truw) at (-0.1,00.2);
\coordinate (true) at (0.1,0.2);

\node (tr1) at (0,3.2) {$\bigtriangleup$};
\coordinate (trd1w) at (-0.1,3.12);
\coordinate (trd1e) at (0.1,3.12);
\coordinate (tru1) at (0,3.35);

\node (f3) at (-0.4,2)[draw,minimum height=20pt]{$f'$};
\node (g3) at (0.4,2)[draw,minimum height=20pt]{$g'$};

\node (tr3) at (0,4.1) {$\bigtriangledown$};
\coordinate (trd3) at (0,4);
\coordinate (truw3) at (-0.1,4.2);
\coordinate (true3) at (0.1,4.2);

\node (tr13) at (0,5.2) {$\bigtriangleup$};
\coordinate (trd1w3) at (-0.1,5.12);
\coordinate (trd1e3) at (0.1,5.12);
\coordinate (tru13) at (0,5.35);
\node (c1) at (0,6) {};

\begin{knot}[clip width=4]
\strand [thick] (cf) to (trd);
\strand [thick] (true) to [out=70, in=-90] (g);
\strand [thick] (g) to (g3);
\strand [thick] (g3) to [out=90,in=-70] (trd1e);
\strand [thick] (tru1) to (trd3);
\strand [thick] (truw) to [out=110,in=-90] (f);
\strand [thick] (f) to (f3);
\strand [thick] (f3) to [out=90,in=-110] (trd1w);
\strand [thick] (true3) to [out=70, in=-70] (trd1e3);
\strand [thick] (tru13) to (c1);
\strand [thick] (truw3) to [out=110,in=-110] (trd1w3);
\end{knot}
\end{tikzpicture}
}
=
\hbox{
\begin{tikzpicture}[baseline=(current  bounding  box.center)]

\node (cf) at (0,-0.5){};

\node (f) at (-0.4,1.2)[draw,minimum height=20pt]{$f$};
\node (g) at (0.4,1.2)[draw,minimum height=20pt]{$g$};

\node (tr) at (0,0.1) {$\bigtriangledown$};
\coordinate (trd) at (0,0);
\coordinate (truw) at (-0.1,00.2);
\coordinate (true) at (0.1,0.2);

\node (tr1) at (0,3.2) {$\bigtriangleup$};
\coordinate (trd1w) at (-0.1,3.12);
\coordinate (trd1e) at (0.1,3.12);
\coordinate (tru1) at (0,3.35);

\node (f3) at (-0.4,2)[draw,minimum height=20pt]{$f'$};
\node (g3) at (0.4,2)[draw,minimum height=20pt]{$g'$};

\coordinate (trd3) at (0,4);

\begin{knot}[clip width=4]
\strand [thick] (cf) to (trd);
\strand [thick] (true) to [out=70, in=-90] (g);
\strand [thick] (g) to (g3);
\strand [thick] (g3) to [out=90,in=-70] (trd1e);
\strand [thick] (tru1) to (trd3);
\strand [thick] (truw) to [out=110,in=-90] (f);
\strand [thick] (f) to (f3);
\strand [thick] (f3) to [out=90,in=-110] (trd1w);
\end{knot}
\end{tikzpicture}
},
$$
where in the second step we used naturality of the braiding in $\dcentcat{A}$. 
\end{proof}

\subsubsection{The Symmetric Tensor Product as Symmetric Monoidal Structure}

Collecting the results from the previous sections, we have shown that:

\begin{thm}\label{STsymtensmainthrm}
	$(\dcentcat{A},\otimes_s,\mathbb{I}_s)$ is a symmetric monoidal category.
\end{thm}

\begin{proof}
	We know that $\otimes_s$ is a functor, by Lemma \ref{STsymisfunct}. To see that $\otimes_s$ is weakly associative, note that we have shown that the maps induced from the associators of $\cat{A}$ give isomorphisms between the two possibilities for the triple product (Lemma \ref{STasso}). As the associators for $\cat{A}$ satisfy the pentagon equations, so will the induced maps. Furthermore, an argument analogous to the proof of functoriality will establish that these isomorphisms are natural. 
	
	For weak unitality, observe that, in Lemmas \ref{STunitisos} and \ref{STtriangleeq}, we have established $\mathbb{I}_s$ as the unit for $\otimes_s$.
	
	To establish symmetry of $\otimes_s$, we recall that we have shown that the symmetry in $\cat{A}$ induces isomorphisms between the swapped orders of taking the symmetric tensor product (Lemma \ref{STsymlemma}). These induced morphisms give a natural transformation that satisfies the hexagon equations.  
\end{proof}

\subsection{Basic Properties of the Symmetric Tensor Product}

\subsubsection{The forgeful functor is lax monoidal}
The forgetful functor $\Phi:\dcentcat{A}\rar \cat{A}$ is a monoidal functor for $\otimes_c$, but not braided with respect to the braiding for $\otimes_c$ and the symmetry of $\cat{A}$. The goal of this section is to show that $\Phi$ is braided (so in this case symmetric) for the symmetry of $\otimes_s$ and $\cat{A}$, but not strongly monoidal. Instead, $\Phi$ turns out to be lax monoidal, we recall that lax monoidality is defined as follows.

\begin{df}\label{STlaxmondef}
	A \emph{lax monoidal functor} from a monoidal category $\cat{Z}$ to a monoidal category $\cat{Y}$ is a functor $F\colon \cat{Z}\rar \cat{Y}$, together with a natural transformation:
	$$
	\mu\colon F(-)\otimes F(-)\Rightarrow F(-\otimes-),
	$$
	and a morphism
	$$
	\mu^0\colon \mathbb{I}_\cat{Y}\rar F(\mathbb{I}_\cat{Y}),
	$$
	that satisfy the compatibility conditions with the associators $\alpha_\cat{Z}$ and $\alpha_\cat{Y}$:
	\begin{center}
		\begin{tikzcd}
			F(z)(F(z')F(z'')) \arrow[r,"\mu"] \arrow[d,"\alpha_\cat{A}"] 	&	F(c)F(z'z'') \arrow[d,"\mu"] \\
			(F(c)F(z'))F(z'') \arrow[d,"\mu"]								&	F(z(z'z'')) \arrow[dl,"F(\alpha_\cat{Z})"]\\
			F((zz')z'') &,
		\end{tikzcd}
	\end{center}
	for all $z,z',z''\in\cat{Z}$, and compatibility with the unitors:
	\begin{center}
		\begin{tikzcd}
			\mathbb{I}_{\cat{A}}F(z)\arrow[r,"\mu_0"]\arrow[d,"\lambda_\cat{A}"]	&F(\mathbb{I}_\cat{Z})F(z) \arrow[d,"\mu_{\mathbb{I},z}"]\\
			F(z)& F(\mathbb{I}_{\cat{Z}}z)\arrow[l,"F(\lambda_\cat{Z})"],			
		\end{tikzcd}
	\end{center}
	and a similar condition for the right unitors.
	
	Now suppose $\cat{Z}$ and $\cat{Y}$ are braided with braidings (or symmetries) $\beta^\cat{Z}$ and $\beta^\cat{Y}$, respectively. Then $F$ is called \emph{braided} (or \emph{symmetric}) if the following diagram
	\begin{center}
		\begin{tikzcd}
			F(z) F(z') \arrow[r,"\beta^\cat{Y}"] \arrow[d,"\mu"] 	&	 F(z')F(z) \arrow[d,"\mu"]\\
			F(zz') \arrow[r,"F(\beta^\cat{Z})"] 					& F(z'z)
		\end{tikzcd}
	\end{center}
	commutes for all $z,z'\in\cat{Z}$.
\end{df}

\begin{prop}\label{STforgetislaxmon}
	The forgetful functor $\Phi\colon (\dcentcat{A},\otimes_s)\rar(\cat{A},\otimes_\cat{A})$ is symmetric lax monoidal.
\end{prop}
\begin{proof}
	Let $z=(a,\beta)$ and $z'=(a',\beta')$ be objects of $\dcentcat{A}$. Recall (Definition \ref{STsymtensobj}) that the symmetric tensor product $z\otimes_s z'$ has as underlying object $\Phi(z\otimes_s z')$ in $\cat{A}$ the object $\Phi(z\otimes_\Pi z')$. We have to provide a natural transformation $\lambda$ from $\Phi(-)\Phi(-)$ to $\Phi(-\otimes_s -)$, so a map:
	$$
	\mu_{z,z'}\colon 	a a' \rar \Phi(z\otimes_s z')= \Phi(z\otimes_\Pi z').
	$$
	We claim that the image under $\Phi$ of the projection $p_{z,z'}$ associated to $\Pi_{z,z'}$ will work. First of all, the forgetful functor is strongly monoidal for $\otimes_c$, so the image of $\mu_{z,z'}:=\Phi(p_{z,z'})$ is certainly a map between $aa'$ and $\Phi(z\otimes_\Pi z')$. As the associators are defined using the projection $p_{z,z'}$, this map is automatically compatible with the associators, c.f. the first diagram in Definition \ref{STlaxmondef}.
	
	Next, we need to provide a map
	$$
	\mu_0\colon  \mathbb{I}_\cat{A} \rar \Phi(\mathbb{I}_s)= \bigoplus_{i\in\cat{O}(\cat{A})} ii^*.
	$$
	Recall that $\mathbb{I}_\cat{A}$ also acts as the unit for $\otimes_c$, to emphasise this we drop the subscript $\cat{A}$. We take $\mu_0$ to be the global dimension $D$ times the inclusion $I$ of $\mathbb{I}\cong \mathbb{I}\mathbb{I}$ into $\mathbb{I}_s$. Tracing trough the second compatibility diagram in Definition \ref{STlaxmondef}, the composite along the right hand side of the diagram computes as:
	$$
	\hbox{
		\begin{tikzpicture}[baseline=(current  bounding  box.center)]
		
		\node (c) at (0,0){$z$};
		
		\node (i) at (-0.3,0.8)[draw]{$I$};
		
		\node (tr1) at (0,1.5) {$\bigtriangleup$};
		\coordinate (trd1ww) at (-0.12,1.42);
		\coordinate (trd1w) at (-0.08,1.42);
		\coordinate (trd1e) at (0.1,1.42);
		\coordinate (tru1) at (0,1.65);

		\node (tr) at (0,2.1) {$\bigtriangledown$};
		\coordinate (trd) at (0,2);
		\coordinate (truw) at (-0.1,2.2);
		\coordinate (true) at (0.1,2.2);
		\coordinate (trum) at (0,2.2);
		
		\coordinate (ctrl) at (0.175,3.3);
		
		\coordinate (co) at (0,4);
		
		\begin{knot}[clip width=4,clip radius = 3pt]
		\strand [thick] (c) to [out=90,in=-90] (trd1e);
		\strand [thick,blue] (i.110) to [out=90,in=-90] (trd1ww);
		\strand [thick,blue] (i.70) to [out=90,in=-90] (trd1w);
		\strand [thick] (tru1) to (trd);
		\strand [thick,blue] (truw) to [out=90,in=120] (ctrl) to [out=-60,in=90] (trum);
		\strand [thick] (true) to [out=70,in=-90] (co);
		\flipcrossings{2}
		\end{knot}
		\end{tikzpicture}
	}
	=
	\hbox{
		\begin{tikzpicture}[baseline=(current  bounding  box.center)]
		
		\node (c) at (0,0){$z$};
		
		\node (i) at (-0.3,0.8)[draw]{$I$};

		\coordinate (trd1ww) at (-0.12,1.42);
		\coordinate (trd1w) at (-0.08,1.42);
		\coordinate (trd1e) at (0.1,1.42);

		\coordinate (truw) at (-0.1,2.2);
		\coordinate (true) at (0.1,2.2);
		\coordinate (trum) at (-0.06,2.2);
		
		\coordinate (ctrl) at (0.175,3.3);
		
		\coordinate (co) at (0,4);
		
		\begin{knot}[clip width=4,clip radius = 3pt]
		\strand [thick] (c) to [out=80,in=-90] (true) to [out=90,in=-90] (co);
		\strand [thick,blue] (i.110) to [out=90,in=-90] (trd1ww) to [out=90,in=-90] (truw) to [out=90,in=120] (ctrl) to [out=-60,in=90] (trum) to [out=-90,in=90] (trd1w) to [out=-90,in=90] (i.70);
		\flipcrossings{1}
		\end{knot}
		\end{tikzpicture}
	}	=
	\hbox{
		\begin{tikzpicture}[baseline=(current  bounding  box.center)]
		
		\node (c) at (0,0){$z$};
		
		\coordinate (co) at (0,4);
		
		\begin{knot}[clip width=4]
		\strand [thick] (c) to (co);
		\end{knot}
		\end{tikzpicture}
	},
	$$
	where the first step uses Equation \eqref{STinclprojprops}, snapping (the resulting ring on the top comes off immediately), and the second step comes from the observation that the unit braids trivially with all other objects.
	
	Finally, we have to show that the symmetry for $\otimes_s$ is sent to the symmetry in $\cat{A}$, but this follows directly from its definition in Lemma \ref{STsymlemma}.
\end{proof}

\subsubsection{Basic Computations with the Symmetric Tensor Product}

In this section we will do some basic computations with the symmetric tensor product that tell us how the symmetric tensor product behaves with respect to the subcategory $\cat{A}\subset\dcentcat{A}$. We start with:

\begin{lem}
	Let $a$ and $a'$ be objects of $\cat{A}\subset\dcentcat{A}$. Then
	$$
	a \otimes_s a' = a\otimes_c a'.
	$$
\end{lem}
\begin{proof}
	The object $z \otimes_s z'$ is defined in terms of the suboject associated to $\Pi_{z,z'}:a\otimes_c a'\rar a\otimes_c a'$. As the objects of $\cat{A}$ are transparent to each other, we see that $\Pi_{z,z'}=\id_{a\otimes_c a'}$, and the result follows.
\end{proof}

The subcategory $\cat{A}$ is $\otimes_s$-orthogonal to the rest of $\dcentcat{A}$, in the sense that the symmetric tensor product between objects of $\cat{A}\subset\dcentcat{A}$ and objects outside this subcategory is zero. To prove this fact, we need the following well-known lemma, that we prove for convenience of the reader:

\begin{lem}\label{STcentofA}
	The centraliser $\cat{Z}_2(\cat{A},\dcentcat{A})$ (see Definition \ref{STtransparentdef}) of $\cat{A}$ in $\dcentcat{A}$ is $\cat{A}$.  
\end{lem}
\begin{proof}
	Let $z=(a,\beta)\in \cat{Z}_2(\cat{A},\dcentcat{A})$. We want to show that $z\in \cat{A}$, for this it suffices to show that $\beta=s_{-,a}$, ie. that for all $a'\in\cat{A}$ we have $\beta_{a'}=s_{a',a}$. But from the condition (Definition \ref{STtransparentdef}) that that $z$ is transparent to $(a',s_{-,a'})\in \cat{A}\subset\dcentcat{A}$ we have that $\beta_{a'}= s_{a,a'}^{-1}=s_{a',a}$.
\end{proof}

\begin{prop}
	Let $a\in \cat{A}$ and let $z\in \dcentcat{A}$ be a simple object not in $\cat{A}$. Then:
	$$
	a \otimes_s z = 0,
	$$
	where $0$ denotes the zero object of $\dcentcat{A}$.
\end{prop}

\begin{proof}
	Recall that $a\otimes_s z$ is defined using the subobject associated to the idempotent $\Pi_{a,z}$ from Lemma \ref{STidempotent}, and is therefore zero if and only if the idempotent is zero on $a\otimes_c z$. Using that the objects of $\cat{A}$ are transparent with respect to each other we see that the idempotent computes as:
	$$
	\Pi_{a,z}=
	\hbox{
		\begin{tikzpicture}[baseline=(current  bounding  box.center)]
		\coordinate (west) at (0,0);
		\coordinate (north) at (0.5,0.4);
		\coordinate (east) at (1,0);
		\coordinate (south) at (0.5,-0.4);
		\node (a) at (-0.5,-1.2) {$a$};
		\node (b) at (0.5,-1.2) {$z$};
		\coordinate (ao) at (-0.5,1);
		\coordinate (bo) at (0.5,1);
		
		\begin{knot}[clip width=4]
		\strand [blue, thick] (west)
		to [out=90,in=-180] (north)
		to [out=0,in=90] (east)
		to [out=-90,in=0] (south)
		to [out=-180,in=-90] (west);
		\strand [thick] (a) to (ao);
		\strand [thick] (b) to (bo);
		\flipcrossings{1,4}
		\end{knot}
		\end{tikzpicture}
	},
	$$
	which is zero if endomorphism of the simple $z$ defined by the right part of the diagram is zero. This in turn happens if and only if the trace of this endomorphism is zero. Its trace is by definition:
	$$
	\sum_{i\in\cat{O}(\cat{A})} \frac{d_i}{D} S(z,i),
	$$
	where $S(z,i)$ is the S-matrix entry (see \cite{Mueger2002}) for $z$ and $i$. By \cite[Lemma 2.13]{Mueger2002}, this trace computes as:
	$$
	\sum_{i\in\cat{O}(\cat{A})} \frac{d_i}{D} S(z,i)=d_{z'} \chi_{\cat{Z}_2(\cat{A},\dcentcat{A})}(z).
	$$
	Here $\chi_{\cat{Z}_2(\cat{A},\dcentcat{A})}$ denotes the characteristic function on the objects of $\cat{Z}_2(\cat{A},\dcentcat{A})$. The centraliser of $\cat{A}$ in its Drinfeld centre is $\cat{A}$ by Lemma \ref{STcentofA}, so all in all we see that $\Pi_{a,z}=0$ if $z$ is not in $\cat{A}$. Hence $a\otimes_s z$ is zero for all simple $z$ not in $\cat{A}$.
\end{proof}

These two results combine to give:

\begin{prop}\label{STsymtenswitha}
	Let $a\in \cat{A}$ and $z \in \dcentcat{A}$. Denote by $z_{\cat{A}}$ the maximal summand of $z$ that is an object of $\cat{A}$. Then:
	$$
	a \otimes_s z = a \otimes_c z_\cat{A}.
	$$
\end{prop}

Note that this also implies, by associativity and symmetry of the symmetric tensor product, that $z\otimes_s z'$ is never an object of $\cat{A}$ if $z$ and $z'$ have no summands in $\cat{A}$.

We can also prove the following relationship between the symmetric unit and the braided tensor product $\otimes_c$ between objects of $a$ and objects of $\dcentcat{A}$.
\begin{lem}\label{STconvvssymunit}
	Let $a\in\cat{A}$ and $z\in \dcentcat{A}$. Then:
	$$
	a \otimes_c z \cong (a\otimes_c\mathbb{I}_s)\otimes_s z.
	$$
\end{lem}

\begin{proof}
	The object $(a \otimes_c \mathbb{I}_s)\otimes_s z$ is computed in terms of the idempotent $\Pi_{a\otimes_c \mathbb{I}_s,z}$. As $a$ is transparent, we have that:
	$$
	\Pi_{a\otimes_c \mathbb{I}_s,z}= \id_a \otimes_c \Pi_{\mathbb{I}_s,z}=\id_a \otimes_c \id_z,
	$$
	where in the second step we used that $\mathbb{I}_s$ is the unit for $\otimes_s$. The result now follows.
\end{proof}

\section{The Symmetric Tensor Product under Tannaka Duality}\label{STtannacases}
Any symmetric fusion category is, by Tannaka Duality (Theorem \ref{STtannathm}), equivalent to the representation category of a finite (super-)group. Furthermore, as discussed in Section \ref{STdcofrepcat}, the Drinfeld centre of such a representation category can be viewed as the category of equivariant vector bundles over (the underlying group of) this (super-)group (Definition \ref{STGeqvb}). This category admits two obvious tensor products, the convolution tensor product (Definition \ref{STconvtensonvb}) and the fibrewise tensor product (Definition \ref{STfibrewisetensorproductdef}). 

The goal of this section is to first show that for $\cat{A}$ Tannakian (Definition \ref{STfibrefunctorandtannadef}), the symmetric tensor product on $\dcentcat{A}$ translates to the fibrewise tensor product when viewing the Drinfeld centre as equivariant vector bundles. After this, we will examine what the symmetric tensor product becomes when the symmetric fusion category is super-Tannakian. We will see that in this case, the symmetric tensor product translates to a twisted version of the fibrewise tensor product that takes into account the super-group structure. 

\subsection{Tannaka Duality for Symmetric Fusion Categories}\label{STtannaprelimsect}
A famous result by Deligne \cite{Deligne1990,Deligne2002} states that every symmetric fusion category over a field of characteristic zero is the representation category of a (super-)group. Before we can state the theorem, we need some definitions.

\begin{df}\label{STfibrefunctorandtannadef}
	Let $\cat{C}$ be a braided fusion category. A braided functor $\cat{C}\rar \Vect$ (or $\cat{C}\rar \svec$) is called a \emph{(super-)fibre functor}. A braided fusion category $\cat{C}$ is called \emph{(super-)Tannakian} if $\cat{C}$ admits a (super-)fibre functor.
\end{df}

Before we state Deligne's Theorem, we recall some basic facts about super-groups\footnote{This definition of super-groups is different from viewing super-groups as group objects in the category of super-manifolds. The definition here is the one used in the context of fusion categories, see for example \cite{Bruillard2016}.}:
\begin{df}\label{STsupergroup}
	A \emph{super-group} $(G,\omega)$ is a group $G$ together with a choice of central element of order two $\omega$. A \emph{representation of a super-group} is a super-vector space $V$ and a homomorphism $G\rar\tn{Aut}_{\svec}(V)$ that takes the element $\omega$ to the grading involution on $V$. The fusion category of such representations $\tn{Rep}(G,\omega)$ is symmetric, with symmetry inherited from $\svec$. Observe that, as $\omega$ is central, an irreducible representation is homogeneous, and $\Rep(G,\omega)$ splits (as a linear category) as the sum of the subcategories of even representations (where $\omega$ acts as the identity) and odd representations (where $\omega$ acts as minus the identity).
\end{df}

\begin{thm}[\cite{Deligne1990,Deligne2002}]\label{STtannathm}
	Let $\cat{A}$ be a symmetric fusion category over a field of characteristic zero. Then $\cat{A}$ admits either a fibre functor or a super-fibre functor, so is either Tannakian or super-Tannakian (Definition \ref{STfibrefunctorandtannadef}). Furthermore, the category $\cat{A}$ is in the Tannakian (or super-Tannakian) case equivalent as a symmetric fusion category to the category of representations of the (super)-group of monoidal natural automorphisms of the (super-)fibre functor (where the grading involution natural isomorphism is taken as the central order 2 element of the supergroup).
\end{thm}

\subsection{The Drinfeld Centre of the Representation Category of a Finite Group}\label{STdcofrepcat}
As discussed in Section \ref{STtannaprelimsect}, every symmetric fusion category $\cat{A}$ is a representation category of a finite group or super-group. It turns out that the Drinfeld centre of a representation category of a finite group $G$ has the interesting feature that it is equivalent (as braided monoidal category) to the Drinfeld centre of the category of $G$-graded vector spaces, as we discuss in this section. We will first discuss the case of $G$ being an ordinary finite group, then we move on to the super-group case.

\subsubsection{The Drinfeld Centre of a Tannakian Category}\label{STdctannacase}
It is well-known (\cite[Chapter 3.2]{Bakalov2001a}) that when $\cat{A}=\tn{Rep}(G)$, there is an equivalence:
\begin{equation}\label{STcentisbund}
\mathcal{E}\colon \dcentcat{A}\xrightarrow{\cong}\Vect_G[G],
\end{equation}
between the Drinfeld centre and the category of equivariant vector bundles over $G$. The latter category is defined as follows:
\begin{df}\label{STGeqvb}
	A \emph{$G$-equivariant vector bundle $V$ on $G$} is a collection of vector spaces $V_g$ for $g\in G$, together with for each $h\in G$ a family of isomorphisms 
	$$
	\rho_h\colon V_g\xrightarrow{\cong} V_{h^{-1}gh},
	$$
	indexed by $g$, and such that $\rho_{h'}\rho_h=\rho_{h'h}$. The vector space $V_g$ will be called \emph{the fibre over $g$}, and the isomorphisms $\rho$ the \emph{action data}.

	The \emph{category $\Vect_G[G]$ of $G$-equivariant vector bundles on $G$} is the category with objects $G$-equivariant bundles over $G$, and morphisms fibrewise linear maps that commute with the $\rho_h$.
\end{df}

\begin{df}\label{STconvtensonvb}
	The \emph{convolution tensor product} $V\otimes W$ of two equivariant vector bundles $V,W$ over $G$ is the equivariant vector bundle with fibres
	$$
	(V\otimes W)_g=\bigoplus_{g_1g_2=g} V_{g_1}\otimes W_{g_2},
	$$
	and action data $\rho_g=\oplus_{g_1g_2=g}\rho^V_{g_1}\otimes \rho^W_{g_2}$.
\end{df}

Furthermore, there is a braiding:

\begin{df}\label{STequivvbbraiding}
	The \emph{braiding isomorphism}
	$$
	\beta_{V,W}: V\otimes W\rar W\otimes V
	$$
	for $V,W\in \Vect_G[G]$, is given by using for each $g_1g_2=g$
	$$
	V_{g_1}\otimes W_{g_2} \xrightarrow{\tau\circ(\rho_{g_2}\otimes \id)} W_{g_2}\otimes V_{g^{-1}_2g_1g_2},
	$$
	where $\tau$ is the switch map of vector spaces, and summing this to a fibrewise map.
\end{df}

This makes $\Vect_G[G]$ into a braided fusion category. It is in fact a modular tensor category, with simples supported by conjugacy classes of $G$. Note that, as the neutral element $e$ is stabilised under conjugation by the whole group, the subcategory of vector bundles supported by the conjugacy class $[e]$ is the representation category of $G$. The inclusion functor from Equation \eqref{STinclusionbrtodrinfeld} is in this model for the Drinfeld centre the functor
\begin{align}\begin{split}
\label{STexplicitannainclusion}
\mathcal{I}\colon \Rep(G) \rar & \Vect_{G}[G]\\
(V,\rho)\mapsto &(\{V_g=\begin{cases}
V \mbox{ for g=e},\\
0 \mbox{ otherwise.}
\end{cases}\}_{g\in G},\rho)
\end{split}
\end{align}
that views a representation of $G$ as a vector bundle over $G$ supported by $[e]$.

\begin{df}\label{STequivvbforget}
	The \emph{forgetful functor from $\Vect_G[G]$ to $\tn{Rep}(G)$} is given by
	\begin{align*}
	\Phi:	 \Vect_G[G] \rar & \tn{Rep}(G)\\
	V=(\{ V_g\},\rho) \mapsto& (\bigoplus_{g\in G} V_g,\rho),
	\end{align*}
	where the action data $\rho$ acts on the direct sum in the obvious way.
\end{df}

Using the forgetful functor, the inverse to the equivalence $\mathcal{E}$ from Equation \eqref{STcentisbund} between $\cat{Z}(\tn{Rep}(G))$ and $\Vect_G[G]$ is in one direction given by taking $V=\{V_g\}$ and mapping it to $(\Phi(V),\beta_{V,-})$. 

%
%
%
%
%
%

\subsubsection{The Drinfeld Centre of a Super-Tannakian Category}\label{STsupertannadc}
We will now discuss the Drinfeld centre of the representation category of a finite supergroup $(G,\omega)$. We will denote the underlying finite group by $G$. We start with the following observation:

\begin{lem}
	For any finite supergroup $(G,\omega)$, there is an equivalence $$\cat{Z}(\Rep(G,\omega))\cong \cat{Z}(\Rep(G))$$ of braided monoidal categories.
\end{lem}
\begin{proof}
	This follows directly from the fact that $\Rep(G,\omega)$ and $\Rep(G)$ are equivalent as monoidal categories and that the Drinfeld centre construction only uses the monoidal structure.
\end{proof}

So, using the equivalence from Equation \eqref{STcentisbund}:

\begin{cor}\label{STsupercentisbund}
	For any finite supergroup $(G,\omega)$, there is an equivalence $$\cat{Z}(\Rep(G,\omega))\cong\Vect_G[G]$$ of braided monoidal categories.	
\end{cor}

In the super-case, the inclusion functor from Equation \eqref{STinclusionbrtodrinfeld} takes a more complicated form than that from Equation \eqref{STexplicitannainclusion}. 

\begin{prop}\label{STsupertannainclfunctor}
	Under the equivalence of $\cat{Z}(\Rep(G,\omega))$ with $\Vect_G[G]$, the inclusion functor
	$$
	\Rep(G,\omega)\hookrightarrow\Vect_G[G]
	$$ 
	from Equation \eqref{STinclusionbrtodrinfeld} is given by sending even and odd representations to vector bundles supported by $[e]$ and $[\omega]$, respectively.
\end{prop}
\begin{proof}
	We need, for an object $(V,\rho)\in \Rep(G,\omega)$ (a super-vector space together with a representation), to examine what object $((V,\rho),s_{-,V})$ is mapped to under the equivalence from Corollary \ref{STsupercentisbund}. It is enough to do this for a homogeneous object of parity $|V|\in \{0,1\}$, as $\Rep(G,\omega)$ is a direct sum of its subcategories of even and odd objects. Assuming $V$ is homogeneous, and $(W,\rho')$ is an object of $\Rep(G,\omega)$, the half-braiding $s_{W,V}$ is given by $(-1)^{|W||V|}\tau_{W,V}$, where $\tau_{W,V}$ is the switch in vector spaces (recall that a super-vector space is just a vector space with a grading, and that the tensor product is just the vector space tensor product). Observe that we can rewrite the factor $(-1)^{|W||V|}$ as $\rho'(\omega^{|V|})$. Comparing this to Definition \ref{STequivvbbraiding}, we see that this means that $((V,\rho),s_{-,V})$ is the bundle with fibre $(V,\rho)$ supported on $[\omega^{|V|}]$. 
\end{proof}

With the choice of $\omega\in G$, we also get a parity for objects of $\Vect_{G}[G]$. 
\begin{lem}\label{SThomogeneitysupertanna}
	Let $V$ be a simple object of $\Vect_G[G]$, then $\omega$ acts by either $\id_{V}$ or $-\id_{V}$.
\end{lem}

\begin{proof}
	The simple objects in $\Vect_G[G]$ are supported by conjugacy classes, and indecomposable as bundles over these. As $\omega$ is central, it has to act by the same linear map on each fibre. As it is an element of order two, this map has to be a block sum of $\pm \id$, and as $\omega$ is central and $V$ is assumed to be indecomposable there can only be one block.
\end{proof}

With this lemma in hand, we can simply define:
\begin{df}\label{STevenodddefindc}
	Let $V$ be a simple object in $\Vect_G[G]$, then $c$ is called \emph{even} (or \emph{odd}) if $\omega$ acts as $\id$ (or $-\id$).
\end{df}

The forgetful functor on $\cat{Z}(\Rep(G,\omega))\cong \Vect_{G}[G]$ is again the functor to $\Rep(G,\omega)$ that takes the direct sum of the fibres. The above definition ensures that this forgetful functor preserves parity.

\subsection{The Symmetric Tensor Product under Tannakia Duality}

We will now take a look at what the symmetric tensor product $\otimes_s$ becomes from the point of view of Tannaka duality. The main results of this section are Theorems \ref{STsymprodtannafibthm} and \ref{STsymtenssupertannacase}.

\subsubsection{Tannakian Case}\label{STtannakacase}

In this section we will examine what the symmetric tensor product on $\dcentcat{A}$ gives in the case where $\cat{A}=\tn{Rep}(G)$, where $G$ is a finite group. We will show that: 

\begin{thm}\label{STsymprodtannafibthm}
Let $G$ be a finite group. Then the equivalence $\mathcal{E}$ from Equation \eqref{STcentisbund} between $(\cat{Z}(\tn{Rep}(G)),\otimes_s)$ and $(\Vect_G[G],\otimes_f)$ is a symmetric monoidal equivalence. Here $\otimes_f$ denotes the fibrewise tensor product from Definition \ref{STfibrewisetensorproductdef}.
\end{thm}

The proof of this theorem will take up the rest of this section. We start by giving the definition of the fibrewise tensor product.

\begin{df}\label{STfibrewisetensorproductdef}
	The \emph{fibrewise tensor product on $\Vect_G[G]$} is given by
	$$
	(V\otimes_f W)_g= V_g\otimes W_g,
	$$
	with $G$-action $\rho_V\otimes\rho_W$. 
\end{df}
This tensor product is clearly symmetric with symmetry given fibrewise by the switch map of vector spaces.

We will now examine what the idempotent $\Pi_{z,z'}$ looks like in $\Vect_G[G]$. In particular, we will establish the following:
\begin{lem}\label{STcomputepitanna}
Let $V,W\in\Vect_G[G]$ then the idempotent $\Pi_{V,W}:V\otimes_c W \rar V\otimes_c W$ is given by
$$
\Pi_{V,W}|_{V_{g_1}\otimes W_{g_2}}=\begin{cases}\id &\tn{ for }g_1=g_2 \\ 0&\tn{ otherwise.} \end{cases}
$$
\end{lem}

\begin{proof}
By definition, $\Pi_{V,W}$ is given by
$$
\sum_{i\in \tn{IrRep(G)}} \frac{d_i }{D}
\hbox{
\begin{tikzpicture}[baseline=(current  bounding  box.center)]

\coordinate (west) at (-1,0.4);
\coordinate (north) at (0,1.1);
\coordinate (east) at (1,0.4);
\coordinate (south) at (0,-0.3);

\node (a) at (-0.5,-1.2) {$V$};
\node (b) at (0.5,-1.2) {$W$};

\node (i) at (-1.250, 0.4){$i$};
\node (id) at (1.25,0.4){$i^*$};

\coordinate (ao) at (-0.5,2);
\coordinate (bo) at (0.5,2);

\begin{knot}[clip width=4]
\strand [thick] (west)
to [out=90,in=-180] (north)
to [out=0,in=90] (east)
to [out=-90,in=0] (south)
to [out=-180,in=-90] (west);
\strand [thick] (a) to (ao);
\strand [thick] (b) to (bo);
\flipcrossings{1,4}
\end{knot}
\end{tikzpicture}
},
$$
where we put the label $i^*$ to emphasise the object going up is $i^*$. Recall, from Section \ref{STdctannacase}, that we are viewing $i\in \Rep(G)$ as an object in $\Vect_G[G]$ by regarding it as a vector bundle supported by $[e]$. The convolution tensor product (Definition \ref{STconvtensonvb}) between any bundle $E$ and a bundle $F=F_e$ supported by $[e]$ has fibres given by 
$$
(E\otimes F)_{g}=E_g\otimes F_e.
$$

We claim that $\Pi_{V,W}$ acts as a sum of endomorphisms of the summands $V_{g_1}\otimes W_{g_2}$ of the fibres over $g=g_1g_2$ of $V\otimes_c W$. The braidings on the $V$ and $W$ strands with $i$ and $i^*$ will individually fibrewise be automorphisms of $V_{g_1}\otimes i$ and $ i^*\otimes W_{g_2}$. Precomposing with co-evaluation and postcomposing with evaluation for $i$ combines these to automorphisms of $V_{g_1}\otimes W_{g_2}$, for each $i$ in the sum. This means the idempotent will be a direct sum of maps
$$
\Pi_{V_{g_1},W_{g_2}}:V_{g_1}\otimes W_{g_2}\rar V_{g_1}\otimes W_{g_2}, 
$$
for each possible combination of fibres $V_{g_1}$ and $W_{g_2}$. 

We now want to compute what these endomorphisms are. By the definition of the braiding (Definition \ref{STequivvbbraiding}), each of these maps $\Pi_{V_{g_1},W_{g_2}}$ is given by the composite of the evaluation and coevaluation for $i$ with, denoting by $\rho^i(g)$ the action of $G$ on the representation $i$,
$$
 V_{g_1}\otimes i\otimes i^* \otimes W_{g_2} \xrightarrow{\id_{V} \otimes \id_i\otimes \rho^{i^*}(g_2)\otimes\id_W}\otimes V_{g_1}\otimes i \otimes i^* \otimes W_{g_2}
$$
and
$$
 V_{g_1}\otimes i\otimes i^* \otimes W_{g_2} \xrightarrow{\id_{V}\otimes \rho^{i}(g_1)\otimes\id_i \otimes\id_W}V_{g_1}\otimes i \otimes i^* \otimes W_{g_2},
$$
where we have gotten rid of unnecessary switch maps between vector spaces. By unitarity of the representations, $\tn{ev} \circ( \id_i \otimes \rho^{i^*}(g_2))= \tn{ev} \circ (\rho^i(g_2^{-1})\otimes \tn{id}_{i^*})$. The evaluation and coevaluation combine to a trace, so we see that 
$$
\Pi_{V_{g_1},W_{g_2}}=\sum\limits_{i\in\tn{IrRep}(G)}\frac{d_i}{D}\tn{tr}(\rho^{i}(g_2^{-1})\rho^{i}(g_1))=\sum\limits_{i\in\tn{IrRep}(G)}\frac{d_i}{D}\chi_i(g^{-1}_2g_1),
$$  
where $\chi_i$ denotes the character of $i$. We recognise the right hand side as $\frac{1}{D}$ times the character of the group algebra, viewed as a representation of $G$, evaluated on $g_2^{-1}g_1$. As the group acts freely on itself, this character is $D$ times the characteristic function of the conjugacy class of the identity element. This proves the lemma.
\end{proof}

\begin{cor}\label{STtannidempsubobj}
The subobject associated to $\Pi_{V,W}$ has fibres
\begin{equation}\label{STidempotentfibres}
(V\otimes_\Pi W)_{g'} = \bigoplus_{g^2=g'} V_g\otimes W_g.
\end{equation}
\end{cor}

To compare the symmetric tensor product to the fibrewise product, we need to see what effect equipping this object with the half-braiding from Equation \eqref{SThalfbraid} has. We claim that this replaces $g^2$ by $g$. This will establish:

\begin{lem}\label{STtannidembraid}
Let $\mathcal{E}$ be the equivalence from Equation \eqref{STcentisbund}. For any $V,W\in \Vect_G[G]$ we have:
$$
\mathcal{E}(V\otimes_s W)= V\otimes_f W.
$$
\end{lem}
\begin{proof}
Unpacking the definition of the half-braiding, we see that the braiding on $V\otimes_s W$ with respect to $a\in \tn{Rep}(G)$ is given by, on each summand in Equation \eqref{STidempotentfibres},
$$
 a V_gW_g   \xrightarrow{ s_{V_g,a}\otimes \id_{W} \circ (\rho^a(g)\otimes \id_{V\otimes W})}  V_g a W_g \xrightarrow{s_{W_g,a}}  V_g W_g a,
$$
where the first map is the braiding from Equation \eqref{STequivvbbraiding} and the second the symmetry in $\cat{A}$. By monoidality of the symmetry $s$, this composite is the same as:
$$
a V_g W_g  \xrightarrow{s_{a,V_gW_g} \circ ( \rho^a(g)\otimes\id_{V\otimes W})}   V_g W_g a.
$$
Comparing this with Definition \eqref{STequivvbbraiding}, this is saying that $V\otimes_s W$ is the bundle with fibres 
$$
(V\otimes_s W)_g=V_g \otimes W_g,
$$
and this is what we wanted to show.
\end{proof}

Combining Corollary \ref{STtannidempsubobj} and Lemma \ref{STtannidembraid} now proves Theorem \ref{STsymprodtannafibthm}.

\subsubsection{Super-Tannakian Case}
We will now examine the case where $\cat{A}$ is super-Tannakian (Definition \ref{STfibrefunctorandtannadef}) and hence, by Deligne's Theorem \ref{STtannathm}, equivalent to $\Rep(G,\omega)$ for some finite super-group (Definition \ref{STsupergroup}). The structure of the Drinfeld centre in this case is discussed in Section \ref{STsupertannadc}. The Drinfeld centre is still $\Vect_{G}[G]$. However, the inclusion of $\Rep(G,\omega)$ into $\Vect_{G}[G]$ will be different, and the symmetric tensor product gives rise to a new tensor product on $\Vect_{G}[G]$. 

\begin{df}\label{STfibrewisesupertensordef}
	Let $(G,\omega)$ be a finite super-group. The \emph{fibrewise super-tensor product} of homogeneous (see Definition \ref{STevenodddefindc}) $V,W\in \Vect_{G}[G]$ with parities $|V|,|W|\in\{0,1\}$ is the $G$-equivariant vector bundle $V\otimes_{f}^{\omega} W$ with fibres
	$$
	(V\otimes_{f}^{\omega} W)_g=V_{\omega^{|W|}g}W_{\omega^{|V|}g},
	$$
	and $G$-action given by the tensor product of the $G$-actions.
\end{df}

\begin{rmk}
	We can interpret Definition \ref{STfibrewisesupertensordef} as follows: for every choice of central order 2 element of a finite group $G$, there is a symmetric tensor product on $\Vect_{G}[G]$.
\end{rmk}

In this section, we will prove the following:

\begin{thm}\label{STsymtenssupertannacase}
	Let $(G,\omega)$ be a finite super-group. Then the equivalence between $(\cat{Z}(\Rep(G,\omega)),\otimes_s)$ and $(\Vect_{G}[G],\otimes_f^{\omega})$ is symmetric monoidal.
\end{thm}

\begin{proof}
The main difficulty in proving this Theorem is that, as asserted by Proposition \ref{STsupertannainclfunctor}, the inclusion functor from $\Rep(G,\omega)$ to $\cat{Z}(\Rep(G,\omega))$ does not only hit bundles supported by $[e]$. This means we have revisit Lemma \ref{STcomputepitanna} and its proof. We will do this step by step below.

The starting point is again that $\Pi_{V,W}$ is given by
	\begin{equation}\label{STpivwsuper}
	\sum_{i\in \tn{IrRep}(G,\omega)} \frac{d_i }{D}
	\hbox{
		\begin{tikzpicture}[baseline=(current  bounding  box.center)]
		
		\coordinate (west) at (-1,0.4);
		\coordinate (north) at (0,1.1);
		\coordinate (east) at (1,0.4);
		\coordinate (south) at (0,-0.3);
		
		\node (a) at (-0.5,-1.2) {$V$};
		\node (b) at (0.5,-1.2) {$W$};
		
		\node (i) at (-1.250, 0.4){$i$};
		\node (id) at (1.25,0.4){$i^*$};
		
		\coordinate (ao) at (-0.5,2);
		\coordinate (bo) at (0.5,2);
		
		\begin{knot}[clip width=4]
		\strand [thick] (west)
		to [out=90,in=-180] (north)
		to [out=0,in=90] (east)
		to [out=-90,in=0] (south)
		to [out=-180,in=-90] (west);
		\strand [thick] (a) to (ao);
		\strand [thick] (b) to (bo);
		\flipcrossings{1,4}
		\end{knot}
		\end{tikzpicture}
	}.
	\end{equation}
Recall (see Definition \ref{STsupergroup}), that the $i$ are either even or odd, and that (Proposition \ref{STsupertannainclfunctor}) even representations are viewed as bundles supported by $[e]$, while odd representations are viewed as bundles supported by $[\omega]$. 

Each even $i$ summand in Equation \eqref{STpivwsuper} will, just as in the Tannakian case, contribute an automorphism of each $V_{g_1}\otimes W_{g_2}$ given by multiplication by $\chi_i(g_2^{-1}g_1)$, regardless of the parity of $V$ and $W$.

Now suppose that $i$ is odd. Since $\omega$ is the only element its conjugacy class, analogous reasoning to that applied in the Tannakian case tells us that for such odd $i$ we get an endomorphism of $V_{g_1}\otimes W_{g_2}$, let us denote it by
	$$
	\Pi_{V_{g_1},W_{g_2}}^i:V_{g_1}\otimes W_{g_2}\rar V_{g_1}\otimes W_{g_2}.
	$$
	We now want to compute what this map is. It is given by the composite of the appropriate evaluation and coevaluation with
	$$
	V_{g_1}\otimes i\otimes i^* \otimes W_{g_2} \xrightarrow{(-1)^{|V|}\id_{V} \otimes \id_i\otimes \rho^{i^*}(g_2)\otimes\id_W} V_{g_1}\otimes i \otimes i^* \otimes W_{g_2}
	$$
	and
	$$
	V_{g_1}\otimes i\otimes i^* \otimes W_{g_2} \xrightarrow{(-1)^{|W|}\id_{V}\otimes \rho^{i}(g_1)\otimes\id_i \otimes\id_W}V_{g_1}\otimes i \otimes i^* \otimes W_{g_2},
	$$
	where $|V|,|W|\in \{0,1\}$ denote the parity of $V$ and $W$ (by restricting to simple objects we can assume $V$ and $W$ to be homogeneous, see Lemma \ref{SThomogeneitysupertanna}). The signs come from the braiding between $V$ and $i$ and $W$ and $i^*$, respectively. From here, we can apply the same arguments as in the Tannakian case to arrive at:
	$$
	\Pi_{V_{g_1},W_{g_2}}=\sum\limits_{i\in\tn{IrRep}_{0}(G,\omega)}\frac{d_i}{D}\tn{tr}(\rho^{i}(g_2^{-1}g_1))+\sum\limits_{i\in\tn{IrRep}_{1}(G,\omega)}\frac{d_i}{D}(-1)^{|V|+|W|}\tn{tr}(\rho^{i}(g_2^{-1}g_1)),
	$$
	where we have denoted sets of representatives of the even and odd simple objects of $\Rep(G,\omega)$ by $\tn{IrRep}_0(G,\omega)$ and $\tn{IrRep}_1(G,\omega)$, respectively. Now, recall that, by definition, $\omega$ acts as $\id$ on even and as $-\id$ on odd $i$. This means that
	$$
	\chi_i(\omega^{|V|+|W|}g_2^{-1}g_1)=\begin{cases}
	\chi_i(g_2^{-1}g_1)&\tn{ for }i\tn{ even}\\
	(-1)^{|V|+|W|}\chi_i(g_2^{-1}g_1)&\tn{ for }i\tn{ odd}
	\end{cases}.
	$$
	We can use this to rewrite:
	$$
	\Pi_{V_{g_1},W_{g_2}}=\sum\limits_{i\in\tn{IrRep}(G)}\frac{d_i}{D}\chi_i(\omega^{|V|+|W|} g^{-1}_2g_1)=\begin{cases}
	\id&\tn{ for } g^{-1}_2g_1=\omega^{|V|+|W|}\\
	0 & \tn{ otherwise}
	\end{cases},
	$$
	which is the super version of Lemma \ref{STcomputepitanna}. This means that, analogous to Corollary \ref{STidempotentfibres}, we have:
	
	\begin{cor}
		The subobject associated to $\Pi_{V,W}$ is the equivariant vector bundle with fibres:
		$$
		(V\otimes_\Pi W)_{g'}=\bigoplus_{g^2=\omega^{|V|+|W|}g'}V_{\omega^{|V|+|W|}g}W_g=\bigoplus_{g^2=\omega^{|V|+|W|}g'}V_{\omega^{|W|}g}W_{\omega^{|V|}g},
		$$
		for $V$ and $W$ homogeneous.
	\end{cor}
As we can decompose any vector bundle into homogeneous summands, this Corollary completely determines the object underlying the symmetric tensor product of any two vector bundles.

Following the exposition of the Tannakian case, our next task is now to determine what the half-braiding (Equation \eqref{SThalfbraid}) is that we will equip this object with to form the symmetric tensor product. 

We will again compute what this braiding is summandwise. So, let $a\in \tn{Rep}(G,\omega)$ be homogeneous and $V_{\omega^{|V|+|W|}g}W_g$ be a summand in the fibre over $\omega^{|V|+|W|}g^2$. Unpacking the definition of the half-braiding, we get, analogously to the Tannakian case:
\begin{align*}
a V_{\omega^{|V|+|W|}g}W_g   \xrightarrow{(-1)^{|V||a|}\tau_{V,a}\otimes \id_W}  &V_{\omega^{|V|+|W|}g} a W_g \\
& \xrightarrow{(\id_{VW}\otimes \rho^a(g))\circ(\id_V \otimes \tau_{W,a})}  V_{\omega^{|V|+|W|}g} W_g a ,
\end{align*}
where $\tau$ denotes the switch map in vector spaces and, for readability, we have dropped the subscripts on $V$ and $W$ in writing down the map. The sign $ (-1)^{|V||a|}$ comes from the symmetry in $\Rep(G,\omega)$. This composes to:
	$$
	a V_{\omega^{|V|+|W|}g}W_g  \xrightarrow{(-1)^{|V||a|} (\id_{V\otimes W}\otimes  \rho^a(g))\circ\tau_{a,VW} }   V_{\omega^{|V|+|W|}g}W_g a.
	$$
Observe that:
$$
(-1)^{(|V|)|a|}\rho^a(g)=\rho^a(\omega^{|V|}g),
$$
so the half-braiding becomes (using naturality of the switch map):
$$
a V_{\omega^{|V|+|W|}g}W_g  \xrightarrow{\tau_{a,VW} \circ (\id_{V\otimes W}\otimes \rho^a(\omega^{|V|}g))}   V_{\omega^{|V|+|W|}g}W_g a.
$$
Now, comparing this with the definition of the half-braiding in $\Vect_G[G]$ (see Equation \eqref{STequivvbbraiding}), this indicates that $V_{\omega^{|V|+|W|}g}W_g$ is, in $V\otimes_s W$, a summand of the fibre over $\omega^{|V|}g$. We have found:
	$$
	(V\otimes_s W)_{\omega^{|V|}g}=V_{\omega^{|V|+|W|}g}  W_g.
	$$
or, reindexing:
$$
(V\otimes_s W)_{g}= V_{\omega^{|W|}g}W_{\omega^{|V|} g}.
$$

This concludes the proof of Theorem \ref{STsymtenssupertannacase}.
\end{proof}

\subsection*{Acknowledgements}
The author thanks the Engineering and Physical Sciences Research Council for the United Kingdom, the Prins Bernhard Cultuurfonds, the Hendrik Mullerfonds, the Foundation Vrijvrouwe van Renswoude and the Vreedefonds for their financial support. The author was supported by the by the Centre for Symmetry and Deformation at the University of Copenhagen (CPH-SYM-DNRF92) and Nathalie Wahl's European Research Council Consolidator Grant (772960).

Chris Douglas and Andr\'e Henriques have the author's thanks for their support and advice while preparing this document. The author is grateful to Ana Ros Camacho for her comments on a draft of this work, and would like to thank Michael M\"{u}ger and Ulrike Tillmann for helpful comments and suggestions. Thanks go to an anonymous referee for helpful and thorough comments.


\newcommand{\etalchar}[1]{$^{#1}$}

\end{document}